\documentclass{amsart}

\usepackage{amsmath,verbatim,bm}
\usepackage{amsfonts}
\usepackage{graphicx}
\usepackage{amssymb}
\usepackage{epstopdf}
\usepackage{graphicx}
\usepackage{cite}
\usepackage{mathrsfs}

\usepackage[usenames]{color}

\definecolor{ascolor}{rgb}{0.8,0,0}
\definecolor{mblcolor}{rgb}{0,.5,0.5}
\definecolor{daocolor}{rgb}{0.0,0,1.0}
\definecolor{pbbcolor}{rgb}{0.8,0,0}

\usepackage[normalem]{ulem}

\DeclareMathAccent{\maxvec}{\mathord}{letters}{"7E}
\DeclareMathOperator*{\argmin}{arg\,min}

\usepackage{subfigure}

\newtheorem{assumption}{Assumption}

\newtheorem{statement}{Statement}

\newtheorem{thm}{Theorem}[section]
\newtheorem{theorem}[thm]{Theorem}

\newtheorem{lemma}[thm]{Lemma}

\newtheorem{remark}[thm]{Remark}

\numberwithin{equation}{section}

\usepackage{accents}
\newlength{\dtildeheight}

\newlength{\dbarheight}

\newcommand{\weakto}{\rightharpoonup}
\newcommand{\eps}{\varepsilon}
\newcommand{\bbC}{\mathbb{C}}  % elasticity tensor
\newcommand{\bbR}{\mathbb{R}}  % The real numbers.
\newcommand{\bbZ}{\mathbb{Z}}  % The integers.

\newcommand{\calE}{\mathcal{E}} % Energy
\newcommand{\calL}{\mathcal{L}} % Lattice
\newcommand{\calR}{\mathcal{R}} % Interaction range
\newcommand{\calT}{\mathcal{T}} % Triangulations
\newcommand{\calU}{\mathcal{U}} % FEM space
\newcommand{\mF}{\mathsf{F}}
\newcommand{\mG}{\mathsf{G}}
\newcommand{\mS}{\mathsf{S}}

\newcommand{\<}{\langle}
\renewcommand{\>}{\rangle}

\renewcommand{\div}{{\rm div}}
\newcommand{\C}{{\rm C}}
\newcommand{\transpose}{{\!\top}}

\renewcommand{\a}{{\rm a}}
\renewcommand{\c}{{\rm c}}
\renewcommand{\o}{{\rm o}}
\renewcommand{\t}{{\rm t}}
\newcommand{\nn}{{\rm nn}}

\newcommand{\core}{{\rm core}}

\newcommand{\err}{{\rm err}}
\newcommand{\op}{{\rm op}}
\newcommand{\ex}{{\rm ex}}

\newcommand{\nabR}{{\nabR}}

\newcommand{\subsubset}{\subset\!\subset}
\newcommand{\supp}{{\rm supp}}

\newcommand{\vol}{{\rm vol}}

\def\Xint#1{\mathchoice
   {\XXint\displaystyle\textstyle{#1}}%
   {\XXint\textstyle\scriptstyle{#1}}%
   {\XXint\scriptstyle\scriptscriptstyle{#1}}%
   {\XXint\scriptscriptstyle\scriptscriptstyle{#1}}%
   \!\int}
\def\XXint#1#2#3{{\setbox0=\hbox{$#1{#2#3}{\int}$}
     \vcenter{\hbox{$#2#3$}}\kern-.5\wd0}}

\def\dashint{\Xint-}

\begin{document}
%%-----------------------------
%%      the top matter
%%-----------------------------
\title[Optimization-based AtC for point defects]{Analysis of an optimization-based atomistic-to-continuum coupling method for point defects}

% At most 5 thanks
%

\author{Derek Olson}\thanks{DO: University of Minnesota.  DO was supported by the Department of Defense (DoD) through the National Defense Science \& Engineering Graduate Fellowship (NDSEG) Program.(olso4056@umn.edu)}
\author{Alexander V. Shapeev}\thanks{AS: Skolkovo Institute of Science and Technology. AS was supported in part by the
AFOSR Award
FA9550-12-1-0187. (a.shapeev@skoltech.ru)}
\author{Pavel B. Bochev}\thanks{PB: Sandia National Laboratories,
Numerical Analysis and Applications, P.O. Box 5800, MS 1320,
Albuquerque, NM 87185-1320 (pbboche@sandia.gov). Sandia National Laboratories is a multi-program laboratory managed and operated by Sandia Corporation, a wholly owned subsidiary of Lockheed Martin Corporation, for the U.S. Department of Energy�s National Nuclear Security Administration under contract DE-AC04-94AL85000. This material is based upon work supported by the U.S. Department of Energy, Office of Science, Office of Advanced Scientific Computing Research (ASCR). Part of this research was carried under the auspices of the Collaboratory on Mathematics for Mesoscopic Modeling of Materials (CM4).}
\author{Mitchell Luskin}\thanks{ML: University of Minnesota. ML was supported in part by the NSF PIRE Grant OISE-0967140,
NSF Grant 1310835, AFOSR Award
FA9550-12-1-0187, and ARO MURI Award W911NF-14-1-0247.
(luskin@umn.edu)}
\date{\today}
\begin{abstract} We formulate and analyze an optimization-based Atomistic-to-Continuum (AtC) coupling method for problems with point defects. Near the defect core the method employs a potential-based atomistic model, which enables accurate simulation of the defect. Away from the core, where site energies become nearly independent of the lattice position,  the method
switches to a more efficient continuum model. The two models are merged by minimizing the mismatch of their states on an overlap region, subject to the atomistic and continuum force
balance equations acting independently in their domains. We prove that the optimization problem is well-posed and establish error estimates.
\end{abstract}
%
%\begin{resume} ... \end{resume}
%
\maketitle
%%-----------------------------
%%      your text
%%-----------------------------
\section*{Introduction}
%%-----------------------------

Atomistic-to-continuum (AtC) coupling methods combine the accuracy of potential-based atomistic models of solids with the efficiency of coarse-grained continuum elasticity models by using the former only in small regions where the deformation of the material is highly variable such as near a crack tip or dislocation.  The past two decades have seen an abundance of interest in AtC methods both in the engineering community to enable predictive simulations of crystalline materials and in the mathematical community to understand the errors introduced by AtC approximations.  Of prime importance is the use of AtC methods to model material defects such dislocations and interacting point defects, which play roles in determining the elastic and plastic response of a material~\cite{phillips2001}.

A prototypical AtC method is an instance of heterogeneous domain decomposition in which different parts of the domain are treated by different mathematical models. In particular, AtC divides the domain into an \emph{atomistic} and \emph{continuum} parts and uses a discrete system involving non-local interactions between atoms on the former and a continuum model, such as hyperelastic continuum mechanics, on the latter.

Depending on how these two models are coupled, AtC methods can be broadly classified as  as either force or
energy-based~\cite{miller2009}.  Energy-based couplings define
a hybrid energy functional as a combination of atomistic and
continuum energy functionals, and this hybrid energy functional
is then minimized over a class of admissible deformations.
Force-based couplings instead derive atomistic and continuum
forces from the separate energies and then equilibrate them.
We refer to~\cite{miller2009, acta.atc} for a review of many
existing AtC methods.

The primary challenge in developing energy-based methods has
been the existence of ``ghost
forces''~\cite{ortiz1996quasicontinuum,acta.atc} near the
interface between the atomistic and continuum regions.
These ghost forces may lead to uncontrollable errors in
predicted strains, and to date, no method has been implemented
that completely eliminates these errors for general many-body
potentials and general interface geometry in two and three
dimensions.
%proposed to completely eliminate
%these errors for sufficiently general interface geometry in two
%and three dimensions, with the exception of a promising
%recently proposed idea that combines an undetermined
%coefficients method and an $\ell^1$ minimization
%\cite{OrtnerZhang2014}}.
Many force-based methods do not
suffer from the perils of ghost forces; however, for two and
three dimensions, establishing the stability of these methods in
the absence of an energy functional remains a difficult task.

Owing to the practical potential of AtC methods, their error analysis has recently attracted significant attention from mathematicians and engineers.
This analysis is well-developed in one dimension, see e.g.,~\cite{acta.atc}, for a thorough review, and analytic results have been obtained in two and three dimensions for
quasinonlocal (QNL) type
methods~\cite{ortner2013analysis,E2006,shimokawa,ortner2012construction,OrtnerShapeevZhang2013,Ortner2012, shapeev_2011}, blended
methods~\cite{xiao2004,koten_2011,li2012,blended2014,lu2013}, and the force-based method~\cite{lu2013} with various limitations.
%Sharp error estimates for the energy-based method
%of~\cite{shapeev_2011} have only been established in two
%dimensions assuming pair interactions with an additional {\em a
%priori} assumption on the magnitude of the true atomistic
%solution in~\cite{ortner2012construction}.
The analysis of the
QNL method of~\cite{shimokawa} and its subsequent
extensions~\cite{E2006,ortner2012construction,OrtnerShapeevZhang2013,shapeev_2011} has been primarily restricted to two dimensions and often involves a restriction on the interface between atomistic and continuum regions~\cite{E2006,OrtnerShapeevZhang2013} or a restriction on permissible interactions between atoms~\cite{shapeev_2011,ortner2012construction}.    The work~\cite{lu2013} is notable in that it has provided results valid in three dimensions but does so under the auspices of a regularity assumption on the atomistic solution.
Most
recently,~\cite{blended2014} has presented a complete analysis
valid in two and three dimensions of the blended quasicontinuum
energy (BQCE)~\cite{koten_2011,bqce12} and blended
quasicontinuum force (BQCF)~\cite{li2012,bqcf.cmame} methods
valid for general finite-range interactions with no geometrical
restrictions on the interface between atomistic and continuum
regions. A recent modification of a BQCE method was also
proposed and analyzed in \cite{OrtnerZhang2014bgfc}.

The purpose of this paper is to analyze an optimization-based AtC, introduced
in~\cite{olsonPro2013,olsonDev2013}, which couches the coupling
of the two models into a constrained minimization problem.
Specifically, a suitable measure of the mismatch between the
atomistic and continuum states, the ``mismatch energy,'' is
minimized over a common \emph{overlap} region, subject to the
atomistic and continuum force balance equations holding in
atomistic and continuum subdomains.  This differs substantially
from energy-based AtC methods such
as~\cite{ortiz1996quasicontinuum,shimokawa,koten_2011,Bauman_08_CM,Shapeev2011}
which minimize a hybrid combination of atomistic and continuum energies, whose purpose is to approximate the original atomistic energy.
Also, unlike the force-based, non-energy
methods~\cite{dobson_esaim,li2012,lu2013}, we do not directly
equilibrate forces but instead employ the force balance
equations as constraints in a minimization problem.

Our approach in the present work is related to non-standard optimization-based domain decomposition methods for partial differential equations (PDEs); see e.g., ~\cite{lions1998,Lions_00_JAM,gervasio_2001,discacciati2013} and the references therein.
In~\cite{olsonPro2013}, we analyzed an optimization-based
AtC formulation for a linear system with next-nearest neighbor
interactions using the $L^2$ norm of the difference between the
states as a cost functional, and in~\cite{olsonDev2013} we
formulated the approach for many dimensions with nonlinear
interactions and studied it numerically for a 1D chain of atoms
interacting through a Lennard-Jones potential.

A useful setting for studying the errors of various AtC
methods, and the setting we utilize in the present work, is a
single defect embedded in an infinite lattice.  We provide a
comprehensive analysis of the optimization-based AtC method in
$\mathbb{R}^d$ for $d \geq 2$ in the context of a
point\footnote{Aside from additional technicalities needed
to account for differences in a suitable reference
configuration and the decay of the elastic deformation fields
of a dislocation, our analysis can also include dislocations.}
defect located at the origin of an infinite lattice and
establish bounds on the error of the method in terms of two
parameters: the ``diameter'' of the defect core, $R_{\rm
core}$, and the size of the continuum region, $R_\c$.
Our results are comparable to the results for the BQCF method in~\cite{blended2014} in that the error of our method is dominated by the continuum error and truncation error committed respectively by using a continuum model in the continuum domain and by reducing an infinite dimensional problem to a finite dimensional one.  In contrast, the leading order error term established in~\cite{blended2014} for the BQCE method is of lower order and is a coupling error resulting from combining the different models. The coupling error can be minimized but never altogether removed~\cite{blended2014,koten_2011}.  Our analytical results have been numerically confirmed in~\cite{olsonDev2013} in one dimension; however, our analysis in the present is restricted to two and three dimensions.

This paper is organized as follows.	We begin by describing the atomistic defect problem in an infinite domain and formulate the associated AtC method in Section \ref{sec:problem}. In Section ~\ref{sec:error}, we prove that the AtC problem has a solution and subsequently estimate a broken norm error. These results rely on an essential norm equivalence property established in Section~\ref{sec:normEquiv}. The norm equivalence result generalizes a 1D linear result established in~\cite{olsonPro2013} and draws upon ideas from heterogeneous domain decomposition methods developed in~\cite{gervasio_2001}.  For the convenience of the readers, we summarize the key notation used throughout the paper in Appendix~\ref{app:notation}.

%%-----------------------------
\section{Problem Formulation}\label{sec:problem}
%%-----------------------------
We consider a point defect such as a vacancy, interstitial, or impurity located at the origin on the infinite lattice, $\mathbb{Z}^d$.  To formulate the AtC method, we will introduce a finite atomistic domain, $\Omega_\a$, surrounding the defect, and a finite continuum domain, $\Omega_\c$, which overlaps with $\Omega_\a$ in $\Omega_\o$.
Restriction of the atomistic energy to $\Omega_\a$ and application of the Cauchy-Born strain energy on $\Omega_\c$ yield notions of restricted atomistic and continuum energies. Minimizing the $H^1$-(semi-)norm of the mismatch between the atomistic and continuum states in $\Omega_\o$, subject to the corresponding Euler-Lagrange equations of these restricted energies in $\Omega_\a$ and $\Omega_\c$, respectively, completes the formulation of the AtC method.

%%-----------------------------
\subsection{Atomistic Model}\label{sec:atom}
%%-----------------------------
In this paper, we will model atoms located on the integer lattice $\bbZ^d$.  We assume the atoms interact via a classical interatomic potential, and the displacement of atoms from their reference configuration will be denoted by $u:\bbZ^d \to \mathbb{R}^d$. We require that atomistic energy can be written as a sum of site energies, $V_\xi$, associated to each lattice site $\xi \in \bbZ^d$.  This site energy is not unique, and there is great freedom in defining it, see e.g. \cite{tadmor2011}.  From the axiom of material frame indifference, $V_\xi$ is allowed to depend only upon interatomic distances.  Furthermore, we assume a finite cut-off radius in the reference configuration, $r_{\rm cut}$, so that  $V_\xi$ depends only on a subset of atoms in the closed ball of radius $r_{\rm cut}$ about $\xi$: $B_{r_{\rm cut}}(\xi)$.  In other words, the set of atoms that $V_\xi$ may depend upon, for an arbitrary $\xi\in \mathbb{Z}^d$, is given by $\xi+ \calR$ where
\[
\calR \subset \big\{\rho \in \mathbb{Z}^d : 0 < |\rho| \leq r_{\rm cut}\big\}.
\]
Note that we measure distance in the reference configuration rather than the deformed configuration.  An example interaction range is displayed in Figure~\ref{intrange} in two dimensions.
\begin{figure}[htp!]
\centering
\includegraphics[width=0.25\textwidth]{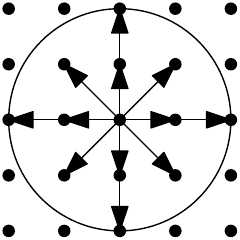}
\caption{A possible interaction range with $r_{\rm cut} = 2$ in $\mathbb{R}^2$. The set $\calR$ consists of all vectors drawn.}
\label{intrange}
\end{figure}

It is convenient to write differences between atoms' displacements using finite difference operators, $D_\rho$ for $\rho \in \calR$, defined by
\[
D_\rho u(\xi) := u(\xi + \rho) - u(\xi).
\]
The collection of finite differences for $\rho \in \calR$ yields a stencil in $(\mathbb{R}^d)^{\calR}$, which we denote by
\[
D u(\xi) :=
\big( D_\rho u(\xi)\big)_{ \rho \in \calR} .
%\big(D_\rho u(\xi)\big)_{\rho \in \calR}.
\]

Thus, formally, the site energy at $\xi$ is a mapping $(\mathbb{R}^d)^{\calR}\mapsto \mathbb{R}$,  which we denote by $V_\xi(Du)$. The atomistic energy is then given by
\begin{equation}\label{atEnergy}
\calE^\a(u) := \sum_{\xi \in \bbZ^d} V_\xi(Du).
\end{equation}

\begin{remark}
By allowing $V$ to depend upon the lattice site, $\xi$, we can include both point defects and dislocations in the analysis. For simplicity, we state our results for the case of point defects.  We refer to~\cite{Ehrlacher2013} for a discussion of how to define $V_\xi$ for various point defects such as vacancies or impurities and the case of dislocations.
\end{remark}

Admissible states of the atomic configuration correspond to  \textit{local minima} of~\eqref{atEnergy}. To define the relevant displacement spaces of lattice functions, we introduce a continuous representation of a discrete displacement via interpolation.
To that end, let $\mathcal{T}_\a$ be a partition of $\mathbb{Z}^d$ into simplices such that
(i) $\xi$ is a node of $\mathcal{T}_\a$ if and only if $\xi \in \bbZ^d$,
(ii) for each $\rho \in \mathbb{Z}^d$ and each $\tau \in \calT_\a$, $\rho + \tau \in \calT_\a$, and
(iii) if $\xi$ and $\eta$ are nodes of the same simplex $\tau\in\calT_\a$ then $\eta-\xi\in\calR$. (The last assumption states that the edges of $\calT_\a$ correspond to neighboring atoms.)
We refer to this as the atomistic triangulation; see Figure~\ref{atomistic_triangulation} for an example in two dimensions.
\begin{figure}[htp!]
\centering
\includegraphics[width=0.25\textwidth]{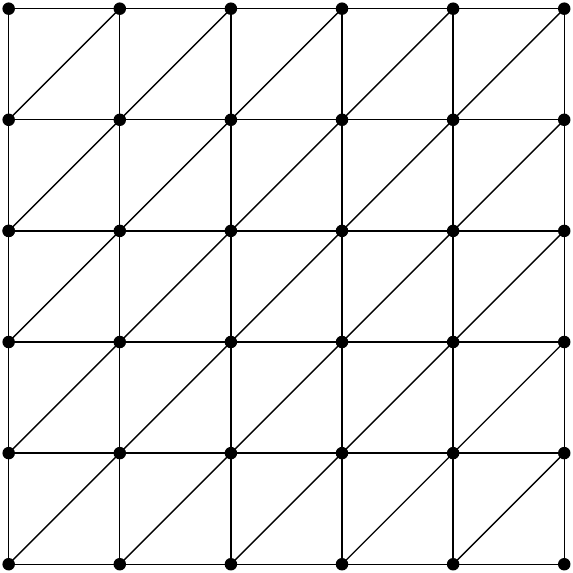}
\caption{An atomistic triangulation of $\mathbb{Z}^2$.}
\label{atomistic_triangulation}
\end{figure}

Let $\mathcal{P}^1(\calT_\a)$ be the standard finite element space of continuous piecewise linear functions with respect to $\mathcal{T}_\a$. The nodal interpolant,  $Iu \in \mathcal{P}^1(\calT_\a)$, of a lattice function $u$ is defined by setting
\[
Iu(\xi) = u(\xi) \quad \forall \xi \in \mathbb{Z}^d.
\]
Using this interpolant, we define the admissible space of displacements as
\begin{equation*}\label{admissDisplace0}
\calU := \left\{u:\mathbb{Z}^d \to \mathbb{R}^d :  \nabla I u \in L^2(\mathbb{R}^d)\right\},
\end{equation*}
and endow it with a semi-norm, $\|\nabla Iu\|_{L^2(\mathbb{R}^d)}$.

The kernel of this semi-norm is the space of constant functions, $\mathbb{R}^d$, and elements of the associated quotient space, $\bm{\calU} := \calU/\mathbb{R}^d$ are equivalence classes
\[
\bm{u} = \left\{v \in \calU : \exists \, c \in \mathbb{R}^d, v - u = c\right\}.
\]
In order to define the interpolation operator on equivalence classes, we define the space
\[
\dot{W}^{1,2}(\mathbb{R}^d) := \left\{f \in W^{1,2}_{\rm loc}(\mathbb{R}^d) : \nabla f \in L^2(\mathbb{R}^d)\right\}
\]
and its quotient space modulo constant functions,
\[
\bm{W}^{1,2}(\mathbb{R}^d) :=  \dot{W}^{1,2}(\mathbb{R}^d)/\mathbb{R}^d.
\]
Since the interpolation operator preserves constants, $I\bm{u} := \left\{Iu : u \in \bm{u}\right\}$ is a well-defined equivalence class. Consequently, the mapping $I:\bm{\calU} \to \bm{W}^{1,2}(\mathbb{R}^d)$ is well-defined, and $\|\nabla I u \|_{L^2(\mathbb{R}^d)}$ induces a norm on $\bm{\calU}$.
Because $\calE^\a(u)$ is invariant under shifts by constants, it is also well-defined on $\bm{\calU}$. As a result, we can state the atomistic problem as
\begin{equation}\label{globalAt}
\bm{u}^{\infty} = \argmin_{\bm{u} \in \bm{\calU}} \calE^\a(\bm{u}),
\end{equation}
where $\argmin$ represents the set of local minimizers and the superscript ``$\infty$'' is used throughout to indicate the \textit{exact atomistic} solution displacement field defined on the infinite lattice $\mathbb{Z}^d$.
Note that minimization over equivalence classes effectively enforces a boundary condition\footnote{This technique is also useful in establishing well-posedness results for linear elliptic systems on all of $\mathbb{R}^d$~\cite{suli2012}.
} $u(\xi) \sim const$ for $\xi \to \infty$.

We formulate and study our AtC method for approximating (\ref{globalAt}) under several hypotheses on the site energy $V_\xi$. First, we assume that the defect core is concentrated at the origin, i.e., outside of this core $V_\xi$ is independent of $\xi$. Succinctly,
\begin{assumption}
There exists $M > 0$ and $V:(\mathbb{R}^d)^{\calR} \to \mathbb{R}$ such that for all $|\xi| > M$, $V_\xi(Du) \equiv V(Du)$.
\end{assumption}
Second, since $\calE^\a({u})$ may be infinite at the reference configuration, $u \equiv 0$, we should instead  consider energy differences from the homogeneous lattice, $\mathbb{Z}^d$. In lieu of this,  without loss of generality, we ask that
\begin{assumption}\label{refassumption}
The site energy vanishes at the reference configuration, i.e., $V(0) = 0$.
\end{assumption}

Finally, we will make the following assumption concerning the regularity of $V_\xi$.
\begin{assumption}\label{siteassumption}
The site potential $V_\xi$ is ${\rm C}^4$ on all of $(\mathbb{R}^d)^{\calR}$.  Furthermore, for $k \in \left\{1,2,3,4\right\}$, there exists $M_{k}$ such that for all multiindices $\alpha, |\alpha| \leq k$
\begin{equation*}\label{siteBound}
|\partial^{\alpha} V_\xi(\bm{\rho})| \leq M_k \quad \forall \xi \in \mathbb{Z}^d, \, \bm{\rho} \in (\mathbb{R}^d)^{\calR},
\end{equation*}
where $\partial^\alpha$ represents the partial derivative.
\end{assumption}

Assumption \ref{siteassumption} allows us to avoid technicalities associated with handling potentials that are singular at the origin, such as the Lennard-Jones potential\footnote{A more realistic assumption would be to assume smoothness in a region of displacements in an energy well, which unduly complicates the analysis.}.
This assumption also implies that $\calE^\a$ is four times Fr\'echet differentiable on the space of displacements
\begin{equation*}\label{testDisplac}
\bm{\calU}_0 := \left\{\bm{u} \in \bm{\calU} :  \supp(\nabla I \bm{u}) \, \mbox{is compact}\right\},
\end{equation*}
from which it is easy to deduce the regularity of the atomistic energy.
\begin{theorem}
The atomistic energy $\calE^\a$ can be extended by continuity to $\bm{\calU}$ and is four times Fr\'echet differentiable on $\bm{\calU}$.
\end{theorem}
We omit the proof, which is a minor modification of the proof of Theorem 2.3 of~\cite{Ehrlacher2013}.

The Euler-Lagrange equation corresponding to the local minimization problem~\eqref{globalAt} is
\begin{equation}\label{ElAt}
\langle \delta \calE^\a(\bm{u}^{\infty}), \bm{v}\rangle  =~ 0 \quad \forall \bm{v} \, \in \bm{\calU}_0.
\end{equation}
We make the following assumption regarding the local minima of (\ref{ElAt}).
\begin{assumption}\label{atCoercive}
There exists a local minimum, $\bm{u}^{\infty} \in \bm{\calU}$, of $\calE^\a(\bm{u})$ and a real number $\gamma_\a > 0$ such that
\begin{equation}\label{condition}
\gamma_\a\| \nabla I \bm{v}\|_{L^2(\mathbb{R}^d)}^2 \leq~ \langle \delta^2 \calE^\a(\bm{u}^{\infty}) \bm{v}, \bm{v}\rangle    \quad \forall \bm{v} \in \bm{\calU}_0.
\end{equation}
\end{assumption}
The condition~\eqref{condition} ensures that the atomistic solution is strongly stable and is critical for the analysis.

For point and line defects, solutions of~\eqref{ElAt} decay algebraically in their elastic far fields~\cite{Ehrlacher2013}.  We quantify the rates of decay using a smooth nodal interpolant of a lattice function, $u:\mathbb{Z}^d \to \mathbb{R}^d$, which we denote by $\tilde{I}u \in W^{3,\infty}_{\rm loc}(\mathbb{R}^d)$.  Its existence follows from~\cite[Lemma 2.1]{blended2014}, which we state below.  We refer to \cite{blended2014} for the proof.  A simplified, one-dimensional result can be found in~\cite{acta.atc}.

\begin{lemma}\label{smoothInterpolant}
There exists a unique operator $\tilde{I}: \calU \to {\rm C}^{2,1}(\mathbb{R}^d)$ such that for all
$\xi \in \mathbb{Z}^d$, (i) $\tilde{I}u$ is multiquintic (i.e., biquintic in the case $d=2$ and triquintic in the case $d=3$) in each cell $\xi + (0,1)^d$, (ii) $\tilde{I}u(\xi) = u(\xi)$, and (iii) for all multiindices $|\alpha| \leq 2$, $\partial^\alpha \tilde{I}u(\xi) = D^{\nn}_\alpha u(\xi)$ where $D^{\nn}_\alpha$ is defined by
\begin{align*}
D^{\nn,0}_i u(\xi) :=~& u(\xi), \\
D^{\nn,1}_i u(\xi) :=~& \frac{1}{2}(u(\xi + e_i) - u(\xi - e_i)) \quad (e_i \mbox{ is the $i$th standard basis vector}), \\
D^{\nn,2}_i u(\xi) :=~& u(\xi + e_i) -2u(\xi) + u(\xi - e_i), \\
D^{\nn}_\alpha u(\xi) :=~& D^{\nn,|\alpha_1|}_{1}\cdots D^{\nn,|\alpha_d|}_{d}u(\xi).
\end{align*}
Furthermore,
\begin{equation}\label{localEst}
\|\nabla^j \tilde{I} u\|_{L^2(\xi + (0,1)^d)} \lesssim~ \|D^j u\|_{\ell^2(\xi + \left\{-1,0,1,2\right\}^d)} \quad \mbox{for} \quad j = 1,2,3,\footnote{In this context, the modified Vinogradov notation $A \lesssim B$ means there is a constant $C$ such that $A \leq C B$ where $C$ may depend on the dimension $d$. After introducing the relevant approximation parameters for the AtC method, we will explicitly state what the constant $C$ is allowed to depend upon.}
\end{equation}
where
\begin{align*}
D^j u(\xi) = \big(D_{\rho_1}D_{\rho_2} \cdots D_{\rho_j}u(\xi)\big)_{(\rho_1, \rho_2, \ldots, \rho_j) \in \calR^j} \quad \mbox{and} \quad \|D^j u\|_{\ell^2(A)}^2 := \sum_{\xi \in A}\sup_{(\rho_1, \rho_2, \ldots, \rho_j) \in \calR^j}|D_{\rho_1}D_{\rho_2} \cdots D_{\rho_j}u(\xi)|^2.
\end{align*}
\end{lemma}

The uniqueness assertion of Lemma~\ref{smoothInterpolant} and the condition that $\partial^\alpha \tilde{I}u(\xi) = D^{\nn}_\alpha u(\xi)$ for all $\xi \in \bbZ^d$ imply that for any constant vector field, $u(\xi) \equiv c \in \mathbb{R}^d$, $\tilde{I}u = c$.  Thus $\tilde{I}$ is well defined as an operator from $\bm{\calU}$ to $\bm{\calU}$ with $\tilde{I}\bm{u} := \left\{\tilde{I}u : u \in \bm{u}\right\}$. From~\eqref{localEst} and it easily follows that
\begin{equation*}\label{smoothEstimate}
\|\nabla \tilde{I}\bm{u}\|_{L^2(\mathbb{R}^d)} \lesssim~ \|\nabla I\bm{u}\|_{L^2(\mathbb{R}^d)}.
\end{equation*}

The following theorem provides a sharp estimate on the algebraic decay of the minimizers for point defects only.
\begin{theorem}[Regularity of a point defect]\label{decayThm}
The local minimum, $\bm{u}^{\infty}$, of~(\ref{globalAt}) satisfies
\begin{align}\label{decayEquation}
\big|\nabla^j \tilde{I}\bm{u}^{\infty}(x)\big| \lesssim~& |x|^{1-j-d} \quad \mbox{for} \, \, j = 1,2,3 \quad x \in \mathbb{R}^d,
\\ \label{decayEquation_no_tilde}
\big|\nabla I\bm{u}^{\infty}(x)\big| \lesssim~& |x|^{-d} \quad \mbox{for} \, \, x \in \mathbb{R}^d.
\end{align}
\end{theorem}

\begin{proof}
Theorem 3.1 and Lemma 3.5 of~\cite{Ehrlacher2013} imply
\[
\left|D^j \bm{u}^{\infty}(\xi)\right| \lesssim~ \left|\xi\right|^{1-j-d} \quad \mbox{for} \, \, j = 1,2,3.
\]
This, along with the local estimate~\eqref{localEst} of $\tilde{I}$ implies~\eqref{decayEquation}.

An analogous local estimate,
\[
\|\nabla I u\|_{L^2(\xi + (0,1)^d)} \lesssim~ \|D u\|_{\ell^2(\xi)},
\]
implies \eqref{decayEquation_no_tilde}.
\end{proof}

Approximation of~\eqref{globalAt} by truncating the support of the admissible functions to a regular polygon or polyhedron $\Omega$ of diameter $N$ is the first step towards an AtC formulation of this problem. The resulting \emph{truncated} displacement space
\begin{equation*}\label{admissDisplaceFinite}
\bm{\calU}_{\Omega} := \left\{ \bm{u} \in \bm{\calU} : \supp(\nabla I \bm{u}) \subset \overline{\Omega} \right\}
\end{equation*}
is finite-dimensional and comprises all functions that are constant outside of $\Omega$.
Restriction of the optimization problem~\eqref{globalAt} to the space $\bm{\calU}_{\Omega}$ yields a finite dimensional atomistic problem
\begin{equation*}\label{truncatedAt}
\bm{u}_{\Omega} = \argmin_{\bm{\calU}_\Omega} \calE^\a(\bm{u}).
\end{equation*}
The corresponding Euler-Lagrange equation, seek $\bm{u}_{\Omega} \in \bm{\calU}_\Omega$ such that
\begin{equation}\label{truncatedAtEL}
\begin{split}
\langle \delta \calE^\a(\bm{u}_{\Omega}), \bm{v}\rangle  =~& 0 \qquad \forall \bm{v} \in \bm{\calU}_\Omega,
\end{split}
\end{equation}
is a finite-dimensional approximation of~\eqref{ElAt}.
The truncated problem \eqref{truncatedAtEL} provides an accurate and computationally feasible approximation for a single point defect~\cite{Ehrlacher2013}. However, its numerical solution quickly becomes intractable as the number of defects increases.

Thus, the next step in the AtC formulation is to replace \eqref{truncatedAtEL} with a local hyperelastic model in parts of the domain that are sufficiently far away from the defect core; at a minimum, we require  $V_\xi \equiv V$ in these regions. In such regions, the hyperelastic model is derived from the Cauchy-Born rule~\cite{born1954}, which defines a strain energy per unit volume according to
\begin{equation*}\label{cbRule}
W(\mG) := V\big((\mG\rho)_{\rho \in \calR}\big) \quad \mbox{for $\mG \in \mathbb{R}^{d \times d}$.}
\end{equation*}

Integration of the strain energy yields a continuum energy
\begin{equation}\label{cbEnergy}
\calE^\c(u) := \int_{\mathbb{R}^d}W(\nabla u(x))\, dx,
\end{equation}
which is defined for a suitable class of functions such as $W^{1,2}(\mathbb{R}^d)$.
We use the Cauchy-Born rule far from the defect core because in the absence of defects it
provides a second-order accurate approximation for smoothly decaying elastic fields~\cite{weinan2007cauchy,blanc2007}.
The advantage of the Cauchy-Born energy~\eqref{cbEnergy} over the atomistic energy~\eqref{atEnergy} is that local minima of the Cauchy-Born energy can efficiently be approximated by the finite element method on a coarser mesh than the atomistic mesh, $\mathcal{T}_\a$.

%%-----------------------------
\subsection{AtC Approximation}\label{sec:approximation}
%%-----------------------------

AtC methods use the more accurate but expensive atomistic model only in a small region surrounding the defect core and alternate to a more computationally efficient continuum model in the bulk of the domain where the lattice and site energy are homogeneous.  The challenge is to couple the models in a stable and accurate manner without creating spurious numerical artifacts.

To describe our AtC approach we consider a configuration comprised of a finite domain $\Omega$, a defect core $\Omega_{\rm core} \subset \Omega$, and atomistic and continuum subdomains $\Omega_\a, \Omega_\c \subset \Omega$.  The analysis of our AtC method requires several technical assumptions on these domains' relative sizes to one another.  For the formulation and understanding of the algorithm, it suffices to choose domains $\Omega_{\rm core} \subset \Omega_\a$ both containing the defect which have diameters of the same magnitude and a finite computational domain $\Omega \supset \Omega_\a$ whose diameter is much larger than that of $\Omega_\a$.   We then set $\Omega_\c := \Omega\backslash \Omega_\a$ and define the overlap region to be $\Omega_\o := \Omega_\a\backslash\Omega_{\rm core}$.

We now describe the specific domain requirements needed for the analysis of the algorithm.  The domains are defined by first selecting a domain $\Omega_{0}$ so that (i) it contains all $\xi$ for which $V_\xi \not\equiv V$; (ii) its boundary, $\partial\Omega_{0}$, is Lipschitz, and (iii) $\partial\Omega_{0}$ is a union of edges from $\calT_\a$.  The domains $\Omega_{\rm core}, \Omega_\a$, and $\Omega$ will be defined as multiples of $\Omega_0$ so $\Omega_0$ provides the essential shape of these domains. We choose integers $R_{\core}\geq 1$ and $\psi_\a \geq 4$ and set $\Omega_{\rm core} = R_{\core}\Omega_0$ and $\Omega_\a = \psi_\a\Omega_{\rm core}$ with the requirement that $(\psi_\a -1)r_{\rm core} \geq 4r_{\rm cut}$, where $r_{\rm core}$ is the radii of the largest circle centered at the origin contained in $\Omega_{\rm core}$.
Next, we select an integer $R_\Omega > R_{\rm core} \cdot \psi_\a$ and set $\Omega = R_\Omega \Omega_0$ whilst requiring that the radii of the largest circle centered at the origin contained in $\Omega$, denoted by $r_\c$, satisfies $r_\c/r_{\rm core} = r_{\rm core}^{\kappa}$ for some integer $\kappa \geq 1$.  The continuum domain is then defined by $\Omega_\c := \Omega \backslash \Omega_{\rm core}$. We also define the ``annular'' overlap region
$\Omega_\o := \Omega_\a\backslash\Omega_{\rm core}$ and an ``extended'' overlap region $\Omega_{\o, \ex} := (2\psi_\a\Omega_{\rm core}) \backslash \Omega_{\rm core}$.

The requirement that $(\psi_\a -1)r_{\rm core} \geq 4r_{\rm cut}$ can now be interpreted as requiring the overlap ``width'' to be twice the size of the interaction range of the site potential. The purpose of $\Omega_{\o, \ex}$ is to have a domain of definition common to both continuum functions defined on $\Omega_\c$ and atomistic functions defined on $\Omega_\a$ which extends just beyond $\Omega_\o$; it will be used explicitly only in the analysis of Section~\ref{sec:normEquiv}.  Finally, the requirement that $r_\c/r_{\rm core} = r_{\rm core}^{\kappa}$ for some integer $\kappa \geq 1$ can be interpreted as forcing the continuum domain to be much larger in size than the atomistic region, which should indeed be the case if we are to reap the benefits of an AtC method.  See Figure~\ref{domainFigure} for an illustration of the domain decomposition in two dimensions.

We also define the domain ``size'' parameters
\begin{equation*}
R_{\a} :=~ R_{\rm core} \cdot \psi_\a \quad \mbox{and} \quad R_{\c} := \frac{1}{2}\rm{Diam}(\Omega_\c),
\end{equation*}
 and let $r_\a$, and $r_\c$ be the radii of the largest circles inscribed in $\Omega_\a$, and  $\Omega$ respectively.\footnote{We define $r_\c$ as the inner radii of $\Omega$ since $\Omega_\c$ has a hole at the defect core and hence does not have an inscribed circle.}

\begin{figure}[htp!]
\centering
\includegraphics[width=0.5\textwidth]{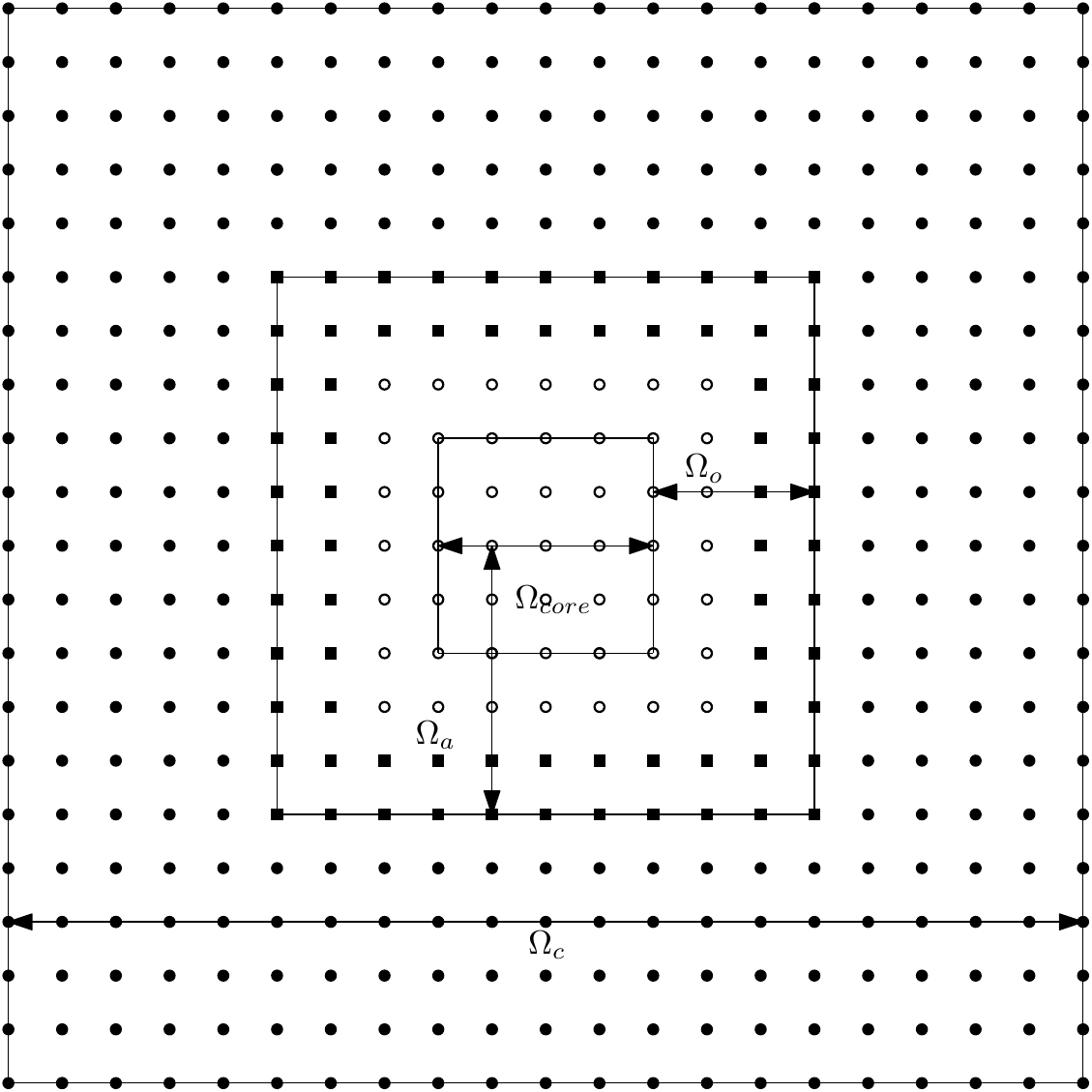}
\caption{An example AtC configuration in two dimensions. The set $\Omega_\a^{\circ\circ}$ is shown as open circles. The solid squares show $\partial_\a\calL_\a$ for the case $\calR = \left\{\pm e_1,  \pm e_2\right\}$.}
\label{domainFigure}
\end{figure}
The atomic lattices associated with the new domains are
\begin{equation*}\label{latticeDomains}
\calL_{\t} := \mathbb{Z}^d \cap \Omega_\t \quad \mbox{where} \quad \t = \a, \c, \o, \rm core,
\end{equation*}
and their atomistic interiors are
\begin{equation*}\label{atInterior}
\calL^\circ_{\t} := \left\{ \xi \in \calL_\t : \xi - \rho \in \calL_\t \quad \forall \rho \in \calR\right\}.
\end{equation*}
The atomistic interiors of the interiors are $\calL^{\circ\circ}_\t=\left(\calL_\t^{\circ}\right)^{\circ}$ while
the atomistic boundary of $\calL_{\t}$ is
\begin{equation*}\label{atBoundary}
\partial_{\a}\calL_{\t} := \calL_{\t} \backslash \calL^{\circ\circ}_\t.
\end{equation*}
See Figure~\ref{domainFigure} for an illustration of $\Omega_\a^{\circ\circ}$ (open circles) and $\partial_\a\calL_\a$ (solid squares) for the case $\calR = \left\{\pm e_1, \pm e_2\right\}$.

\begin{remark}
Throughout the paper we state results involving a parameter $R_{\rm core}^*$ such that if $R_{\rm core} \geq R_{\rm core}^*$,  then a solution to a specific problem defined on the domains constructed above will be guaranteed to exist. Because $R_{\c} \gg R_{\rm core}$ by virtue of $r_\c/r_{\rm core} = r_{\rm core}^{\kappa}$, this will automatically ensure that $R_{\c} \gg R_{\rm core}^*$ as well. These results always assume AtC domain configurations constructed according to the above guidelines.
Furthermore, when stating inequalities, we will use modified Vinogradov notation, $A \lesssim B$ in lieu of $A \leq C \cdot B$, where $C>0$ is a constant.  This constant may only depend upon $\Omega_0, d, R_{\rm core}^*, r_{\rm cut}$, $\psi_\a$, and an additional constant, $\beta$, introduced in Section~\ref{res:cont} as the minimum angle of a finite element mesh.
\end{remark}

%\subsubsection{Approximation parameters}
%\begin{remark}
%\dao{Derek:  I don't think this remark should be introduced so early, but should be placed later.  Consistency errors and stability estimates in this paper will depend upon the %parameters
% $R_{\rm core}, R_c$, and the mesh size function $h(x)$.  Alternatively, if $R_\c = \psi_\a^\kappa R_{\rm core}$, these estimates will depend on $h(x)$, $\kappa$ and $R_{\rm core}$. %One could then attempt to minimize these errors with respect to these parameters along the same lines as in ~\cite{olsonDev,bqce12,bqcf13}.
% However, we will not pursue this optimization and will instead make the minimum requirement that $h(x) \lesssim (|x|/R_{\rm core})^\upsilon$ for some $(d+2)/2 \geq \upsilon \geq 1$.}
%\end{remark}

\subsubsection{Restricted Atomistic Problem}
The basis for defining an atomistic problem restricted to $\Omega_\a$ are the Euler-Lagrange equations~\eqref{truncatedAtEL}.  By requiring $\bm{u}_{\Omega} \in \bm{\calU}_{\Omega}$, we are effectively imposing  Dirichlet boundary conditions (in the sense of equivalence classes) for the variational problem by requiring the function to be constant outside $\Omega$.  Accordingly, we will define a restricted atomistic problem by also specifying Dirichlet boundary conditions on $\partial_\a\calL_\a$.

The admissible displacement space for this problem is $\bm{\calU}^\a := \calU^\a / \mathbb{R}^d$
where
\begin{equation*}\label{admissibleAt}
\calU^\a := \left\{ u^\a: \calL_\a \to \mathbb{R}^d \right\}.
\end{equation*}
The elements of $\bm{\calU^\a}$ are equivalence classes, $\bm{u}^\a$, of lattice functions on $\calL_\a$ differing by a constant $c \in \mathbb{R}^d$. We again use $I$ to denote the piecewise linear interpolant of a lattice function on $\calL_\a$ and endow $\bm{\calU}^\a$ with the norm
 $\|\nabla I \bm{u}^\a\|_{L^2(\Omega_\a)}$.
We then define a restricted atomistic energy functional on $\bm{\calU}^\a$ via
\begin{equation*}\label{restrictedAtEnergy}
\tilde{\calE}^\a(\bm{u}^\a) := \sum_{\xi \in \calL_\a^{\circ\circ}}V_\xi(D\bm{u}^\a(\xi)).
\end{equation*}

We seek to minimize $\tilde{\calE}^\a(\bm{u}^\a)$ over $\bm{\calU}^\a$ subject to Dirichlet boundary conditions on $\partial_\a\calL_\a$.  The set of all possible boundary values is the quotient space
$\bm{\Lambda}^\a := \Lambda^\a / \mathbb{R}^d$, where
\begin{equation*}\label{virtualAt}
\Lambda^\a := \left\{\lambda_\a: \partial_\a\calL_\a \to \mathbb{R}^d\right\}.
\end{equation*}
Elements of $\bm{\Lambda}^\a$ are denoted again by $\lambda_\a$ (without boldface). Thus, the restricted atomistic problem reads
\begin{equation}\label{restrictedAtMin}
%\begin{split}
\bm{u}^\a =
%~&
\argmin_{\bm{\calU}^\a} \tilde{\calE}^\a(\bm{w}^\a) \quad
\text{subject to}\quad \bm{u}^\a =
%~&
\lambda_\a \quad \mbox{on} \quad \partial_\a\calL_\a.
%\end{split}
\end{equation}
We refer to $\lambda_\a$ as a \textit{virtual atomistic control} using the terminology of~\cite{gervasio_2001}.  They are virtual  because $\partial_\a\calL_\a$ is an artificial rather than a physical boundary. They are controls because by varying  $\lambda_\a$ we can vary, i.e. ``control,'' the solutions of \eqref{restrictedAtMin}.

The Euler-Lagrange equation for~\eqref{restrictedAtMin} is seek $\bm{u}^\a\in  \bm{\calU}^\a$ such that
\begin{equation}\label{restrictedAtEL}
\begin{split}
\langle \delta \tilde{\calE}^\a(\bm{u}^\a), \bm{v}^\a\rangle  =~& 0 \quad \forall \bm{v}^\a \in \bm{\calU}^\a_0,  \\
  \bm{u}^\a =~& \lambda_\a \quad \mbox{on} \quad \partial_\a\calL_\a ,
\end{split}
\end{equation}
where the space of atomistic test functions,
$\bm{\calU}^\a_0 := \calU^\a_0 / \mathbb{R}^d$, is the quotient space of
\begin{equation*}\label{restrictedAtTestNoDot}
\calU^\a_0 := \left\{u^\a \in \calU^\a : \exists \, c \in \mathbb{R}^d, \, u^\a|_{\partial_\a\calL_\a} = c  \right\}.
\end{equation*}
After extending $\bm{v}^\a \in \bm{\calU}^\a_0$ by a constant to a function defined
on all of $\mathbb{R}^d$, \cite[(2.5) in Lemma 2.1]{Ehrlacher2013} implies
\begin{equation}\label{equivNorm}
\sum_{\xi \in \calL_\a^{\circ\circ}}\sup_{\rho \in \calR} |D_\rho \bm{v}^\a|^2 \lesssim \|\nabla I\bm{v}^\a\|_{L^2(\Omega_\a)}^2 \quad \forall \bm{v}^\a \in \bm{\calU}^\a_0.
\end{equation}
The following result is then a direct consequence of Assumption~\ref{siteassumption} and~\eqref{equivNorm}.
\begin{theorem}\label{atLippy}
The restricted energy functional $\tilde{\calE}^\a$ is four times Fr\'echet differentiable on $\bm{\calU}^\a$, and each derivative is uniformly bounded in the parameter $R_{\rm core}$.  In particular, $\delta^2\tilde{\calE}^\a$ is Lipschitz continuous on $\bm{\calU}^\a$ with Lipschitz bound independent of $R_{\rm core}$.
\end{theorem}

Given the exact solution $\bm{u}^{\infty}$, we will later require solving~\eqref{restrictedAtEL} where we take $\lambda_\a = \bm{u}^{\infty}|_{\partial_\a\calL_\a}$.  To do that, first set $\bm{u}_\a^{\infty} := \bm{u}^{\infty}|_{\calL_\a}$, and next note that elements of $\bm{\calU}^\a_0$ can be extended by a constant to functions defined on all of $\mathbb{Z}^d$, and this extension will belong to $\bm{\calU}_0$.  By identifying $\bm{v}^\a \in \bm{\calU}^\a_0$ as an element of $\bm{\calU}_0$, we have
\begin{equation*}\label{atLifting}
\big<\delta \tilde{\calE}^\a(\bm{u}^{\infty}_\a), \bm{v}^\a\big> = \left< \delta \calE^\a(\bm{u}^{\infty}), \bm{v}^\a\right> = 0.
\end{equation*}
The final equality holds since $\bm{u}^{\infty}$ solves the Euler Lagrange equations~\eqref{ElAt}.  Similarly, Assumption~\ref{atCoercive} implies
\begin{equation}\label{atLiftingStable}
\gamma_\a \|\nabla I\bm{v}^\a\|_{L^2(\Omega_\a)}^2 = \gamma_\a \|\nabla I\bm{v}^\a\|_{L^2(\mathbb{R}^d)}^2 \leq~ \left< \delta^2 \calE^\a(\bm{u}^{\infty}_\a)\bm{v}^\a, \bm{v}^\a\right> =~ \big< \delta^2 \tilde{\calE}^\a(\bm{u}^{\infty})\bm{v}^\a, \bm{v}^\a\big>
\end{equation}
Hence the solution to~\eqref{restrictedAtEL} for $\lambda_\a = \bm{u}^{\infty}|_{\partial_\a\calL_\a}$ is precisely $\bm{u}_\a^{\infty} := \bm{u}^{\infty}|_{\calL_\a}$.  To avoid unnecessary notation, we will often drop the subscript and just write $\bm{u}^{\infty}$ as the solution to this problem.

%%%%%%%%
%%%%%%%%
\subsubsection{Restricted Continuum}\label{res:cont}
%%%%%%%%
%%%%%%%%
We define the continuum subproblem analogously by using the Euler-Lagrange equations corresponding to minimizing the Cauchy-Born energy~\eqref{cbEnergy}.  In addition to the atomistic mesh, $\mathcal{T}_\a$, that covers $\Omega_\a$ and $\Omega_\c$, we introduce a continuum partition, $\mathcal{T}_h$, of $\Omega_\c$.  We use $\mathcal{T}_h$ to define the admissible continuum finite element displacement space. Let $\mathcal{N}_h$ be the nodes of $\calT_h$. We call a continuum mesh \emph{fully resolved} over a domain $U$ if for each $T \in \calT_h$ with $T \subset U$, we have $T \in \calT_\a$.  In other words, the continuum and atomistic mesh coincide over $U$. Further define
$$
h_T := {\rm{Diam}}(T), \quad\mbox{and}\quad
h(x) := \sup_{\left\{T \in \calT_h: x \in T\right\}}h_T.
$$
For example,  if $x$ is a vertex of a triangle, then $h(x)$ is the largest diameter of the triangles which share this vertex.  Our error estimates require the following assumptions on $\calT_h$.
\begin{assumption}\label{meshSize}
The continuum mesh, $\calT_h$, satisfies
\begin{description}
\item[O.1] The continuum mesh is fully resolved on $\Omega_{\o, \ex}$.
\item[O.2] Nodes in $\mathcal{N}_h$ are also nodes of $\calT_\a$.
\item[O.3] The elements $T\in \calT_h$ satisfy a minimum angle condition for some fixed $\beta > 0$.
\end{description}
\end{assumption}

We will also need the inner and outer continuum boundaries defined as
$$
\Gamma_{\rm core} := \partial \Omega_{\rm core}
\quad\mbox{and}\quad
\Gamma_{\c} := \partial \Omega_{\c} \backslash \Gamma_{\rm core},
$$
respectively.

Our analysis uses two families of interpolants. The first family comprises the standard piecewise linear interpolants
\begin{align*}
I_hu \in \mathcal{P}^1(\calT_h),
& \quad
I_hu(\zeta) = u(\zeta) \quad \forall \zeta \in \mathcal{N}_h,
\\
I u \in \mathcal{P}^1(\calT_\a),
& \quad
I u(\xi) = u(\xi) \quad \forall \xi \, \in \calL,
\end{align*}
defined on the finite element mesh $\mathcal{T}_h$ and the atomistic mesh $\mathcal{T}_\a$, respectively.
The second family comprises Scott-Zhang (quasi-)interpolants~\cite{scott1990, brenner2008}
$S_\a$, $S_{\a,n}$, and $S_{h,n}$. The first, $S_\a$, is defined on
 $\Omega_\c$ with the atomistic mesh, $\calT_\a$; the second, $S_{\a,n}$ is defined on a domain $\tilde{\Omega}_\a$ with a mesh $\tilde{\calT}_{\a,n} = \epsilon_n\calT_\a$ for some $\epsilon_n > 0$; and finally, $S_{h,n}$ is defined on a domain $\tilde{\Omega}_\c$ with mesh $\tilde{\calT}_{h,n} = \epsilon_n\calT_{h}$.
 (We refer to Section~\ref{reduction} for precise definition of these domains.)
 We recall that for a given domain $V$, a mesh partition $\calT$ and a function $f \in H^1(V)$,  the Scott-Zhang interpolant $S f$ has the following four properties~\cite[Chapter 4]{brenner2008}:
\begin{description}
\item[P.1] (Projection) $Sf = f$ for all $f \in \mathcal{P}^1(\calT)$.
\item[P.2] (Preservation of Homogeneous Boundary Conditions) If $f$ is constant on $\partial V$, then so is $Sf$.
\item[P.3] (Stability of semi-norm) $\|\nabla S f\|_{L^2(V)} \lesssim~ \|\nabla f\|_{L^2(V)}$ - the implied constant depending upon the shape regularity constant, or minimum angle of the mesh $\calT$.
\item[P.4] (Interpolation Error for $S$) $\|S f - f\|_{L^2(V)} \lesssim \max_{T \in \calT} {\rm Diam}(T) \|\nabla f\|_{L^2(V)}$.
\end{description}

The space of admissible continuum displacements is $\bm{\calU}^\c_h := \calU^\c_h / \mathbb{R}^d$, where
\begin{equation*}\label{adContNoDot}
\calU^\c_h := \left\{u^\c \in {\rm C}^0(\Omega_\c) : u^\c |_{T} \in \mathcal{P}^1(T) ~~ \forall T \in \mathcal{T}_h, \, \exists \, K \in \mathbb{R}^d, \, u^\c = K \ \, \mbox{on} \, \, \Gamma_{\c}  \right\}.
\end{equation*}
The norm on this space is $\| \nabla \bm{u}^\c\|_{L^2(\Omega_\c)}$. Similar to the definition of $\bm{\calU}_\Omega$, we require the elements of $\bm{\calU}^\c_h$ to be constant on the \emph{outer} continuum boundary $\Gamma_\c$, which enables their extension to infinity by a constant. We do not place such a requirement on the \emph{inner} continuum boundary because $\Gamma_{\rm core}$ is an artificial boundary.
There we will employ \textit{virtual continuum} boundary controls belonging to the space $\bm{\Lambda}^\c := \Lambda^\c / \mathbb{R}^d$ where
\begin{equation*}\label{virtualContNoDot}
\Lambda^\c := \left\{\lambda_\c :\mathcal{N}_h \cap \Gamma_{\rm core}  \to \mathbb{R}^d \right\}.
\end{equation*}
Since $\Gamma_{\rm core}$ represents a curve, we can define the piecewise linear interpolant of $\lambda_\c \in \Lambda_\c$ with respect to $\mathcal{N}_h \cap \Gamma_{\rm core}$ by $I\lambda_\c(\xi) = \lambda_\c(\xi)$ for all $\xi \in \mathcal{N}_h \cap \Gamma_{\rm core}$.  Again, if $\lambda_\c$ is constant, then $I\lambda_\c$ is as well so that this operator is well defined on $\bm{\Lambda}^\c$.  Henceforth, we will always identify elements of $\bm{\Lambda}^\c$ with their piecewise linear interpolant on $\Gamma_{\rm core}$ without explicitly using $I$.

The restricted continuum energy functional on $\bm{\calU}^\c_h$ is then
\begin{equation*}\label{restrictedContEnergy}
\tilde{\calE}^\c(\bm{u}^\c) := \int_{\Omega_\c} W(\nabla \bm{u}^\c(x))\, dx = \sum_{T \in \mathcal{T}_h} W(\nabla \bm{u}^\c(x))\left|T\right|,
\end{equation*}
where $|T|$ represents the volume of the simplex $T$. Given $\lambda_\c \in \bm{\Lambda}^\c$, we consider the following restricted continuum problem
\begin{equation}\label{restrictedContMin}
\bm{u}^\c = \argmin_{\bm{\calU}^\c_h} \tilde{\calE}^\c(\bm{w}^\c) \quad
\mbox{such that}\quad
\bm{u}^\c = \lambda_\c \quad \mbox{on} \quad \Gamma_{\rm core}.
\end{equation}
An appropriate space of test functions for~\eqref{restrictedContMin} is $\bm{\calU}^\c_{h,0} := \calU^\c_{h,0} / \mathbb{R}^d$, where
\begin{equation*}\label{contTestFunctions}
\calU^\c_{h,0} := \left\{u^\c \in \calU^\c_h : \exists \, K \in \mathbb{R}^d, u^\c|_{\Gamma_{\rm core}} = K\right\}.
\end{equation*}
We note that this space requires functions to be constant on both $\Gamma_{\rm core}$ and $\Gamma_\c$, but these constants may differ.

Thus, the Euler-Lagrange equation for~\eqref{restrictedContMin} is given by: seek $\bm{u}^\c \in \bm{\calU}^\c_h$ such that
\begin{equation}\label{restrictedContEL}
\begin{split}
\langle \delta\tilde{\calE}^\c(\bm{u}^\c), \bm{v}^\c\rangle =~& 0 \quad \forall \bm{v}^\c \in \bm{\calU}^\c_{h,0},  \\
  \bm{u}^\c =~& \lambda_\c \quad \mbox{on} \quad \Gamma_{\rm core}.
\end{split}
\end{equation}
The following lemma is an analogue of Lemma~\ref{atLippy}.
\begin{lemma}\label{LipCont}
The restricted continuum energy functional $\tilde{\calE}^\c$ is four times continuously Fr\'echet differentiable on $\bm{\calU}^\c_h$ with derivatives bounded uniformly in the parameter $R_\c$.  Moreover, $\delta^2\tilde{\calE}^\c$ is Lipschitz continuous with Lipschitz bound independent of $R_\c$.
\end{lemma}

\subsubsection{Continuum Error}
This section estimates the error between the restricted continuum solution and exact atomistic solution on $\Omega_\c$. We refer to this error as the \textit{continuum error}.
We will first define an operator which maps functions in $\bm{\calU}$ to functions in $\bm{\calU}^\c_h$.
Application of this operator to $\bm{u}^{\infty}$ yields a representation of the atomistic solution in
 $\bm{\calU}^\c_h$ which can be inserted into
  the variational equation~\eqref{restrictedContEL} to obtain the consistency error.

To this end, let $\eta$ be a smooth bump function equal to $1$ on $B_{3/4}(0)$ and vanishing off of $B_1(0)$.  Given $R>0$ and an annulus $A_R := B_R\backslash B_{3/4R}$, we follow~\cite{blended2014, Ehrlacher2013} to define an operator $T_R: \bm{\calU} \to \bm{\calU}_{\Omega}$  according to
\begin{equation*}\label{trunc_operator}
T_R\bm{u}(x) = \eta(x/R)\big(\textstyle{\tilde{I}\bm{u} - \dashint_{A_R}\tilde{I}\bm{u}\, dx}\big).
\end{equation*}
Above,  $\dashint_U f\, dx = \frac{1}{|U|}\int f\, dx$ is the average value of $f$.
We then set
\begin{equation*}\label{PiProjector}
\Pi_h\bm{u} = I_h\left(\left(T_{r_\c}\bm{u}\right)|_{\Omega_\c}\right).
\end{equation*}

We will use  $\Pi_h \bm{u}^{\infty}$ in~\eqref{restrictedContEL} to obtain the consistency error. The following lemma estimates the error of this operator   over $\Omega_\c$.
We note that the proof below is standard and is similar to, e.g., \cite[Lemma 2.1]{suli2012}.  Moreover, $r_{\rm core} \lesssim R_{\rm core} \lesssim r_{\rm core}$ and $r_{\c} \lesssim R_{\c} \lesssim r_{\c}$ so that estimates in terms of $R_{\rm core}$ and $R_{\c}$ can be phrased in terms of $r_{\rm core}$ and $r_\c$ and vice versa.

\begin{lemma}\label{PiError}
\begin{equation}\label{PiErrorEq}
\|\nabla \Pi_h\bm{u}^{\infty} - \nabla \tilde{I}\bm{u}^{\infty}\|_{L^2(\Omega_\c)} \lesssim~ R_{\rm core}^{-d/2-1} + R_\c^{-d/2}.
%\|h\nabla^2 \tilde{I}\bm{u}^{\infty}\|_{L^2(\Omega_\c)} + \|\nabla\tilde{I}\bm{u}^{\infty}\|_{L^2(\mathbb{R}^d\backslash B_{3r_\c/4}(0))}  + \|\nabla^2 \tilde{I}\bm{u}^{\infty}\|_{L^2(\Omega_\c)}
\end{equation}
\end{lemma}
\begin{proof}
Recalling the definition $\Pi_h = I_h T_{r_\c}$, we first estimate the error by
\begin{equation}\label{pi_trunc_estimate}
%\begin{split}
%&
\|\nabla I_h T_{r_\c} \bm{u}^{\infty} - \nabla \tilde{I} \bm{u}^{\infty}\|_{L^2(\Omega_\c)}
%\\&
~\leq~ \|\nabla I_h T_{r_\c}\bm{u}^{\infty} - \nabla T_{r_\c}\bm{u}^{\infty}\|_{L^2(\Omega_\c)} + \|\nabla T_{r_\c}\bm{u}^{\infty} - \nabla \tilde{I} \bm{u}^{\infty}\|_{L^2(\Omega_\c)}. %+  \|\nabla \tilde{I}\bm{u}^{\infty} - \nabla I \bm{u}^{\infty}\|_{L^2(\Omega_\c)}\\
%\end{split}
\end{equation}

We can easily estimate the second term:
\begin{equation*}\label{pi_trunc_esty}
\begin{split}
&\|\nabla T_{r_\c}\bm{u}^{\infty} - \nabla \tilde{I} \bm{u}^{\infty}\|_{L^2(\Omega_\c)}
\\
&\lesssim~ \big\|\textstyle{\frac{1}{r_\c}\nabla \eta(x/r_\c)\big(\tilde{I}\bm{u}^{\infty} - \dashint_{A_{r_\c}}\tilde{I}\bm{u}^{\infty}\, dx\big) + [\eta(x/r_\c)-1]\nabla\tilde{I}\bm{u}^{\infty}}\big\|_{L^2(\Omega_\c)}
\\
&\lesssim~  \frac{1}{r_\c}\big\|\textstyle{\nabla \eta(x/r_\c)\big(\tilde{I}\bm{u}^{\infty} - \dashint_{A_{r_\c}}\tilde{I}\bm{u}^{\infty}\, dx\big)}\big\|_{L^2(A_{r_\c})} + \|(\eta(x/r_\c)-1)\nabla\tilde{I}\bm{u}^{\infty}\|_{L^2(\mathbb{R}^d\backslash B_{3r_\c/4})}
\\
&\lesssim~  \|\nabla \tilde{I}\bm{u}^{\infty} \|_{L^2(A_{r_\c})}  + \|\nabla\tilde{I}\bm{u}^{\infty}\|_{L^2(\mathbb{R}^d\backslash B_{3r_\c/4})}
~\lesssim~ \|\nabla\tilde{I}\bm{u}^{\infty}\|_{L^2(\mathbb{R}^d\backslash B_{3r_\c/4})}.
\end{split}
\end{equation*}
In the second to last inequality, we have used the fact that $\nabla \eta(x/r_\c)$ vanishes off $A_{r_\c}$ and the Poincar\'e inequality. Employing the decay rates in Theorem~\ref{decayThm}, we obtain
\begin{equation}\label{pi_trunc_25}
\|\nabla T_{r_\c}\bm{u}^{\infty} - \nabla \tilde{I} \bm{u}^{\infty}\|_{L^2(\Omega_\c)} \lesssim~ R_\c^{-d/2}.
\end{equation}

Similarly, the first term of~\eqref{pi_trunc_estimate} can be estimated by first using standard finite element approximation results for smooth functions, the definition of $T_{r_\c}$, the fact that $h/r_\c \leq 1$, and the Poincar\'e inequality:
\begin{equation*}\label{pi_trunc_3}
\begin{array}{l}
\|\nabla I_h T_{r_\c}\bm{u}^{\infty} - \nabla T_{r_\c}\bm{u}^{\infty}\|_{L^2(\Omega_\c)}
\lesssim~ \|h \nabla^2   T_{r_\c}\bm{u}^{\infty}\|_{L^2(\Omega_\c)}
\\[1.5ex]
\qquad
\displaystyle
\lesssim~ \big\|\textstyle{h \nabla^2\big( \eta(x/r_\c)(\tilde{I}\bm{u}^{\infty} - \dashint_{A_{r_\c}}\tilde{I}\bm{u}^{\infty}\, dx)\big)}   \big\|_{L^2(\Omega_\c)}
\\[1.5ex]
\qquad
\displaystyle
=~ \frac{1}{r_\c}\big\|\textstyle{(h/r_\c)\nabla^2\eta(x/r_\c)(\tilde{I}\bm{u}^{\infty} - \dashint_{A_{r_\c}}\tilde{I}\bm{u}^{\infty}\, dx)}\big\|_{L^2(A_{r_\c})} +
\displaystyle \|\nabla \tilde{I}\bm{u}^{\infty} \nabla (\eta(x/r_\c)) \|_{L^2(A_{r_\c})}
\\[1.5ex]
\qquad \qquad
\displaystyle
+ \|h\eta(x/r_\c)\nabla^2 \tilde{I}\bm{u}^{\infty}\|_{L^2(\Omega_\c)}
\\[1.5ex]
\qquad
\displaystyle
\lesssim~ \| \nabla \tilde{I}\bm{u}^{\infty}\|_{L^2(A_{r_\c})} + \frac{1}{r_\c}\| h\nabla \tilde{I}\bm{u}^{\infty} \|_{L^2(A_{r_\c})} + \|h \nabla^2 \tilde{I}\bm{u}^{\infty}\|_{L^2(\Omega_\c)}
\\[1.5ex]
\qquad
\displaystyle
\lesssim~ \|\nabla\tilde{I}\bm{u}^{\infty}\|_{L^2(A_{r_\c})}  + \|h\nabla^2 \tilde{I}\bm{u}^{\infty}\|_{L^2(\Omega_\c)}.
\end{array}
\end{equation*}
A straightforward application of the regularity estimates in Theorem~\ref{decayThm} and the conditions on $h(x)$ in Assumption~\ref{meshSize} give
\begin{equation}\label{pi_trunc_35}
\|\nabla I_h T_{r_\c}\bm{u}^{\infty} - \nabla T_{r_\c}\bm{u}^{\infty}\|_{L^2(\Omega_\c)} \lesssim~R_\c^{-d/2} +  R_{\rm core}^{-d/2-1}.
\end{equation}
Combining~\eqref{pi_trunc_25} and~\eqref{pi_trunc_35} and keeping only the leading order terms yields~\eqref{PiErrorEq}.
\end{proof}

The following Lemma provides information about the stability of the Hessian of $\tilde{\calE}^\c$ evaluated at $\Pi_h \bm{u}^{\infty}$.
\begin{lemma}\label{PiLemma}
There exists $R_{\rm core}^* >0$ and $\gamma_\c > 0$ such that for all $R_{\rm core} \geq R_{\rm core}^*$ (and all continuum partitions $\mathcal{T}_h$ satisfying the requirements of Section~\ref{res:cont}),
\begin{align*}
\gamma_\c \|\nabla \bm{v}^\c\|^2_{L^2(\Omega_\c)} \leq~& \big<\delta^2\tilde{\calE}^\c(\Pi_h \bm{u}^{\infty})\bm{v}^\c,\bm{v}^\c\big> \quad \forall \bm{v}^\c \in \bm{\calU}^\c_{h,0}.
\end{align*}
\end{lemma}
\begin{proof}
For $\bm{u}\in\bm{\calU}$ define
\[
\calE^\a_{\rm hom}(\bm{u}) := \sum_{\xi \in \mathbb{Z}^d} V(D\bm{u}).
\]
From~\cite[Proposition 2.6]{Ehrlacher2013} and Assumption~\ref{atCoercive}, we deduce that
\begin{equation*}\label{rabbit1}
\langle \delta^2\calE^\a_{\rm hom}(\bm{0})\bm{v}, \bm{v}\rangle  \geq \gamma_\a \|\nabla I\bm{v}\|_{L^2(\mathbb{R}^d)}^2  \quad \forall \bm{v} \in \bm{\calU}_0,
\end{equation*}
while \cite[Lemma 5.2]{theil2012}  implies
\[
\langle \delta^2\calE^\c(0)\bm{v}, \bm{v}\rangle  \geq \gamma_\a \|\nabla \bm{v}\|^2_{L^2(\mathbb{R}^d)} \quad \quad \forall v \in H^1_0(\mathbb{R}^d).
\]
Furthermore, extending $\bm{v}^\c\in\mathcal{\bm{U}}^\c_{h,0}$ by a constant to all of $\mathbb{R}^d$ yields

\begin{equation*}\label{cheeseNeg}
\begin{split}
%&
\langle \delta^2\tilde{\calE}^\c(\Pi_h \bm{u}^{\infty})\bm{v}^\c,\bm{v}^\c\rangle
%\\
&~=~ \langle \delta^2\calE^\c(\Pi_h \bm{u}^{\infty})\bm{v}^\c,\bm{v}^\c\rangle   -  \langle \delta^2\calE^\c(\bm{0})\bm{v}^\c, \bm{v}^\c\rangle  + \langle \delta^2\calE^\c(\bm{0})\bm{v}^\c, \bm{v}^\c\rangle  \\
&~\geq~ -\big|\textstyle{\langle \delta^2\tilde{\calE}^\c(\Pi_h \bm{u}^{\infty})\bm{v}^\c,\bm{v}^\c\rangle  - \langle \delta^2\calE^\c(\bm{0})\bm{v}^\c, \bm{v}^\c\rangle }\big|
+ \langle \delta^2\calE^\c(\bm{0})\bm{v}^\c, \bm{v}^\c\rangle  \\
&~\geq~ -\big|\textstyle{\langle \delta^2\tilde{\calE}^\c(\Pi_h \bm{u}^{\infty})\bm{v}^\c,\bm{v}^\c\rangle  - \langle \delta^2\calE^\c(\bm{0})\bm{v}^\c, \bm{v}^\c\rangle }\big| + \gamma_\a\|\nabla \bm{v}^\c\|^2_{L^2(\Omega_\c)}
%&~\gtrsim~ -\|\nabla \Pi_h \bm{u}^{\infty}\|_{L^\infty(\Omega_\c)}\cdot \|\nabla \bm{v}^\c\|^2_{L^2(\Omega_\c)} + \gamma_\a\|\nabla \bm{v}^\c\|^2_{L^2(\Omega_\c)}.
\end{split}
\end{equation*}
if and only if
\begin{equation}\label{cheese}
\begin{split}
\langle \delta^2\tilde{\calE}^\c(\Pi_h \bm{u}^{\infty})\bm{v}^\c,\bm{v}^\c\rangle - \gamma_\a\|\nabla \bm{v}^\c\|^2_{L^2(\Omega_\c)} \geq~& -\big|\textstyle{\langle \delta^2\tilde{\calE}^\c(\Pi_h \bm{u}^{\infty})\bm{v}^\c,\bm{v}^\c\rangle  - \langle \delta^2\calE^\c(\bm{0})\bm{v}^\c, \bm{v}^\c\rangle }\big| \\
\gtrsim~& -\|\nabla \Pi_h \bm{u}^{\infty}\|_{L^\infty(\Omega_\c)}\cdot \|\nabla \bm{v}^\c\|^2_{L^2(\Omega_\c)},
\end{split}
\end{equation}
the final bound being a consequence of the Lipschitz continuity of $W$.

Next,
\begin{align*}
\|\nabla \Pi_h \bm{u}^{\infty}\|_{L^\infty(\Omega_\c)}
&~\leq~ \|\nabla T_{r_\c} \bm{u}^{\infty}\|_{L^\infty(\Omega_\c)} \\
&~=~  \big\|\nabla  \big[\eta(x/r_\c)\big(\textstyle{\tilde{I}\bm{u} - \dashint_{A_{r_\c}}\tilde{I}\bm{u}\, dx}\big) \big]\big\|_{L^\infty(\Omega_\c)} \\
&~=~  \big\|\nabla  (\eta(x/r_\c)) \big(\textstyle{\tilde{I}\bm{u} - \dashint_{A_{r_\c}}\tilde{I}\bm{u}\, dx}\big)  +  \eta(x/r_\c) \nabla \big(\textstyle{\tilde{I}\bm{u} - \dashint_{A_{r_\c}}\tilde{I}\bm{u}\, dx}\big) \big\|_{L^\infty(\Omega_\c)} \\
&~\leq~ \big\|\nabla  (\eta(x/r_\c)) \big(\textstyle{\tilde{I}\bm{u} - \dashint_{A_{r_\c}}\tilde{I}\bm{u}\, dx}\big)\big\|_{L^\infty(A_{r_\c})}  +  \big\|\eta(x/r_\c) \nabla \big(\textstyle{\tilde{I}\bm{u} - \dashint_{A_{r_\c}}\tilde{I}\bm{u}\, dx}\big) \big\|_{L^\infty(\Omega_\c)} \\
&~\lesssim~ \frac{1}{r_\c} \big\|\big(\textstyle{\tilde{I}\bm{u} - \dashint_{A_{r_\c}}\tilde{I}\bm{u}\, dx}\big)\big\|_{L^\infty(A_{r_\c})} + \|\nabla \tilde{I}\bm{u}\|_{L^\infty(\Omega_\c)} \\
&~\lesssim~ \|\nabla \tilde{I}\bm{u}\|_{L^\infty(A_{r_\c})} + \|\nabla \tilde{I}\bm{u}\|_{L^\infty(\Omega_\c)} \\
&~\lesssim~ \|\nabla \tilde{I}\bm{u}\|_{L^\infty(\Omega_\c)}.
\end{align*}
Using this result in~\eqref{cheese} together with~\eqref{decayEquation} yields
\begin{align*}
\langle \delta^2\tilde{\calE}^\c(\Pi_h \bm{u}^{\infty})\bm{v}^\c,\bm{v}^\c\rangle - \gamma_\a\|\nabla \bm{v}^\c\|^2_{L^2(\Omega_\c)}
&~\gtrsim~ -\|\nabla \tilde{I} \bm{u}^{\infty}\|_{L^\infty(\Omega_\c)}\|\nabla \bm{v}^\c\|^2_{L^2(\Omega_\c)} \\
&~\gtrsim~ - (R_{\rm core})^{-d}\|\nabla \bm{v}^\c\|^2_{L^2(\Omega_\c)}.
\end{align*}
Denoting the implied constant in the inequality by $C > 0$, this can be written as
\[
\langle \delta^2\tilde{\calE}^\c(\Pi_h \bm{u}^{\infty})\bm{v}^\c,\bm{v}^\c\rangle \geq \big(-C(R_{\rm core})^{-d} + \gamma_\a\big)\|\nabla \bm{v}^\c\|^2_{L^2(\Omega_\c)}.
\]
Choosing $R_{\rm core}^*$ such that $-C(R_{\rm core}^*)^{-d} + \gamma_\a \geq \gamma_\a/2$ completes the proof with $\gamma_\c := \gamma_\a/2$.

\end{proof}

For the proof of existence of a solution to the restricted continuum problem, we rely on the following quantitative version of the inverse function theorem~\cite{acta.atc,ortnerInverse}.

%Before stating the existence theorem for the continuum problem along with a corresponding estimate, we state a quantitative version of the inverse function theorem which will be used in the proof and also later for proving the existence of a solution to our optimization based AtC method.

\begin{theorem}[Inverse Function Theorem]\label{inverseFunctionTheorem}
Let $X$ and $Y$ be Banach spaces with $f:X \to Y$ a continuously differentiable function on an open set $U$ containing $x_0$. Let $y_0 = f(x_0)$ with $\|y_0\|_Y < \eta$.  Furthermore, suppose that $\delta f(x_0)$ is invertible and such that $\|\delta f(x_0)^{-1}\|_{\calL(Y,X)} < \sigma$, $B_{2\eta\sigma}(x_0) \subset U$, $\delta f$ is Lipschitz continuous on $B_{2\eta\sigma}(x_0)$ with Lipschitz constant $L$, and $2L\eta\sigma^2 < 1$. Then there exists a unique continuously differentiable function $g:B_{\eta}(y_0) \to B_{2\eta\sigma}(x_0)$ such that
\[
g(y_0) = x_0 \quad \mbox{and} \quad f(g(y)) = y \quad \forall y \in B_\eta(y_0) \,.
\]
In particular, there exists $\bar{x} = g(0) \in X$ such that $f(\bar{x}) = 0$ and
\begin{align*}
\|g(y_0) - g(0)\|_X = \|x_0 - \bar{x}\|_{X} <~& 2\eta\sigma. \\
\end{align*}
\end{theorem}

\begin{theorem}[Continuum Error]\label{contModelError}
Let $\lambda_\c^{\infty} := \bm{u}^{\infty}|_{\Gamma_{\rm core}}$. There exists $R_{\rm core}^* > 0$ such that for all $R_{\rm core} \geq R_{\rm core}^*$,  the variational problem
\begin{equation}\label{contModelEq}
%\begin{split}
\langle \delta\tilde{\calE}^\c(\bm{u}), \bm{v}^\c\rangle  =
%~&
0 \quad \forall \bm{v}^\c \in \bm{\calU}^\c_{h,0}  \quad\mbox{subject to}\quad
  \bm{u} =
  %~&
  \lambda_\c^{\infty} \quad \mbox{on} \quad \Gamma_{\rm core},
%\end{split}
\end{equation}
has  a solution $\bm{u}^{\rm con}$ such that
\begin{equation}\label{contModelEstimate}
\begin{split}
\|\nabla \bm{u}^{\rm con} - \nabla I \bm{u}^{\infty}\|_{L^2(\Omega_\c)} \lesssim~ R_{\rm core}^{-d/2-1} + R_{\c}^{-d/2}.
\end{split}
\end{equation}
Furthermore, there exists $\gamma'_\c$ such that
\begin{equation}\label{contStability}
\big< \delta^2 \tilde{\calE}^\c(\bm{u}^{\rm con})\bm{v}^\c, \bm{v}^\c \big> \geq \gamma'_\c\|\nabla \bm{v}^\c\|_{L^2(\Omega_\c)}^2.
\end{equation}
\end{theorem}

\begin{proof}
The proof uses ideas from~\cite{theil2012, blended2014}. We employ Theorem \ref{inverseFunctionTheorem} by linearizing $f = \delta \tilde{\calE}^\c(\cdot)$ about $x_0 = \Pi_h \bm{u}^{\infty}$.
Let $R_{\rm core}^*$ be as in Lemma~\ref{PiLemma}. Then
$\delta^2\tilde{\calE}^\c(\Pi_h \bm{u}^{\infty})^{-1}$ exists and is bounded by $\gamma_\c^{-1}$ for all $R_{\rm core} \geq R_{\rm core}^*$.
Moreover, $\delta^2\tilde{\calE}^\c$ is Lipschitz continuous by Lemma~\ref{LipCont}.  It remains to estimate the dual norm of the residual
\begin{equation}\label{residual}
\sup_{ \bm{v}^\c \in \bm{\calU}^\c_{h,0}, \bm{v}^\c \neq \bm{0}} \frac{\langle \delta \tilde{\calE}^\c(\Pi_h \bm{u}^{\infty}), \bm{v}^\c\rangle }{\|\bm{v}^\c\|_{L^2(\Omega_\c)}}.
\end{equation}
This task requires an atomistic version of the stress.
Following~\cite{theil2012}, let $\zeta(x)$ be the nodal basis function at the origin of the atomistic partition
 $\calT_\a$, i.e., $\zeta(0) = 1$ and $\zeta(\xi) = 0$ for $0 \neq \xi \in \mathbb{Z}^d$.
This allows us to write the interpolant of a lattice function $v$ as
$Iv(x) = \sum_{\xi \in \mathbb{Z}^d} v(\xi)\zeta(x-\xi)$.
Further define the ``quasi-interpolant,'' $v^*$, by
\[
v^*(x) := (Iv * \zeta)(x),
\]
and note that $v^* \in W^{3,\infty}_{\rm loc}$ \cite{theil2012, atInterpolant}. Letting $\chi_{\xi,\rho}(x) := \int_0^1\zeta(\xi + t\rho - x)\, dt$, the atomistic stress, $\mS^\a(\bm{u}, x)$, is then defined by
\begin{equation}\label{atStress}
% \int_{\mathbb{R}^d}\mS^\a(\bm{u}, x) : \nabla \pbb{\bm{v}^*}
% :=
%\langle \delta \calE^\a(\bm{u}), \bm{v}^*\rangle
%=~
%\int_{\mathbb{R}^d}\sum_{\xi \in \mathbb{Z}^d}\sum_{\rho \in \calR} \chi_{\xi, \rho} V_{\xi,\rho}(D\bm{u}) \otimes \rho : \nabla \pbb{\bm{v}^*},
\int_{\mathbb{R}^d} \mS^\a(\bm{u}, x) : \nabla I\bm{v} \, dx :=\left<\delta \calE^\a(\bm{u}), \bm{v}^*\right> =~ \int_{\mathbb{R}^d}\sum_{\xi \in \mathbb{Z}^d}\sum_{\rho \in \calR} \chi_{\xi, \rho} V_{\xi,\rho}(D\bm{u}) \otimes \rho : \nabla I\bm{v} \, dx.
\end{equation}
See \cite{theil2012, blended2014} for further details.

We now estimate the residual~\eqref{residual}.  Fix an element $\bm{v}^\c \in \bm{\calU}^\c_{h,0}$, and assume it has been extended to all of $\mathbb{R}^d$.
Let  $\bm{w}^{\c} = S_{\a}\bm{v}^\c$ where $S_{\a}$ is the Scott-Zhang interpolant onto $\calT_\a$. Note that $I\bm{w}^\c = IS_{\a}\bm{v}^\c = S_{\a}\bm{v}^\c$ for these choices.

We now subtract $0 = \langle \delta \calE^\a(\bm{u}^{\infty}), \bm{w}^{\c,*}\rangle $ from the numerator of \eqref{residual}:
\begin{displaymath} %\label{otter1}
\begin{split}
&
    \langle \delta \tilde{\calE}^\c(\Pi_h \bm{u}^{\infty}), \bm{v}^\c\rangle
\\ &=~
    \langle \delta \tilde{\calE}^\c(\Pi_h \bm{u}^{\infty}), \bm{v}^\c\rangle  - \langle \delta \calE^\a(\bm{u}^{\infty}), \bm{w}^{\c,*}\rangle
\\ &=~
    \langle \delta \tilde{\calE}^\c(\Pi_h \bm{u}^{\infty})-\delta \tilde{\calE}^\c(\tilde{I} \bm{u}^{\infty}), \bm{v}^\c\rangle
    +
    \langle \delta \tilde{\calE}^\c(\tilde{I} \bm{u}^{\infty}), \bm{v}^\c - S_\a \bm{v}^\c \rangle
    +
    (\langle \delta \tilde{\calE}^\c(\tilde{I} \bm{u}^{\infty}), S_\a\bm{v}^\c \rangle
    -
    \langle \delta \calE^\a(\bm{u}^{\infty}), \bm{w}^{\c,*}\rangle
    )
\\&=:~
    E_1 + E_2 + E_3.
\end{split}
\end{displaymath}
In the above, we have used the notation
$
\langle \delta \tilde{\calE}^\c(\Pi_h \bm{u}^{\infty}), w\rangle  := \int_{\Omega_\c}  W'(\nabla \Pi_h \bm{u}^{\infty}) : \nabla w
$
for an arbitrary $w\in H^1(\Omega_\c)$.

$E_{1}$ can be easily estimated:
\begin{displaymath} %\label{otter8}
\begin{split}
\langle \delta \tilde{\calE}^\c(\Pi_h \bm{u}^{\infty}) - \delta \tilde{\calE}^\c(\tilde{I}\bm{u}^{\infty}), \bm{v}^\c \rangle
&~\lesssim~ \|\nabla\Pi_h \bm{u}^{\infty} - \nabla \tilde{I}\bm{u}^{\infty}\|_{L^2(\Omega_\c)}\|\nabla \bm{v}^\c\|_{L^2(\Omega_\c)}
\\&~\lesssim~
(R_{\rm core}^{-d/2-1} + R_\c^{-d/2})\|\nabla \bm{v}^\c\|_{L^2(\Omega_\c)} \quad \mbox{by Lemma~\ref{PiError}.}
\end{split}
\end{displaymath}

We estimate $E_2$  by integrating by parts
\begin{displaymath} %\label{otter7}
\begin{split}
\langle \delta \tilde{\calE}^\c(\tilde{I}\bm{u}^{\infty}) , \bm{v}^\c - S_{\a}\bm{v}^\c\rangle
&\qquad~= \int_{\Omega_\c} W'(\nabla \tilde{I}\bm{u}^{\infty}): \nabla (\bm{v}^\c - S_{\a}\bm{v}^\c) \\
&\qquad~= \int_{\Omega_\c} \div{(W'(\nabla \tilde{I}\bm{u}^{\infty}))} \cdot (\bm{v}^\c - S_{\a}\bm{v}^\c) \\
&\qquad~\leq \|\div{(W'(\nabla \tilde{I}\bm{u}^{\infty}))} \|_{L^2(\Omega_\c)} \cdot \| \bm{v}^\c - S_{\a}\bm{v}^\c\|_{L^2(\Omega_\c)} \\
&\qquad~\lesssim \|    \nabla^2 \tilde{I}\bm{u}^{\infty} \|_{L^2(\Omega_\c)} \| \nabla \bm{v}^\c\|_{L^2(\Omega_\c)},
\\
&\qquad~\lesssim R_\core^{-d/2-1} \| \nabla \bm{v}^\c\|_{L^2(\Omega_\c)},
%\lesssim~& (\|\nabla^2 W'(\tilde{I}\bm{u}^{\infty})\|_{L^2(\Omega_\c)} + \|h\nabla^2 \tilde{I}\bm{u}^{\infty}\|_{L^2(\Omega_\c)})\|\nabla \bm{v}^\c\|_{L^2(\Omega_\c)} \\
%\lesssim~&  (\|\nabla^3\tilde{I}\bm{u}^{\infty}\|_{L^2(\Omega_\c)} + \|h\nabla^2 \tilde{I}\bm{u}^{\infty}\|_{L^2(\Omega_\c)})\|\nabla \bm{v}^\c\|_{L^2(\Omega_\c)}
\end{split}
\end{displaymath}
where we have used the chain rule, bounded the second derivatives of $\tilde{I}\bm{u}^{\infty}$ by $\|    \nabla^2 \tilde{I}\bm{u}^{\infty} \|_{L^2(\Omega_\c)}$, utilized the interpolation estimate {\bf{P.4}} for $S_\a$, and applied the decay rates of Theorem \ref{decayThm}.

We estimate $E_3$ by observing
\begin{align*}
    E_3
=~&
    \int_{\Omega_\c}W'(\nabla \tilde{I} \bm{u}^{\infty}): \nabla S_{\a}\bm{v}^\c  - \int_{\Omega_\c} \mS^\a(\bm{u}^{\infty}, x) :\nabla I\bm{w}^\c
\\=~&
    \int_{\Omega_\c}\big(W'\big(\nabla \tilde{I} \bm{u}^{\infty}\big)
    -
    \mS^\a(\bm{u}^{\infty}, x)\big): \nabla S_{\a}\bm{v}^\c.
\\\leq~&
    \|W'(\nabla \tilde{I} \bm{u}^{\infty})-\mS^\a(\bm{u}^{\infty}, x)\|_{L^2(\Omega_\c)}\|\nabla S_{\a}\bm{v}^\c\|_{L^2(\Omega_\c)}
\\\leq~&
    \|W'(\nabla \tilde{I} \bm{u}^{\infty})-\mS^\a(\bm{u}^{\infty}, x)\|_{L^2(\Omega_\c)}\|\bm{v}^\c\|_{L^2(\Omega_\c)},
\end{align*}
where in the last step we used the stability of the Scott-Zhang interpolant {\bf{P.3}}.
One may then modify the arguments in~\cite[Lemma 4.5, Equations (4.22)--(4.24)]{theil2012} to prove that\footnote{The difference is that our choice of $\tilde{I}u$ is not the same as the smooth interpolant used there.}
%\begin{equation}\label{stressError0}
\[
E_{3} \lesssim~  (\|\nabla^3 \tilde{I}\bm{u}^{\infty}\|_{L^2(\Omega_\c)} + \|\nabla^2 \tilde{I}\bm{u}^{\infty}\|_{L^4(\Omega_\c)}^2) \|\bm{v}^\c\|_{L^2(\Omega_\c)},
\]
%\end{equation}
%
and using the regularity theorem, Theorem~\ref{decayThm}, shows
$ %\label{stressError}
E_3 \lesssim~ R_{\rm core}^{-d/2-2} \|\bm{v}^\c\|_{L^2(\Omega_\c)}.
$

Combining the bounds on $E_1,E_2$, and $E_3$ yields the residual estimate
%\as
%
\begin{equation}\label{residualBeaver}
\begin{split}
\sup_{\bm{v}^\c \in \bm{\calU}^\c_h, \bm{v}^\c \neq \bm{0}} \frac{\langle \delta \tilde{\calE}^\c(\Pi_h \bm{u}^{\infty}), \bm{v}^\c\rangle }{\|\bm{v}^\c\|_{L^2(\Omega_\c)}} \lesssim~& R_{\rm core}^{-d/2-1} + R_{\c}^{-d/2}.
\end{split}
\end{equation}
The inverse function theorem then implies the existence of $\bm{u}^{\rm con}$ satisfying~\eqref{contModelEq} and
\begin{equation}\label{beaverLast}
\|\nabla \bm{u}^{\rm con} - \nabla \Pi_h \bm{u}^{\infty}\|_{L^2(\Omega_\c)} \lesssim~ R_{\rm core}^{-d/2-1} + R_{\c}^{-d/2}.
\end{equation}

To prove~\eqref{contModelEstimate}, observe that
\begin{align*}
\|\nabla \bm{u}^{\rm con} - \nabla I \bm{u}^{\infty}\|_{L^2(\Omega_\c)} \leq~&
 \|\nabla \bm{u}^{\rm con} - \nabla \Pi_h \bm{u}^{\infty}\|_{L^2(\Omega_\c)} + \|\nabla \Pi_h \bm{u}^{\infty} - \nabla \tilde{I} \bm{u}^{\infty} \|_{L^2(\Omega_\c)} + \| \nabla \tilde{I} \bm{u}^{\infty} - \nabla I \bm{u}^{\infty}\|_{L^2(\Omega_\c)}.
\end{align*}
Hence, combining~\eqref{beaverLast} and Lemma~\ref{PiError} yields
\begin{equation}\label{penguin1}
\|\nabla \bm{u}^{\rm con} - \nabla I \bm{u}^{\infty}\|_{L^2(\Omega_\c)} \lesssim~  R_{\rm core}^{-d/2-1} + R_{\c}^{-d/2} + \| \nabla \tilde{I} \bm{u}^{\infty} - \nabla I \bm{u}^{\infty}\|_{L^2(\Omega_\c)}.
\end{equation}
Since $\tilde{I}\bm{u}^{\infty}$ is in $H^2(\Omega_\c)$ and $I \bm{u}^{\infty} = I  (\tilde{I} \bm{u}^{\infty})$, standard finite element approximation theory and the decay estimates in Theorem~\ref{decayThm} give
\begin{equation}\label{penguin2}
\| \nabla \tilde{I} \bm{u}^{\infty} - \nabla I \bm{u}^{\infty}\|_{L^2(\Omega_\c)} = \| \nabla \tilde{I} \bm{u}^{\infty} - \nabla I (\tilde{I} \bm{u}^{\infty} )\|_{L^2(\Omega_\c)}\lesssim~ \|\nabla^2 \tilde{I} \bm{u}^{\infty}\|_{L^2(\Omega_\c)} \lesssim~ R_{\rm core}^{-d/2-1}.
\end{equation}
The last inequalities~\eqref{penguin1} and~\eqref{penguin2} imply the desired estimate~\eqref{contModelEstimate}.

To prove the inequality~\eqref{contStability}, note that
\begin{align*}
\big<\delta^2\tilde{\calE}(\bm{u}^{\rm con})\bm{v}^\c, \bm{v}^\c\big> =~& \big<\big(\delta^2\tilde{\calE}(\bm{u}^{\rm con}) - \delta^2\tilde{\calE}(\Pi_h\bm{u}^{\infty})\big)\bm{v}^\c, \bm{v}^\c\big> + \big<\delta^2\tilde{\calE}(\Pi_h\bm{u}^{\infty})\bm{v}^\c, \bm{v}^\c\big> \\
\gtrsim~&  -\|\nabla \bm{u}^{\rm con} -\nabla  \Pi_h\bm{u}^{\infty}\|_{L^2(\Omega_\c)}\|\nabla v^\c\|_{L^2(\Omega_\c)}^2 + \gamma_\c\|\nabla v^\c\|_{L^2(\Omega_\c)}^2 \\
\gtrsim~& (\gamma_\c - R_{\rm core}^{-d/2-1} + R_{\c}^{-d/2})\|\nabla v^\c\|_{L^2(\Omega_\c)}^2.
\end{align*}
Choosing an appropriate $R_{\rm core}^*$ and $\gamma_\c'$ completes the proof.
\end{proof}

\subsection{The AtC Coupled Problem}
%\commentdao{I removed the tilde from this section.  Is it still clear?}

We couple the restricted atomistic and continuum subproblems by minimizing their mismatch on the overlap region. In this paper, we measure the mismatch by the $H^1$ (semi-)norm of the difference between the continuum solution and the finite element interpolant of the atomistic solution. Thus, our
 AtC formulation seeks an optimal solution
$(u^\a,u^\c) \in \calU^\a \times \calU^\c_h$,
$(\lambda_\a,\lambda_\c)\in \Lambda^\a\times\Lambda^\c$ of the following constrained optimization problem:
\begin{equation}\label{atcOpt}
\begin{array}{c}
\displaystyle
\min_{\{u^\a,u^\c,\lambda^\a,\lambda^\c\} } \|\nabla Iu^\a - \nabla u^\c\|_{L^2(\Omega_\o)}
\quad \text{subject to}
\\[2ex]
\left\{
\begin{array}{rl}
\langle \delta\tilde{\calE}^\a(u^\a), v^\a\rangle=0 &\forall v^\a \in \calU^\a_0 \\[1ex]
u^\a = \lambda_\a & \mbox{on} ~\partial_\a\calL_\a
\end{array}
\right.;
\left\{
\begin{array}{c}
\langle\delta\tilde{\calE}^\c(u^\c), v^\c\rangle=0 \quad \forall v^\c \in \calU^\c_{h,0} \\[1ex]
u^\c = 0 \ \  \mbox{on} ~ \Gamma_{\c} \quad \mbox{and}\quad
u^\c = \lambda_\c \ \  \mbox{on} ~ \Gamma_{\rm core}
\end{array}
\right.;
\int_{\Omega_\o} \left(Iu^\a - u^\c\right) \, dx = 0.
\end{array}
\end{equation}

Alternatively, we may pose the AtC problem on quotient spaces:
\begin{equation}\label{atcOptEquiv}
\begin{array}{c}
\displaystyle
\min_{\{\bm{u}^\a,\bm{u}^\c,\lambda^\a,\lambda^\c\} }
\|\nabla I\bm{u}^\a - \nabla \bm{u}^\c\|_{L^2(\Omega_\o)}
\quad \text{subject to}
\\[2ex]
\left\{
\begin{array}{rl}
\langle \delta\tilde{\calE}^\a(\bm{u}^\a), \bm{v}^\a\rangle=0 &\forall \bm{v}^\a\in\bm{\calU}^\a_0 \\[1ex]
\bm{u}^\a = \lambda_\a & \mbox{on} ~\partial_\a\bm{\calL}_\a
\end{array}
\right.,\quad
\left\{
\begin{array}{rl}
\langle\delta\tilde{\calE}^\c(\bm{u}^\c), \bm{v}^\c\rangle=0 & \forall \bm{v}^\c\in\bm{\calU}^\c_{h,0} \\[1ex]
\bm{u}^\c = \lambda_\c & \mbox{on} \quad \Gamma_{\rm core}
\end{array}
\right.
\end{array}.
\end{equation}
It is easy to see that (\ref{atcOpt}) and (\ref{atcOptEquiv}) are equivalent in the sense that every minimizer,  $(u^\a, u^\c)$, of the former generates an equivalence class,  $(\bm{u}^\a, \bm{u}^\c)$, that is a minimizer of the latter and vice versa. Indeed, if $(u^\a, u^\c)$ solves (\ref{atcOpt}) then for all $(v^\a, v^\c) \in \calU^\a \times \calU^\c_h$,
$$
\|\nabla I\bm{u}^\a - \nabla \bm{u}^\c\|_{L^2(\Omega_\o)}
=
\|\nabla I u^\a - \nabla u^\c\|_{L^2(\Omega_\o)}
\le
\|\nabla I v^\a - \nabla v^\c\|_{L^2(\Omega_\o)}
=\|\nabla I\bm{v}^\a - \nabla \bm{v}^\c\|_{L^2(\Omega_\o)}.
$$
Thus, $(\bm{u}^\a, \bm{u}^\c)$ is a minimizer of (\ref{atcOptEquiv}). The reverse statement follows by an analogous argument.

Notwithstanding the equivalence of the two problems, \eqref{atcOptEquiv} is more convenient for the analysis and so we will study the existence of AtC solutions $(\bm{u}^{\rm atc}_\a, \bm{u}^{\rm atc}_\c)$ in quotient spaces. The formulation~\eqref{atcOpt} was previously used in a numerical implementation~\cite{olsonDev2013}.  Our main result is as follows.
\begin{theorem}[Existence and Error Estimate]\label{mainTheorem}
Let $\bm{u}^{\infty}_{\a} := \bm{u}^{\infty}|_{\calL_\a} $ and
$\bm{u}^{\infty}_{\c} := \bm{u}^{\infty}|_{\calL_\c}$.
There exists $R_{\rm core}^*$ such that for all $R_{\rm core} \geq R_{\rm core}^*$,
the minimization problem~\eqref{atcOptEquiv} has a solution $(\bm{u}^{\rm atc}_\a, \bm{u}^{\rm atc}_\c)$ and
\begin{equation}\label{atcErrorEstimate}
\begin{split}
&\|\nabla\left(I\bm{u}^{\rm atc}_\a-I\bm{u}^{\infty}_\a\right)\|_{L^2(\Omega_\a)}^2
+
\|\nabla\left(\bm{u}^{\rm atc}_\c-I\bm{u}^{\infty}_\c\right)\|_{L^2(\Omega_\c)}^2 \lesssim~ R_{\rm core}^{-d/2-1} + R_{\c}^{-d/2}.
%&~\lesssim~ \|h \nabla^2 \tilde{I}\bm{u}^{\infty}\|_{L^2(\Omega_\c)} + \|\nabla\tilde{I}\bm{u}^{\infty}\|_{L^2(\mathbb{R}^d\backslash B_{3r_\c/4}(0))} +  \|\nabla^3 \tilde{I}\bm{u}^{\infty}\|_{L^2(\Omega_\c)} + \|\nabla^2 \tilde{I}\bm{u}^{\infty}\|_{L^4(\Omega_\c)}.
\end{split}
\end{equation}
\end{theorem}
We prove this result in the remainder of the paper.

%%%%
%%%%
\section{Error Analysis}\label{sec:error}
%%%%
%%%%
To carry out the error analysis of the AtC problem we switch to an equivalent reduced space formulation of \eqref{atcOptEquiv} and apply the inverse function theorem.

%%%%%%
%%%%%%
\subsection{Reduced space formulation of the AtC problem}
%%%%%%
%%%%%%

The restricted atomistic \eqref{restrictedAtMin} and continuum \eqref{restrictedContMin} problems  have solutions for any given  $\lambda_\a \in \bm{\Lambda^\a}$ and $\lambda_\c \in \bm{\Lambda^\c}$. These solutions define mappings $\bm{U}^\a:\bm{\Lambda}^\a \to \bm{\calU}^\a$, and  $\bm{U}^\c:\bm{\Lambda}^\c \to \bm{\calU}^\c_h$, respectively, which will be employed in Theorems~\ref{atRegularity} and~\ref{contRegularity}.
Using these mappings, we can eliminate the states from
\eqref{atcOptEquiv} and obtain an equivalent unconstrained minimization problem in terms of the virtual controls only:
\begin{equation}\label{eq:inf}
\left(\lambda_\a^{\rm atc}, \lambda_\c^{\rm atc}\right) = \argmin_{(\lambda_\a, \lambda_\c)\in\bm{\Lambda}^\a\times\bm{\Lambda}^\c}
J(\lambda_\a, \lambda_\c),
\end{equation}
where $J$ is defined as
\begin{equation*}\label{eq:Opt}
J\left(\lambda_\a, \lambda_\c\right) =~ \frac{1}{2}\|\nabla I\bm{U}^\a(\lambda_\a) - \nabla \bm{U}^\c(\lambda_\c)\|^2_{L^2(\Omega_\o)}.
\end{equation*}
The Euler-Lagrange equation of (\ref{eq:inf}) is given by
\begin{equation}\label{firstOrder}
\langle \delta J(\lambda_\a, \lambda_\c), (\mu_\a, \mu_\c) \rangle ~ =~ 0, \quad \forall (\mu_\a, \mu_\c) \in \bm{\Lambda}^\a \times \bm{\Lambda}^\c,
\end{equation}
and using $(\cdot,\cdot)_{L^2(\Omega_\o)}$ to denote the $L^2$ inner product, the first variation of $J$ is
\begin{equation*}\label{firstJvar}
\langle \delta J(\lambda_\a, \lambda_\c), (\mu_\a, \mu_\c)\rangle  =
\left(\nabla\left(I\bm{U}^\a(\lambda_\a) - \bm{U}^\c(\lambda_\c)\right), \nabla\left(I\delta \bm{U}^\a(\lambda_\a)[\mu_\a] - \delta \bm{U}^\c(\lambda_\c)[\mu_\c]\right)\right)_{L^2(\Omega_\o)}.
\end{equation*}
In terms of the reduced problem, the AtC error in~\eqref{atcErrorEstimate} assumes the form
\begin{equation}\label{virtualErrorEstimate}
\|\nabla\big(I\bm{U}^\a(\lambda_\a^{\rm atc})-I\bm{u}^{\infty}_\a \big)\|_{L^2(\Omega_\a)}^2
+
\|\nabla\big(\bm{U}^\c(\lambda_\c^{\rm atc})- I\bm{u}^{\infty}_{\c}\big)\|_{L^2(\Omega_\c)}^2.
\end{equation}
%Again, recalling the previously used notations
%\begin{align}
%\lambda^{\infty}_{\a} :=~& \bm{u}^{\infty}|_{\partial_a\calL_\a} \\
%\lambda^{\infty}_{\c} :=~& \bm{u}^{\infty}|_{\Gamma_{\rm core}},
%\end{align}
%the error we wish to estimate in~\eqref{atcErrorEstimate} can be written in this framework as
%

Analysis of (\ref{virtualErrorEstimate}) requires several problem-dependent norms, and solutions of linearized problems on $\Omega_\a$ and $\Omega_\c$ define these norms. Let $\delta \bm{U}^\a(\lambda^{\infty}_\a)[\cdot]: \bm{\Lambda^\a} \to \calU^\a$  be the solution to the linearized problem\footnote{We show subsequently that $\bm{U}^\a$ is differentiable, and $\delta \bm{U}^\a(\lambda^{\infty}_\a)[\cdot]$ is the Gateaux derivative of $\bm{U}^\a$ at $\lambda^{\infty}_\a$.}
\begin{equation}\label{EL:At}
\begin{split}
\big<\delta^2 \tilde{\calE}^\a(\bm{U}^\a(\lambda^{\infty}_\a))\delta \bm{U}^\a(\lambda^{\infty}_\a)[\mu_\a], \bm{v}^\a\big> =~& 0 \quad \forall \bm{v}^\a \, \in \bm{\calU}^\a_0, \\
\delta \bm{U}^\a(\lambda^{\infty}_\a)[\mu_\a] =~& \mu_\a \quad \mbox{on} ~ \partial_\a\calL_\a,
\end{split}
\end{equation}
and let $\delta \bm{U}^\c(\lambda^{\infty}_\c)[\cdot]: \bm{\Lambda^\c} \to \calU^\c$ be the solution to a similar continuum linearized problem
\begin{equation*}\label{EL:Con}
\begin{split}
\big<\delta^2 \tilde{\calE}^\c(\bm{u}^{\rm con} )\delta \bm{U}^\c(\lambda^{\infty}_\c)[\mu_\c], \bm{v}^\c\big> =~& 0 \quad \forall \bm{v}^\c \in \bm{\calU}^\c_{h,0}, \\
\delta \bm{U}^\c(\lambda^{\infty}_\c)[\mu_\c] =~& \mu_\c \quad \mbox{on} \quad \Gamma_{\rm core}.
\end{split}
\end{equation*}
It is easy to see that
$$
\|\mu_\a\|_{\bm{\Lambda}^\a} := \|\nabla I\delta \bm{U}^\a(\lambda^{\infty}_\a)[\mu_\a]\|_{L^2(\Omega_\a)} \quad{\mbox{and}}\quad
\|\mu_\c\|_{\bm{\Lambda}^\c} := \|\nabla \delta \bm{U}^\c(\lambda^{\infty}_\c)[\mu_\c]\|_{L^2(\Omega_\c)}
$$
define norms norms on $\bm{\Lambda}^\a$, and $\bm{\Lambda}^\c$, respectively, while their sum
\begin{equation}\label{errorNorm}
\|(\mu_\a, \mu_\c)\|_{\rm err}^2 :=~ \| \mu_\a \|_{\bm{\Lambda}^\a}^2 +\| \mu_\c \|_{\bm{\Lambda}^\c}^2,
\end{equation}
is a norm on  $\bm{\Lambda}^\a \times \bm{\Lambda}^\c$. In  Section~\ref{sec:normEquiv}
we shall prove
\begin{equation*}\label{optNorm}
\|(\mu_\a, \mu_\c)\|_{\rm op} := \|\nabla\left(I\delta \bm{U}^\a(\lambda_\a^{\infty})[\mu_\a] - \delta \bm{U}^\c(\lambda^{\infty}_\c)[\mu_\c]\right) \|_{L^2(\Omega_\o)}
\end{equation*}
is a norm equivalent to $\|\cdot\|_{\rm err}$ from~\eqref{errorNorm}. We state this result below for further reference within this section.
\begin{theorem}[Norm Equivalence]\label{normEquivTheorem}
There exists $R_{\rm core}^* > 0$ such that for all $R_{\rm core} \geq R_{\rm core}^*$,
\begin{equation}\label{eq:normEquiv}
\|\cdot \|_{\op} \lesssim~ \|\cdot\|_{\err} \lesssim~ \|\cdot \|_{\op}.
\end{equation}
%(Recall the implied constants may depend upon $R_{\rm core}^*$.)
\end{theorem}

%%%%%%
%%%%%%
\subsection{The Inverse Function Theorem framework}
%%%%%%
%%%%%%
We consider the first order optimality condition (\ref{firstOrder}) for~\eqref{eq:inf},
and apply the inverse function theorem, Theorem~\ref{inverseFunctionTheorem}, with $f = \delta J$ and  $X=\bm{\Lambda}^\a \times \bm{\Lambda}^\c$ equipped with the $\|\cdot \|_{\rm op}$ norm to show that~\eqref{firstOrder} has a solution.

To apply the theorem, we must prove there exist $L, \eta, \sigma$ such that
$$
\sup_{(\lambda_\a,\lambda_\c) \text{ near } (\lambda_\a^{\infty},\lambda_\c^{\infty} )}\|\delta^3 J(\lambda_\a,\lambda_\c)\|
\leq L\,,\quad
\|\delta J(\lambda_\a^{\infty},\lambda_\c^{\infty})\|
\leq \eta,\,\quad
\mbox{and}\quad
\|(\delta^2 J(\lambda_\a^{\infty},\lambda_\c^{\infty}))^{-1}\|
\leq \sigma.
$$

Each of these results requires differentiability of the functional, $J$, which in turn requires differentiability of the functions $\bm{U}^\a$ and $\bm{U}^\c$.  We prove the necessary differentiability results and boundedness of the third derivative of $J$ in Section~\ref{secReg}.  The second bound above is a consistency error estimate and is proven in Section~\ref{S:consistency} while the final estimate is a stability result proven in Section~\ref{S:stability}.

\subsubsection{Regularity}\label{secReg}

We use the following version of the implicit function theorem to obtain existence and regularity results for
 $\bm{U}^\a$ and $\bm{U}^\c$. The theorem may be obtained by adapting the proof of the implicit function theorem in~\cite{hubbard2009} to Banach spaces and by tracking the constants involved.
\begin{theorem}[Implicit Function Theorem]\label{implicitFunctionTheorem}
Let $X$, $Y$, and $Z$ be Banach spaces with $U \subset X \times Y$ an open set.  Let $f:X \times Y \to Z$ be continuously differentiable with $(x_0, y_0) \in U$ satisfying $f(x_0, y_0) = 0$.  Suppose that $\delta_yf(x_0,y_0):Y \to Z$ is a bounded, invertible linear transformation with $\big\|(\delta_yf(x_0,y_0))^{-1}\big\| =: \theta$.  Also set $\phi := \|\delta_xf(x_0,y_0)\|$ and
\[
\sigma := \max \left\{1 + \theta\phi, \theta\right\}.
\]
If there exists $\eta$ such that
\begin{enumerate}
\item $B_{2\eta \sigma}((x_0,y_0)) \subset U$,
\item $ \| \delta f(x_1,y_1) - \delta f(x_2,y_2)\| \leq \frac{1}{2\eta \sigma^2}\|(x_1,y_1) - (x_2, y_2)\|$  for all  $(x_1, y_1), (x_2, y_2) \in B_{2\eta \sigma}((x_0,y_0))$,
\end{enumerate}
then there is a unique continuously differentiable function $g:B_\eta(x_0) \to B_{2\eta \sigma}(y_0)$ such that $g(x_0) = y_0$ and $f(x, g(x)) = 0$ for all $x \in B_\eta(x_0)$.  The derivative of $g$ is
\[
\delta g(x) = -\left[\delta_yf(x,g(x))^{-1}\right]\left[\delta_xf(x,g(x))\right].
\]
Moreover, if $f$ is ${\rm{C}}^k$, then $g$ is ${\rm{C}}^k$, and derivatives of $g$ can be bounded in terms of derivatives of $f$ and $\delta_yf(x_0,g(x_0))^{-1}$.
\end{theorem}
\begin{theorem}[Regularity of $\bm{U}^\a$]\label{atRegularity}
Under Assumptions~\ref{siteassumption} and~\ref{atCoercive}, there
exists $R_{\rm core}^*>0$ such that for all $R_{\rm core} \geq R_{\rm core}^*$, there exists a mapping
$\bm{U}^\a:\bm{\Lambda^\a} \to \bm{\calU}^\a$ such that $\bm{U}^\a(\lambda_\a)$ solves~\eqref{restrictedAtMin} and which is ${\rm C}^3$ on an open ball $V$ centered at $\lambda_\a^{\infty}$ in $\bm{\Lambda}^\a$. The radius of $V$ is independent of $R_{\rm core}$,
and the derivatives of $\bm{U}^\a$ are also bounded uniformly in $R_{\rm core} \geq R_{\rm core}^*$.
\end{theorem}
\begin{proof}
We apply Theorem \ref{implicitFunctionTheorem} with $X=\bm{\Lambda}^\a$,
$Y=\bm{\calU}^\a_0$, $Z=\left(\bm{\calU}^\a_0\right)^*$, $U = X \times Y$, and
\[
f\left(\lambda_\a, \bm{v}^\a\right) := \delta \tilde{\calE}^\a\left(h\left(\lambda_\a, \bm{v}^\a\right)\right),
\]
where  $h$ is an auxiliary function $X \times Y \to \bm{\calU}^\a$  defined by (recall $\delta \bm{U}^\a(\lambda_\a^{\infty})[\mu^\a]$ is defined to solve \eqref{EL:At})
\[
h\left(\lambda_\a, \bm{v}^\a\right) = \bm{v}^\a + \bm{u}^{\infty}_\a + \delta \bm{U}^\a(\lambda^{\infty}_\a)\left[\lambda_\a - \lambda_\a^{\infty}\right].
\]
Because $h$ is affine, $f$ is ${\rm C}^k$ provided that $\tilde{\calE}^\a$ is ${\rm C}^{k+1}$ on $\bm{\calU}^\a$.  Hence, Theorem~\ref{atLippy} implies $f$ is ${\rm C}^3$.  For the point $(x_0, y_0)$, we take the point $(\lambda_\a^{\infty}, \bm{0})$ so that $h\left(x_0, y_0\right) = \bm{u}^{\infty}_\a$.  The chain rule shows
\[
\delta_yf(x_0,y_0) =~ \delta^2\tilde{\calE}^\a\left(h\left(x_0, y_0\right)\right) \circ \delta_y h\left(x_0, y_0\right).
\]
In conjunction with $\delta_y h\left(x_0, y_0\right)[\bm{v}^\a] = \bm{v}^\a$, it follows that
$\delta_yf(x_0,y_0):Y \to Z$ is given by
\[
\langle \delta_yf(x_0,y_0)\bm{v}^\a, \bm{w}^\a\rangle  =~ \langle \delta^2\tilde{\calE}^\a\left(\bm{u}^{\infty}_\a\right)\bm{v}^\a, \bm{w}^\a\rangle .
\]

Since both $\bm{v}^\a$ and $\bm{w}^\a$ are elements of $\bm{\calU}^\a_0$ they can be extended by a constant to all of $\mathbb{Z}^d$ while keeping the gradient norms of $I\bm{v}^\a$ and $I\bm{w}^\a$ the same.
Then using Assumption~\ref{atCoercive}, we find
$$
\langle \delta_yf(x_0,y_0)\bm{v}^\a, \bm{v}^\a\rangle  =~ \langle \delta^2\tilde{\calE}^\a\left(\bm{u}^{\infty}_\a\right)\bm{v}^\a, \bm{v}^\a\rangle  =~ \langle \delta^2\calE^\a\left(\bm{u}^{\infty}\right)\bm{v}^\a, \bm{v}^\a\rangle
\geq~ \gamma_\a \| \nabla I \bm{v}^\a\|_{L^2(\mathbb{R}^d)}^2 = \gamma_\a \| \nabla I \bm{v}^\a\|_{L^2(\Omega_\a)}^2.
$$
This shows $\delta_yf(x_0,y_0)$ is coercive, and consequently that $\delta_yf(x_0,y_0)^{-1}$ exists with norm bounded  by $\theta := \gamma_\a^{-1}$.
Using again the chain rule, we obtain
\[
\delta_x f(x_0,y_0) = \delta^2\tilde{\calE}^\a\left(h\left(x_0, y_0\right)\right) \circ \delta_x h\left(x_0, y_0\right) = 0
\]
so that $\phi = \|\delta_xf(x_0,y_0)\| = 0$.

Next, observe that $h$ is Lipschitz on its entire domain with Lipschitz constant 1, and $\delta^2\tilde{\calE^\a}$ is Lipschitz with some Lipschitz constant $M$, as guaranteed by Theorem~\ref{atLippy}. As a result, $\delta f$ is Lipschitz with Lipschitz constant $M$. Now we may choose $\eta$ small enough so that $\frac{1}{2\eta\sigma^2} \leq M$, which means both conditions $(1)$ and $(2)$ in the statement of implicit function theorem are fulfilled. This allows us to deduce the existence of an implicit function $g:B_\eta(\lambda^{\infty}_\a) \to B_{2\eta\sigma}(\bm{0})$, which we use to define a mapping $\bm{U}^\a$ via
\[
\bm{U}^\a(\lambda_\a) = h\left(\lambda_\a, g(\lambda_\a)\right) = g(\lambda_\a) + \bm{u}^{\infty}_\a + \delta \bm{U}^\a(\lambda_\a^{\infty})\left[\lambda_\a - \lambda^{\infty}\right].
\]
Since $f$ is $\rm{C}^3$, the implicit function theorem ensures $g$ is also $\rm{C}^3$.  Thus $\bm{U}^\a$ is $\rm{C}^3$.  The radius of $V$ is $\eta$, which is clearly independent of $R_{\rm core}$, and the uniform bounds on the derivatives of $\bm{U}^\a$ follow by noting derivatives of $f$ correspond to derivatives of the restricted atomistic energy (which is uniformly bounded by Theorem~\ref{atLippy}) and using the final remark in the statement of the implicit function theorem.
\end{proof}

\begin{remark}
We note that the Gateaux derivative, $\delta \bm{U}^\a(\lambda_\a)[\mu_\a]$,  of $\bm{U}^\a$ at $\lambda_\a$ in the direction of $\mu_\a$ solves the problem
\begin{equation*}\label{EL:lin1}
\begin{split}
\langle \delta^2 \tilde{\calE}^\a(\bm{U}^\a(\lambda_\a))\delta \bm{U}^\a(\lambda_\a)[\mu_\a], \bm{v}^\a\rangle  =~& 0 \quad \forall \bm{v}^\a \in \bm{\calU}^\a_0, \\
\delta \bm{U}^\a(\lambda_\a)[\mu_\a] =~& \mu_\a \quad \mbox{on} \quad \partial_\a\calL_\a,
\end{split}
\end{equation*}
thus justifying our usage of notation in the proof.
\end{remark}

With only minor modifications, the proof of Theorem \ref{atRegularity} can be adapted to establish the regularity of $\bm{U}^\c$.
\begin{theorem}[Regularity of $\bm{U}^\c$]\label{contRegularity}
There exists $R_{\rm core}^*>0$ such that for all $R_{\rm core} \geq R_{\rm core}^*$, there exists a mapping
$\bm{U}^\c:\bm{\Lambda^\c} \to \bm{\calU}^\c$ such that $\bm{U}^\c(\lambda_\c)$ solves~\eqref{restrictedContMin} and which is ${\rm C}^3$ on an open ball $V$ centered at $\lambda_\c^{\infty}$ in $\bm{\Lambda}^\c$.  The derivatives of $\bm{U}^\c$ are bounded uniformly in $R_{\rm core}$, and the radius of $V$ is independent of $R_{\rm core}$.
\end{theorem}

The proof of Theorem~\ref{mainTheorem} relies on a stability result that enables the application of the inverse function theorem. This stability result requires the following auxiliary lemma.
\begin{lemma}\label{lem:aux}
There exists $R_{\rm core}^*$ such that for all $R_{\rm core} \geq R_{\rm core}^*$ and all $\mu_\a, \nu_\a \in \bm{\Lambda}^\a$ and all $\mu_\c, \nu_\c \in \bm{\Lambda}^\c$,
\begin{equation}\label{stabConj2}
\|\nabla \left(I\delta^2 \bm{U}^\a(\lambda^{\infty}_\a)[\mu_\a,\nu_\a] - \delta^2 \bm{U}^\c(\lambda^{\infty}_\c)[\mu_\c,\nu_\c]\right)\|_{L^2(\Omega_\o)}
	 \lesssim~ \|(\mu_\a, \mu_\c)\|_{\rm op} \cdot \|(\nu_\a, \nu_\c)\|_{\rm op}.
\end{equation}
\end{lemma}

\begin{proof}
The triangle inequality implies
\begin{equation}\label{eq:aux00}
\begin{split}
&\|\nabla \left(I\delta^2 \bm{U}^\a(\lambda_\a^{\infty})[\mu_\a,\nu_\a] - \delta^2 \bm{U}^\c(\lambda_\c^{\infty})[\mu_\c,\nu_\c]\right)\|_{L^2(\Omega_\o)} \\
&\qquad\qquad\qquad\qquad\qquad\quad
\leq~ \|\nabla I\delta^2 \bm{U}^\a(\lambda_\a^{\infty})[\mu_\a,\nu_\a]\|_{L^2(\Omega_\a)} + \|\nabla \delta^2 \bm{U}^\c(\lambda_\c^{\infty})[\mu_\c,\nu_\c]\|_{L^2(\Omega_\c)}.
\end{split}
\end{equation}
We then utilize Theorems~\ref{atRegularity} and~\ref{contRegularity} to obtain an upper bound on Hessian of the atomistic mapping:
\begin{align}\label{eq:aux0_a}
\| \delta^2 \bm{U}^\a(\lambda_\a^{\infty})[\mu_\a, \nu_\a]\|_{\bm{\calU}^\a} \lesssim~& \|\mu_\a\|_{\bm{\Lambda}^\a} \cdot \|\nu_\a\|_{\bm{\Lambda}^\a},
\end{align}
and a similar bound for the Hessian of the continuum mapping:
\begin{align}\label{eq:aux0_c}
\| \delta^2 \bm{U}^\c(\lambda_\c^{\infty})[\mu_\c, \nu_\c]\|_{\bm{\calU}^\c} \lesssim~& \|\mu_\c\|_{\bm{\Lambda}^\c} \cdot \|\nu_\c\|_{\bm{\Lambda}^\c}.
\end{align}
Inequalities~\eqref{eq:aux0_a}--\eqref{eq:aux0_c} may in turn be used to bound the right hand side of~\eqref{eq:aux00} and further applying the norm equivalence theorem, Theorem~\ref{normEquivTheorem}, yeilds
\begin{equation*}\label{eq:aux1}
\begin{split}
\|\nabla \left(I\delta^2 \bm{U}^\a(\lambda_\a^{\infty})[\mu_\a,\nu_\a] - \delta^2 \bm{U}^\c(\lambda_\c^{\infty})[\mu_\c,\nu_\c]\right)\|
&\quad\lesssim~\|\mu_\a\|_{\bm{\Lambda}^\a} \cdot \|\nu_\a\|_{\bm{\Lambda}^\a}  + \|\mu_\c\|_{\bm{\Lambda}^\c} \cdot \|\nu_\c\|_{\bm{\Lambda}^\c} \\
&\quad\leq~ \left(\|\mu_\a\|_{\bm{\Lambda}^\a} + \|\mu_\c\|_{\bm{\Lambda}^\c} \right) \left(\|\nu_\a\|_{\bm{\Lambda}^\a} + \|\nu_\c\|_{\bm{\Lambda}^\c}\right) \\
&\quad\lesssim~ \|(\mu_\a, \mu_\c)\|_{\rm op} \cdot \|(\mu_\a, \mu_\c)\|_{\rm op}.
\end{split}
\end{equation*}
\end{proof}

We proceed to establish regularity of the reduced space functional $J$.

\begin{theorem}[Regularity of $J$]\label{regJ}
Let $V^\a$ and $V^\c$ be the neighborhoods of $\lambda_\a^{\infty}$ and $\lambda_\c^{\infty}$ in $\bm{\Lambda}_\a$ and $\bm{\Lambda}_\c$ on which $\bm{U}^\a$ and $\bm{U}^\c$ are ${\rm C}^3$.  Then $J$ is ${\rm C}^3$ on $V^\a \times V^\c$ and its $\ell^{\rm th}$ derivatives can be bounded by derivatives of $\bm{U}^\a$ and $\bm{U}^\c$ of order at most $\ell$.
\end{theorem}

\begin{proof}
Theorems \ref{atRegularity}--\ref{contRegularity} guarantee that $\bm{U}^\a$ and $\bm{U}^\c$ are ${\rm C}^3$ on $V^\a$ and $V^\c$.  Moreover, the interpolant, $I$, is a linear operator so $\lambda^\a \mapsto I \bm{U}^\a(\lambda^\a)$ will also be ${\rm C}^3$ on $V^\a$.
The assertion of the theorem then follows from the fact that $J = \|\nabla I\bm{U}^\a(\lambda_\a) - \nabla \bm{U}^\c(\lambda_\c)\|^2_{L^2(\Omega_\o)}$ is a composition of a quadratic form and the ${\rm C}^3$ functions $ I \bm{U}^\a(\lambda^\a)$ and $ \bm{U}^\c(\lambda^\c)$.
\end{proof}

%%%%%%
\subsubsection{Consistency}\label{S:consistency}
%%%%%%
The consistency error measures the extent to which $\bm{u}^{\infty}$ fails to satisfy the approximate problem, which in this case is the reduced space formulation (\ref{eq:inf}). Thus, we seek an upper bound for

\begin{equation}\label{consistency}
\|\delta J(\lambda_\a^{\infty}, \lambda_\c^{\infty})\|_{\rm op} \\
~= \sup_{\|\left(\mu_\a, \mu_\c\right)\|_{\rm op} = 1}\left|\left(\nabla \left(I\bm{U}^\a(\lambda^{\infty}_\a) - \bm{U}^\c(\lambda^{\infty}_\c)\right), \nabla \left(I\delta \bm{U}^\a(\lambda^{\infty}_\a)[\mu_\a] - \delta \bm{U}^\c(\lambda^{\infty}_\c)[\mu_\c]\right)\right)_{L^2(\Omega_\o)}\right|.
\end{equation}
\begin{theorem}[Consistency Error]\label{consProp}
There exists $R_{\rm core}^* > 0$ such that for all $R_{\rm core} \geq R_{\rm core}^*$, we have
\begin{equation}\label{cons2}
\begin{split}
\|\delta J(\lambda^{\infty}_\a, \lambda^{\infty}_\c)\|_{\rm op} \lesssim~& R_{\rm core}^{-d/2-1} + R_{\c}^{-d/2}.
%\|h \nabla^2 \tilde{I}\bm{u}^{\infty}\|_{L^2(\Omega_\c)} + \|\nabla\tilde{I}\bm{u}^{\infty}\|_{L^2(\mathbb{R}^d\backslash B_{3r_\c/4}(0))} \\
%&\qquad+  \|\nabla^3 \tilde{I}\bm{u}^{\infty}\|_{L^2(\Omega_\c)} + \|\nabla^2 \tilde{I}\bm{u}^{\infty}\|_{L^4(\Omega_\c)}.
\end{split}
\end{equation}
\end{theorem}

\begin{proof}
Applying the Cauchy-Schwarz inequality to (\ref{consistency}) yeields
\begin{equation*}
\begin{array}{l}
\displaystyle
\|\delta J(\lambda_\a^{\infty}, \lambda_\c^{\infty})\|_{\rm op}
\\[2ex]
\displaystyle
\qquad
\leq~ \sup_{\|\left(\mu_\a, \mu_\c\right)\|_{\rm op} = 1}\|\nabla \left(I\bm{U}^\a(\lambda^{\infty}_\a) - \bm{U}^\c(\lambda^{\infty}_\c)\right)\|_{L^2(\Omega_\o)}
 \|  \nabla \left(I\delta \bm{U}^\a(\lambda^{\infty}_\a)[\mu_\a] - \delta \bm{U}^\c(\lambda^{\infty}_\c)[\mu_\c]\right)\|_{L^2(\Omega_\o)} \cdot
 \\[2ex]
\displaystyle
\qquad
=~  \|\nabla \left(I\bm{U}^\a(\lambda^{\infty}_\a) - \bm{U}^\c(\lambda^{\infty}_\c)\right)\|_{L^2(\Omega_\o)}.
\end{array}
\end{equation*}
Note that $\lambda^{\infty}_\a$ and $\lambda^{\infty}_\c$ are traces of the exact atomistic solution and so,
\[
\|\nabla \left(I\bm{U}^\a(\lambda^{\infty}_\a) - \bm{U}^\c(\lambda^{\infty}_\c)\right)\|_{L^2(\Omega_\o)} = \|\nabla I \bm{u}^{\infty}_\a - \nabla \bm{u}^{\rm con}\|_{L^2(\Omega_\o)},
\]
is the simply the continuum error made by replacing the atomistic model with the continuum model on $\Omega_\o$.  Thus, \eqref{cons2} follows directly from~\eqref{contModelEstimate} in Theorem~\ref{contModelError}.
\end{proof}

%%%%%%
%%%%%%
\subsubsection{Stability}\label{S:stability}
%%%%%%
%%%%%%
In this section we prove that the bilinear form
$
\langle  \delta^2 J(\lambda^{\infty}_\a, \lambda^{\infty}_\c)\, \cdot , \cdot \rangle
$
is coercive.
\begin{theorem}\label{stabilityTheorem}
There exists $R_{\rm core}^*$ such that for each $R_{\rm core} \geq R_{\rm core}^*$
\begin{equation*}\label{Jstab}
\langle  \delta^2 J(\lambda^{\infty}_\a, \lambda^{\infty}_\c)(\mu_\a, \mu_\c) , (\mu_\a, \mu_\c) \rangle  \geq~ {\textstyle\frac12} \|(\mu_\a, \mu_\c)\|_{\rm op}^2, \quad \forall (\mu_\a, \mu_\c) \in \bm{\Lambda}^\a \times \bm{\Lambda}^\c.
\end{equation*}
\end{theorem}

\begin{proof}
The Hessian of $J$ is given by
\begin{equation*}\label{jHess}
\begin{split}
&\langle  \delta^2 J(\lambda^{\infty}_\a, \lambda^{\infty}_\c)(\mu_\a, \mu_\c) , (\mu_\a, \mu_\c) \rangle  =~ \| \nabla \left(I\delta \bm{U}^\a(\lambda^{\infty}_\a)[\mu_\a] -  \delta \bm{U}^\c(\lambda^{\infty}_\c)[\mu_\c]\right)\|_{L^2(\Omega_\o)}^2 \\
&\qquad+ \left( \nabla \left(I\bm{U}^\a(\lambda^{\infty}_\a) - \bm{U}^\c(\lambda^{\infty}_\c)\right),   \nabla \left(I\delta^2 \bm{U}^\a(\lambda^{\infty}_\a)[\mu_\a,\mu_\a] -  \delta^2 \bm{U}^\c(\lambda^{\infty}_\c)[\mu_\c,\mu_\c]\right)\right)_{L^2(\Omega_\o)}.
\end{split}
\end{equation*}

Using the definition of $\|\cdot\|_{\rm op}$, this is equivalent to
\begin{equation*}\label{hessEst}
\begin{split}
&\langle  \delta^2 J(\lambda^{\infty}_\a, \lambda^{\infty}_\c)(\mu_\a, \mu_\c) , (\mu_\a, \mu_\c) \rangle  =\\
&\qquad\qquad ~ \|(\mu_\a, \mu_\c)\|_{\rm op}^2 + \left( \nabla \left(I\bm{U}^\a(\lambda^{\infty}_\a) - \bm{U}^\c(\lambda^{\infty}_\c)\right),   \nabla \left(I\delta^2 \bm{U}^\a(\lambda^{\infty}_\a)[\mu_\a,\mu_\a] -  \delta^2 \bm{U}^\c(\lambda^{\infty}_\c)[\mu_\c,\mu_\c]\right)\right)_{L^2(\Omega_\o)}.
\end{split}
\end{equation*}
Lemma~\ref{lem:aux} implies the existence of $R_{\rm core}^{*,1}$ and $C_{\rm stab}$ such that for all $R_{\rm core} \geq R_{\rm core}^{*,1}$,
\begin{equation*}\label{stabConj8}
\|
	\nabla \left(I\delta^2 \bm{U}^\a(\lambda^{\infty}_\a)[\mu_\a,\mu_\a]
	-
	\delta^2 \bm{U}^\c(\lambda^{\infty}_\c)[\mu_\c,\mu_\c]\right)\|_{L^2(\Omega_\o)}
	 \leq~ C_{\rm stab}\|(\mu_\a, \mu_\c)\|_{\rm op}^2.
\end{equation*}
We then have that
\begin{align*}
&\left( \nabla \left(I\bm{U}^\a(\lambda^{\infty}_\a) - \bm{U}^\c(\lambda^{\infty}_\c)\right),   \nabla \left(I\delta^2 \bm{U}^\a(\lambda^{\infty}_\a)[\mu_\a,\mu_\a] -  \delta^2 \bm{U}^\c(\lambda^{\infty}_\c)[\mu_\c,\mu_\c]\right)\right)_{L^2(\Omega_\o)} \\
&\qquad \geq~ - \|\nabla \left(I\bm{U}^\a(\lambda^{\infty}_\a) - \bm{U}^\c(\lambda^{\infty}_\c)\right)\|_{L^2(\Omega_\o)} \cdot
 \|\nabla \left(I\delta^2 \bm{U}^\a(\lambda^{\infty}_\a)[\mu_\a,\mu_\a] -  \delta^2 \bm{U}^\c(\lambda^{\infty}_\c)[\mu_\c,\mu_\c]\right)\|_{L^2(\Omega_\o)} \\
&\qquad \geq~ - C_{\rm stab}\|\nabla \left(I\bm{U}^\a(\lambda^{\infty}_\a) - \bm{U}^\c(\lambda^{\infty}_\c)\right)\|_{L^2(\Omega_\o)}\cdot \|(\mu_\a, \mu_\c)\|_{\rm op}^2.
\end{align*}
This implies
\begin{align*}
&\langle  \delta^2 J(\lambda^{\infty}_\a, \lambda^{\infty}_\c)(\mu_\a, \mu_\c) , (\mu_\a, \mu_\c) \rangle
\geq~ \|(\mu_\a, \mu_\c)\|_{\rm op}^2 - C_{\rm stab} \|\nabla \left(I\bm{U}^\a(\lambda^{\infty}_\a) - \bm{U}^\c(\lambda^{\infty}_\c)\right)\|_{L^2(\Omega_\o)}\cdot \|(\mu_\a, \mu_\c)\|_{\rm op}^2 \\
&\qquad =~ \left(1- C_{\rm stab} \|\nabla \left(I\bm{U}^\a(\lambda^{\infty}_\a) - \bm{U}^\c(\lambda^{\infty}_\c)\right)\|_{L^2(\Omega_\o)}\right)\|(\mu_\a, \mu_\c)\|_{\rm op}^2,
\end{align*}
where we recall $\|\nabla \left(I\bm{U}^\a(\lambda^{\infty}_\a) - \bm{U}^\c(\lambda^{\infty}_\c)\right)\|_{L^2(\Omega_\o)}$ is the continuum error.  By Theorem~\ref{contModelError}, there exists $R_{\rm core}^{*,2}$ such that for all $R_{\rm core} \geq R_{\rm core}^{*,2}$,
\[
\left(1- C_{\rm stab} \|\nabla \left(I\bm{U}^\a(\lambda^{\infty}_\a) - \bm{U}^\c(\lambda^{\infty}_\c)\right)\|_{L^2(\Omega_\o)}\right) \geq 1/2.
\]
Taking $R_{\rm core}^* = \max\left\{R_{\rm core}^{*,1}, R_{\rm core}^{*,2}\right\}$ completes the proof.
\end{proof}

%%%%%%
%%%%%%
\subsubsection{Error Estimate}
%%%%%%
%%%%%%
Having proven regularity of $J$, a consistency estimate, and a stability result, we are now in a position to prove our main error result, Theorem~\ref{mainTheorem}.  This will be a consequence of following theorem providing important information about the AtC formulation.

\begin{theorem}\label{thm:errorEstimate}
There exists $R_{\rm core}^*>0$ such that for all $R_{\rm core} \geq R_{\rm core}^*$,
the reduced space problem \eqref{eq:inf} has a solution $(\lambda_\a^{\rm atc}, \lambda_\c^{\rm atc})$, such that
\begin{equation}\label{opEst}
\|(\lambda^{\infty}_\a, \lambda^{\infty}_\c)- (\lambda^{\rm atc}_\a, \lambda^{\rm atc}_\c)\|_{\rm op} \lesssim~ R_{\rm core}^{-d/2-1} + R_{\c}^{-d/2} .
\end{equation}
\end{theorem}
\begin{proof}
We apply the inverse function theorem, Theorem~\ref{inverseFunctionTheorem}, with $f = \delta J$,
$X=\bm{\Lambda}^\a \times \bm{\Lambda}^\c$ endowed with the norm $\|\cdot\|_{\rm op}$, $Y= \left(\bm{\Lambda}^\a \times \bm{\Lambda}^\c \right)^*$  endowed with the dual norm  $\|\cdot \|_{\rm op^*}$, and $x_0 = (\lambda^{\infty}_\a, \lambda^{\infty}_\c)$. Let $R_{\rm core}^*$ be the maximum of the $R_{\rm core}^*$ guaranteed to exist in Theorems~\ref{atRegularity},~\ref{contRegularity},~\ref{consProp} and,~\ref{stabilityTheorem}.  Noting that $\|f(x_0)\|_{\rm op^*}$ is the consistency error defined in Section~\ref{S:consistency}, Theorem~\ref{consProp}, implies the bound
\[
\|f(x_0)\|_{\rm op^*} \lesssim R_{\rm core}^{-d/2-1} + R_{\c}^{-d/2} =: \eta.
\]
Observe also that $\delta f(x_0) = \delta^2 J(\lambda^{\infty}_\a, \lambda^{\infty}_\c)$ and the existence of a coercivity constant, $\sigma := 1/2$, from Section~\ref{S:stability} implies $\|\delta f(x_0)^{-1}\| < \sigma^{-1} = 2$.

Furthermore, Theorems~\ref{atRegularity} and~\ref{contRegularity} provide constants  $\eta_\a$ and $\eta_\c$ such that
 $\bm{U}^\a$ and $\bm{U}^\c$ are ${\rm C}^3$ on $B_{\eta_\a}(\lambda^{\infty}_\a)$ and $B_{\eta_\c}(\lambda^{\infty}_\c)$ respectively. By Theorem~\ref{regJ}, $\delta^3 J$ is bounded by derivatives of $\bm{U}^\a$ and $\bm{U}^\c$ of order at most $3$.  Furthermore, Theorems~\ref{atRegularity} and~\ref{contRegularity} state that derivatives of $\bm{U}^\a$ and $\bm{U}^\c$ are uniformly bounded in $R_{\rm core}$.  We may therefore conclude that the third derivative of $J$ is also uniformly bounded in $R_{\rm core}$ $\geq R_{\rm core}^*$.
This implies $\delta f = \delta^2 J$ is Lipschitz on $B_{\eta_\a}(\lambda^{\infty}_\a) \times B_{\eta_\c}(\lambda^{\infty}_\c)$ with a Lipschitz constant that we denote by $L$.

The bound $2L\eta(2)^{2} < 1$ holds since the consistency error $\eta$ may be made small for $R_{\rm core}^*$ large enough.  Analogously, $B_{4\eta}(\lambda^{\infty}_\a, \lambda^{\infty}_\c) \subset B_{\eta_\a}(\lambda^{\infty}_\a) \times B_{\eta_\c}(\lambda^{\infty}_\c)$ for small enough $\eta$.  Theorem~\ref{inverseFunctionTheorem}, can now be invoked to deduce the existence of a minimizer, $(\lambda_\a^{\rm atc}, \lambda_\c^{\rm atc}) \in B_{4\eta}(\lambda^{\infty}_\a, \lambda^{\infty}_\c)$ of $J$, satisfying the stated bounds~\eqref{opEst}.
\end{proof}

We now provide a proof of Theorem~\ref{mainTheorem}, which is our main result.

\begin{proof}[Proof of Theorem~\ref{mainTheorem}]

Let $R_{\rm core}^*$ be the maximum of the $R_{\rm core}^*$ from Theorem~\ref{thm:errorEstimate} and Theorem~\ref{normEquivTheorem} so there exists $(\lambda_\a^{\rm atc}, \lambda_\c^{\rm atc})$ satisfying~\eqref{opEst}. Furthermore, $\left(\bm{U}^\a(\lambda_\a^{\rm atc}), \bm{U}^\c(\lambda_\c^{\rm atc})\right)$ solve the minimization problem~\eqref{atcOptEquiv}. Hence,
\begin{equation*}\label{error:analysisBroke}
\begin{split}
&
\|\nabla\left(I\bm{u}^{\infty}_\a - I\bm{u}^{\rm atc}_\a\right)\|_{L^2(\Omega_\a)}^2
+
\|\nabla\left(I\bm{u}^{\infty}_\c - \bm{u}^{\rm atc}_\c\right)\|_{L^2(\Omega_\c)}^2
\\
&=~ \| \nabla I\left(\bm{u}^{\infty} - \bm{U}^\a(\lambda_\a^{\rm atc})\right) \|_{L^2(\Omega_\a)}^2 + \| \nabla \left(I\bm{u}^{\infty} - \bm{U}^\c(\lambda_\c^{\rm atc}) \right) \|_{L^2(\Omega_\c)}^2 \\
%&=~ \| \nabla  \left(\bm{U}^\a(\lambda^{\infty}_\a) - \bm{U}^\a(\lambda_a^{\rm atc})\right) \|_{L^2(\Omega_\a)}^2
%+~ \| \nabla \left(I\bm{u}^{\infty} - \bm{U}^\c(\lambda^{\infty}_\c) + \bm{U}^\c(\lambda^{\infty}_\c) - \bm{U}^\c(\lambda_c^{\rm atc}) \right) \|_{L^2(\Omega_\c)}^2  \\
&\lesssim
\| \nabla I\left(\bm{U}^\a(\lambda^{\infty}_\a) \!-\! \bm{U}^\a(\lambda_\a^{\rm atc})\right) \|_{L^2(\Omega_\a)}^2 + \| \nabla \left(I\bm{u}^{\infty}\! - \!\bm{U}^\c(\lambda^{\infty}_\c)\right)\|_{L^2(\Omega_\c)}^2
+\|\nabla\left(\bm{U}^\c(\lambda^{\infty}_\c) \!-\! \bm{U}^\c(\lambda_\c^{\rm atc})\right) \|_{L^2(\Omega_\c)}^2.
\end{split}
\end{equation*}
The second term above is the continuum error. To handle the remaining terms we recall that $\bm{U}^\a$ and $\bm{U}^\c$ are Lipschitz  on $B_{\eta_\a}(\lambda^{\infty}_\a)$ and $B_{\eta_\c}(\lambda^{\infty}_\c)$ by virtue of $\delta \bm{U}^\a$ and $\delta \bm{U}^\c$  being uniformly bounded on these sets. Then, using norm-equivalence (\ref{eq:normEquiv}), Theorem \ref{contModelError} and Theorem \ref{thm:errorEstimate}  yields
\begin{equation*}\label{error:analysis}
\begin{split}
&\| \nabla \left(I\bm{u}^{\infty} - I\bm{u}^{\rm atc}_\a\right) \|_{L^2(\Omega_\a)}^2 + \| \nabla \left(I\bm{u}^{\infty} - \bm{u}^{\rm atc}_\c\right) \|_{L^2(\Omega_\c)}^2 \\
&~\quad\lesssim~ \| \lambda^{\infty}_\a - \lambda_\a^{\rm atc} \|_{\bm{\Lambda}^\a}^2 + \| \nabla \left(I\bm{u}^{\infty} - \bm{U}^\c(\lambda^{\infty}_\c)\right)\|_{L^2(\Omega_\c)}^2
+~  \|\lambda^{\infty}_\c - \lambda_\c^{\rm atc} \|_{\bm{\Lambda}^\c}^2 \\
&~\quad=~ \|(\lambda^{\infty}_\a, \lambda^{\infty}_\c) - (\lambda_\a^{\rm atc}, \lambda_\c^{\rm atc})\|_{\rm err}^2 + \| \nabla \left(I\bm{u}^{\infty} - \bm{U}^\c(\lambda^{\infty}_\c)\right)\|_{L^2(\Omega_\c)}^2 ~\lesssim~ R_{\rm core}^{-d-2} + R_{\c}^{-d}.
\end{split}
\end{equation*}
Taking square roots completes the proof.
\end{proof}

%%%%
%%%%
\section{Norm Equivalence}\label{sec:normEquiv}
%%%%
%%%%
The main result of this section is the norm equivalence result stated in Theorem~\ref{normEquivTheorem}. The proof of the left-hand inequality, $\|(\mu_\a,\mu_\c)\|_{\op} \lesssim \|(\mu_\a,\mu_\c)\|_{\err}$, is clear so we focus only on the right-hand inequality. We recall that the finite element mesh $\mathcal{T}_h$ is subject to a minimum angle condition for some $\beta>0$ and state a precise version of the right inequality in Theorem~\ref{normEquivTheorem}.
\begin{theorem}\label{th:norm_equiv}
There exists $C,R_{\rm core}^* > 0$ such that for all domains $\Omega_\a, \Omega_\c$ and meshes $\calT_h$ constructed according to the guidelines of Section~\ref{sec:approximation} (in particular $\psi_\a R_\core = R_\a$) with $R_{\rm core} \geq R_{\rm core}^*$, there holds
\begin{equation}\label{eq:norm_equiv}
\|(\mu_\a,\mu_\c)\|_{\err} \leq C \|(\mu_\a,\mu_\c)\|_{\op}
\quad\forall (\mu_\a,\mu_\c) \in \bm{\Lambda}^\a\times \bm{\Lambda}^\c .
\end{equation}
Equivalently, for all $(\bm{w}^\a, \bm{w}^\c) \in \bm{\mathcal{U}}^\a \times \bm{\mathcal{U}}^\c_h$ such that
\begin{align}
\label{eq:wan_def}
& \<\delta^2 \tilde{\mathcal{E}}^\a(\bm{u}^{\infty}_\a) \bm{w}^\a, \bm{v}^\a\> =~ 0 \quad\forall \bm{v}^\a \in \bm{\mathcal{U}}^\a_0
\qquad\text{ and}
\\ \label{eq:wcn_def}
& \<\delta^2 \tilde{\mathcal{E}}^\c(\bm{u}^{\rm con}) \bm{w}^\c, \bm{v}^\c\> =~ 0 \quad \forall \bm{v}^\c \in \bm{\mathcal{U}}^\c_{h,0}
\end{align}
we have
\begin{equation}\label{eq:norm_equiv_alt}
\|\nabla I\bm{w}^\a\|_{L^2(\Omega_\a)}^2 + \|\nabla \bm{w}^\c\|_{L^2(\Omega_\c)}^2
\leq C
\|\nabla\left(I\bm{w}^\a - \bm{w}^\c\right)\|_{L^2(\Omega_\o)}^2.
\end{equation}
\end{theorem}
Equivalence of \eqref{eq:norm_equiv} and \eqref{eq:norm_equiv_alt} follows directly from definitions of $\|\cdot\|_{\err}$, $\|\cdot\|_{\op}$, $\bm{U}^\a$, and $\bm{U}^\c$.

In Section~\ref{reduction} we show that proving Theorem~\ref{th:norm_equiv}  reduces to proving the following result.

\begin{theorem}\label{lem:norm_equiv:desired_result}
There exists $0<c<1$ and $R_{\rm core}^* >0 $ such that for all domains $\Omega_\a, \Omega_\c$ and meshes $\calT_h$ satisfying the requirements of Section~\ref{sec:approximation} and $R_{\rm core} \geq R_{\rm core}^*$,
\begin{equation*}\label{eq:norm_equiv:desired_result}
\sup_{\bm{w}^\a, \bm{w}^\c\ne 0} \frac{\left(\nabla I \bm{w}^\a, \nabla \bm{w}^\c\right)}{\|\nabla (I \bm{w}^\a)\|_{L^2(\Omega_\o)} \|\nabla \bm{w}^\c\|_{L^2(\Omega_\o)}} \leq c,
\end{equation*}
for all $(\bm{w}^\a, \bm{w}^\c) \in \bm{\mathcal{U}}^\a \times \bm{\mathcal{U}}^\c_h$ such that
\begin{align*}
& \<\delta^2 \tilde{\mathcal{E}}^\a(\bm{u}^{\infty}_\a) \bm{w}^\a, \bm{v}^\a\> = 0 \quad\forall \bm{v}^\a \in \bm{\mathcal{U}}^\a_0,
\\
& \<\delta^2 \tilde{\mathcal{E}}^\c( \bm{u}^{\rm con}) \bm{w}^\c, \bm{v}^\c\> = 0 \quad \forall \bm{v}^\c \in \bm{\mathcal{U}}^\c_{h,0}.
\end{align*}
\end{theorem}
We prove Theorem \ref{lem:norm_equiv:desired_result} in Section ~\ref{proofSection} by using extension results  from Theorems~\ref{stein}--\ref{buren}. The latter allow us to bound solutions to the atomistic and continuum subproblems in terms of the solution on $\Omega_\o$ only.

%%%%%%
%%%%%%
\subsection{Reduction}\label{reduction}
%%%%%%
%%%%%%
Before proving Theorem~\ref{lem:norm_equiv:desired_result} in Section~\ref{proofSection}, here we show that it does indeed imply the assertion of Theorem~\ref{th:norm_equiv}. The first step is to bound solutions of the atomistic and continuum problems in terms of their values over the overlap region. To this end as well as for the proof, of Theorem~\ref{th:norm_equiv}, we argue by contradiction. Our argument involves scaled versions of~\eqref{eq:wan_def} and~\eqref{eq:wcn_def}.
We distinguish objects in the scaled domain by using a tilde accent,
 i.e. $\tilde{\calL}_{\a,n} = \epsilon_n\calL_{\a,n}$.

In each proof, we will consider sequences $R_{\rm core}^{*,n} \to \infty$ and $R_{\c,n} \to \infty$ with $R_{\c,n}/R_{\rm core}^{*,n} \to \infty$ with corresponding domains $\Omega_{\a,n}, \Omega_{\c,n}$, etc. and lattices $\calL_{\a,n}, \calL_{\c,n}$, etc.  Given $\bm{w}^\a_n$ and $\bm{w}^\c_n$, we will then set $\eps_n=1/R_{\core,n}$, and scale by $\eps_n$ to obtain functions $\tilde{\bm{w}}^\c_n(\eps_n x) = \eps_n \bm{w}^\c_n(x)$ and $\tilde{\bm{w}}^\a_n(\eps_n x) = \eps_n \bm{w}^\a_n(x)$.  Thus, each $\tilde{\bm{w}}^\a_n$ is defined on $\tilde{\calL}_{\a,n} = \epsilon_n\calL_{\a,n}$.  Note also that the domains $\tilde{\Omega}_\core := \epsilon_n\Omega_{\core,n}$ and $\tilde{\Omega}_{\a}$ have fixed radii of $1$ and $\psi_\a$ respectively.
The domains in the sequence $\{\tilde{\Omega}_{\c,n}\}$ have fixed inner boundaries but their outer boundaries tend to infinity since $R_{\c,n}/R_{\rm core}^{*,n} \to \infty$.  Because each $\bm{w}^\c_n$ is constant on the outer boundary of $\Omega_{\c,n}$, we may extend each of them outside of this region to infinity to obtain scaled functions $\tilde{\bm{w}}^\c_n$ defined on $\tilde{\Omega}_\c := \mathbb{R}^n\backslash\tilde{\Omega}_\core$.  Using this notation, we also have $\tilde{\calL}_n := \epsilon_n\calL$.

The functions $\tilde{\bm{w}}^\a_n$ and $\tilde{\bm{w}}^\c_n$ now satisfy scaled versions of~\eqref{eq:wan_def} and~\eqref{eq:wcn_def} in which the displacement spaces are parametrized by $n$ in the obvious manner: $\tilde{\bm{\calU}}^\a_{n}, \tilde{\bm{\calU}}^\a_{0,n}, \tilde{\bm{\calU}}^\c_{h,n}$, and $\tilde{\bm{\calU}}^\c_{h,0,n}$.   For clarity, we introduce several new notations.
%Recall that the site potential is $V_\xi: (\mathbb{R}^d)^{\calR} \to \mathbb{R}$.
We use $V_{\xi,\rho}$ to denote the partial derivative of $V_\xi$ with respect to the finite difference $D_\rho u$ and $V_{\xi, \rho \tau}$ to denote second partial derivatives.  We further define scaled finite differences and finite difference stencils for $\xi \in \tilde{\calL}_{\a,n}$ and $\rho \in \calR$ by
$$
D_{\epsilon_n\rho}\tilde{\bm{u}}(\xi) = \frac{\tilde{\bm{u}}(\xi + \epsilon_n\rho) - \tilde{\bm{u}}(\xi)}{\epsilon_n}
\quad\mbox{and}\quad
D_{\epsilon_n}\tilde{\bm{u}}(\xi) = \left(D_{\epsilon_n\rho}\tilde{\bm{u}}(\xi)\right)_{\rho \in \calR}.
$$
The norm~\eqref{equivNorm} scales to
\begin{equation*}\label{equivNorm1}
\|D_{\epsilon_n}\tilde{\bm{v}}\|_{\ell^2_{\epsilon_n}(\tilde{\calL}^{\circ\circ}_{\a,n})}^2 = \sum_{\xi \in  \tilde{\calL}^{\circ\circ}_{\a,n}} \sup_{\rho \in \calR}|D_{\epsilon_n\rho}\tilde{\bm{v}}|^2\epsilon_n^d,
\end{equation*}
for which there continues to hold
\[
 \|D_{\epsilon_n}\tilde{\bm{v}}\|_{\ell^2_{\epsilon_n}(\tilde{\calL}^{\circ\circ}_{\a,n})} \lesssim \|\nabla I_n \tilde{\bm{v}}\|_{L^2(\tilde{\Omega}_{\a,n})}.
\]

The function $\tilde{\bm{w}}^\a_n$ satisfies the following scaled variational equation:
\begin{align}
& \sum_{\xi \in \tilde{\calL}_{\a,n}^{\circ\circ}} \sum_{\rho, \tau \in \calR } V_{\xi, \rho\tau}(D_{\epsilon_n}\tilde{\bm{u}}^{\infty}_{\a,n}(\xi)) \cdot D_{\epsilon_n\rho}\tilde{\bm{w}}^\a_n, D_{\epsilon_n\tau} \tilde{\bm{v}}^\a \epsilon^d =~ 0 \quad\forall \bm{v}^\a \in \tilde{\bm{\calU}}^\a_{0,n}, \nonumber \\
&~\equiv \sum_{\xi \in \tilde{\calL}_{\a,n}^{\circ\circ}} V_\xi''(D_{\epsilon_n}\tilde{\bm{u}}^{\infty}_{\a,n}(\xi)) \!:\! D_{\epsilon_n}\tilde{\bm{w}}^\a_n : D_{\epsilon_n}\tilde{\bm{v}}^\a \epsilon_n^d =~ 0 \quad\forall \tilde{\bm{v}}^\a \in \tilde{\bm{\mathcal{U}}}^\a_{0,n}. \label{eq:wan_def_scale}
\end{align}

It will be convenient to express (\ref{eq:wan_def_scale}) as an integral for those specific $\tilde{\bm{v}}^\a$ for which $D_{\epsilon_n}\tilde{\bm{v}}^\a$ vanishes on $\tilde{\calL}_{\a,n}\backslash\tilde{\calL}_{\a,n}^{\circ\circ}$ and where $V_\xi \neq V$.  This requires an additional tool.  The cell, $\varsigma_\xi$, based on $\xi \in \tilde{\calL}_{n}$ is
\[
\varsigma_\xi := \left\{ x \in \mathbb{R}^d:  0 \leq x_i-\xi_i < \epsilon_n, i=1,\ldots,d\right\}.
\]
Let $\bar{I}_n$ be a piecewise constant interpolation operator defined by
\[
\bar{I}_n f(x) :=  f(\xi) \quad \mbox{where} \, x \in \varsigma_\xi.
\]
Then for such a $\tilde{\bm{v}}^\a$,
\begin{equation}\label{intForm}
\begin{split}
&\sum_{\xi \in \tilde{\calL}_{\a,n}^{\circ\circ}} V_\xi''( D_{\epsilon_n}\tilde{\bm{u}}^{\infty}_{\a,n}(\xi)) \!:\! D_{\epsilon_n}\tilde{\bm{w}}^\a_n : D_{\epsilon_n}\tilde{\bm{v}}^\a \epsilon_n^d
=~ \sum_{\xi \in \tilde{\calL}_{\a,n}^{\circ\circ}} V_\xi''( D_{\epsilon_n}\tilde{\bm{u}}^{\infty}_{\a,n}(\xi)) \!:\! D_{\epsilon_n}\tilde{\bm{w}}^\a_n : D_{\epsilon_n}\tilde{\bm{v}}^\a \, {\vol}(\varsigma_\xi \cap \tilde{\Omega}_\a) \\
&\quad=~ \sum_{\xi \in \tilde{\calL}_{\a,n}} V_\xi''( D_{\epsilon_n}\tilde{\bm{u}}^{\infty}_{\a,n}(\xi)) \!:\! D_{\epsilon_n}\tilde{\bm{w}}^\a_n : D_{\epsilon_n}\tilde{\bm{v}}^\a \, {\vol}(\varsigma_\xi \cap \tilde{\Omega}_\a) \\
&\quad=~ \int_{\tilde{\Omega}_{\a}} \bar{I}_n  V''( D_{\epsilon_n}\tilde{\bm{u}}^{\infty}_{\a,n}) \!:\! \bar{I}_n D_{\epsilon_n}\tilde{\bm{w}}^\a_n : \bar{I}_nD_{\epsilon_n}\tilde{\bm{v}}^\a \, dx \\
&\quad=~ \int_{\tilde{\Omega}_{\a}} \bar{I}_n  V''( D_{\epsilon_n}\tilde{\bm{u}}^{\infty}_{\a,n}) \!:\! \bar{I}_n D_{\epsilon_n}\tilde{\bm{w}}^\a_n : \bar{I}_nD_{\epsilon_n}\tilde{\bm{v}}^\a \, dx.
\end{split}
\end{equation}
Observe that we have replaced $V_\xi''$ with $V''$ in the integral since $D_{\epsilon_n}\tilde{\bm{v}}^\a$ is assumed to vanish where $V \neq V_\xi$.

Similarly, $\tilde{\bm{w}}^\c_n$ satisfies an analogous scaled version of~\eqref{eq:wcn_def_scale}:
\begin{equation} \int_{\tilde{\Omega}_{\c,n}} \sum_{\rho, \tau \in \mathcal{R} } \langle V''_{,\rho\tau}\left(\nabla \tilde{\bm{u}}^{\rm con}_n \calR \right)\nabla_\rho \tilde{\bm{w}}^\c_n,\nabla_{\tau} \tilde{\bm{v}}^\c\rangle  \, dx
%=~ 0\quad \forall\tilde{\bm{v}}^\c \in \tilde{\bm{\mathcal{U}}}^\c_{h,0,n} \nonumber \\
\equiv \int_{\tilde{\Omega}_{\c,n}} W''(\nabla \tilde{\bm{u}}^{\rm con}_n) : \nabla \tilde{\bm{w}}^\c_n : \nabla \tilde{\bm{v}}^\c_n \, dx=~ 0 \quad \forall \tilde{\bm{v}}^\c \in \tilde{\bm{\mathcal{U}}}^\c_{h,0,n} . \label{eq:wcn_def_scale}
\end{equation}
%Further define the fourth order tensor, $\bbC$, by the duality pairings
%\begin{align*}
%&\bbC : \mG : \mF := \sum_{\rho, \tau \in \mathcal{R} } \langle V''_{,\rho\tau}( 0) \mG\rho,\mF\tau\rangle \\
%&\bbC : D_{\epsilon_n}v : D_{\epsilon_n}w :=  \sum_{\rho, \tau \in \mathcal{R} } \langle V''_{,\rho\tau}( 0) D_{\epsilon_n\rho}v,D_{\epsilon_n\tau}w\rangle.
%\end{align*}My version:
	Further define the fourth order tensor, $\bbC = W''(\bm{0})$ and note the relation
\[
	(\bbC : \mG) : \mF := \sum_{\rho, \tau \in \mathcal{R} } V''_{,\rho\tau}(\bm{0}) \mG\rho \cdot \mF\tau = (V''(\bm{0}):(\mF\calR)):(\mG\calR) \quad \forall \mG,\mF \in \mathbb{R}^{d \times d},
\]
where $\mF\calR = (\mF\rho)_{\rho \in \calR}$.

The next lemma bounds solutions of the atomistic and continuum problems in terms of their values over the overlap region.

\begin{lemma}\label{lem:norm_equiv:reduce_to_Omega_o}
Suppose that  $\bm{w}^\a$ and $\bm{w}^\c$  are such that equations~\eqref{eq:wan_def} and~\eqref{eq:wcn_def} hold.
Then, there exists $R_{\rm core}^* > 0$ such that
\begin{align}
\|\nabla I\bm{w}^\a\|_{L^2(\Omega_\a)}
\lesssim~&
\|\nabla I \bm{w}^\a\|_{L^2(\Omega_\o)}
\qquad\text{and} \label{camel1}
\\
\|\nabla \bm{w}^\c\|_{L^2(\Omega_\c)}
\lesssim~&
\|\nabla \bm{w}^\c\|_{L^2(\Omega_\o)}\label{camel2},
\end{align}
for all domains $\Omega_\a, \Omega_\c$ and continuum meshes $\calT_h$ constructed according to the guidelines of Section~\ref{sec:approximation} with $R_{\rm core} \geq R_{\rm core}^*$.
\end{lemma}

\begin{proof}
%\commentdao{The next two paragraphs are two different variants of the same thing.  The first option is the literal negation of the statement leading to a sequence statement.  The second option goes directly to the sequences.  Please read these and let me know if you think the second is clear enough without the negation or if the first is more logical/clear.}

%Assume to the contrary that for all $R_{\rm core}^* > 0$, there exist domains $\Omega_\a, \Omega_\c$, a continuum mesh $\calT_h$ satisfying the requirements of Section~\ref{sec:approximation} such that $R_{\rm core} \geq R_{\rm core}^*$, and functions $\bm{w}^\a$ and $\bm{w}^\c$ satisfying~\eqref{eq:wan_def} and~\eqref{eq:wcn_def} but such that
%\begin{align*}
%\frac{\|\nabla I\bm{w}^\a\|_{L^2(\Omega_\a)}}{\|\nabla I \bm{w}^\a\|_{L^2(\Omega_\o)}} \geq~ n,
%\qquad
%\frac{\|\nabla \bm{w}^\c\|_{L^2(\Omega_\c)}}{\|\nabla \bm{w}^\c\|_{L^2(\Omega_\o)}} \geq~ n.
%\end{align*}
%In particular, by taking a sequence $R_{\rm core}^{*,n} \to \infty$, there exist sequences $R_{\rm core,n} \geq R_{\rm core}^{*,n}$, $R_{\c,n}$, $\Omega_{\a,n}, \Omega_{\c,n}$,  $\calT_{h,n}$, $\bm{w}^\c_n$ and $\bm{w}^\a_n$, such that $R_{\core,n}\to\infty$, $R_{\c,n} \to \infty$, $R_{\c,n}/R_{\rm core,n} = R_{\rm core,n}^\kappa \to \infty$ with
%\begin{align}
%\frac{\|\nabla I_n\bm{w}^\a_n\|_{L^2(\Omega_\a)}}{\|\nabla I_n \bm{w}^\a_n\|_{L^2(\Omega_{\o,n})}} \to~ \infty,
%\qquad
%\frac{\|\nabla \bm{w}^\c_n\|_{L^2(\Omega_\c)}}{\|\nabla \bm{w}^\c_n\|_{L^2(\Omega_{\o,n})}} \to~ \infty. \label{fish1}
%\end{align}

Assume that (\ref{camel1})--(\ref{camel2}) do not hold. Then, there exists a sequence $R_{\rm core}^{*,n} \to \infty$, with corresponding sequences $R_{\rm core,n} \geq R_{\rm core}^{*,n}$, $R_{\c,n}$, $\Omega_{\a,n}, \Omega_{\c,n}$,  $\calT_{h,n}$, $\bm{w}^\c_n$ and $\bm{w}^\a_n$, such that $R_{\core,n}\to\infty$, $R_{\c,n} \to \infty$, $R_{\c,n}/R_{\rm core,n} = R_{\rm core,n}^\kappa \to \infty$ with
\begin{align}
\frac{\|\nabla I_n\bm{w}^\a_n\|_{L^2(\Omega_\a,n)}}{\|\nabla I_n \bm{w}^\a_n\|_{L^2(\Omega_{\o,n})}} \to~ \infty,
\qquad
\frac{\|\nabla \bm{w}^\c_n\|_{L^2(\Omega_\c,n)}}{\|\nabla \bm{w}^\c_n\|_{L^2(\Omega_{\o,n})}} \to~ \infty. \label{fishy1}
\end{align}
After scaling the lattice, the domains, and the functions by $\epsilon_n := \frac{1}{R_{\rm core,n}}$ we find from (\ref{fishy1}) that
\begin{equation}\label{fish3}
\frac{\|\nabla I_n\tilde{\bm{w}}^\a_n\|_{L^2(\tilde{\Omega}_{\a})}}{\|\nabla I_n \tilde{\bm{w}}^\a_n\|_{L^2(\tilde{\Omega}_{\o})}} \to \infty.
\end{equation}
 Extend $I_n\tilde{\bm{w}}^\a_n|_{\tilde{\Omega}_\o}$ to $\mathbb{R}^d$ using the extension operator $R$ from Theorem~\ref{buren}.  Then we have
\[
\|\nabla (R(I_n\tilde{\bm{w}}^\a_n|_{\tilde{\Omega}_\o}))\|_{L^2{(\tilde{\Omega}_{\a})}} \leq~ C(\tilde{\Omega}_{\o}) \|\nabla I_n\tilde{\bm{w}}^\a_n\|_{L^2(\tilde{\Omega}_\o)}.
\]
Moreover, $R(I_n\tilde{\bm{w}}^\a_n|_{\tilde{\Omega}_\o})=I_n\tilde{\bm{w}}^\a_n$ on $\partial_\a\tilde{\mathcal{L}}_\a$. Let $S_{\a,n}$ be the Scott-Zhang interpolant operator from $H^1(\tilde{\Omega}_\a)$ to
\[
\left\{u \in {\rm C}(\tilde{\Omega}_\a) : u|_{\tau} \in \mathcal{P}_1(\tau) \quad \forall \tau \in \tilde{\mathcal{T}}_{\a,n}\right\}.
\]
Then $S_{\a,n} R(I_n\tilde{\bm{w}}^\a_n|_{\tilde{\Omega}_\o})$ defines an atomistic function in $\bm{\mathcal{U}}^\a_n$,  which is equal to $\tilde{\bm{w}}^\a_n$ on $\partial_\a \tilde{\calL}_{\a,n}$ since $R(I_n\tilde{\bm{w}}^\a_n|_{\tilde{\Omega}_\o})$ is piecewise linear on $\tilde{\Omega}_\o$ and due to the projection property of $S_{\a,n}$.  This implies that $\tilde{\bm{z}}^\a_n := S_{\a,n}R(I_n\tilde{\bm{w}}^\a_n|_{\tilde{\Omega}_\o})|_{\tilde{\Omega}_\a} - \tilde{\bm{w}}^\a_n \in \tilde{\bm{\mathcal{U}}}^\a_{0,n}$ and  that $\tilde{\bm{z}}^\a_n$ solves the problem
\[
\<\delta^2 \tilde{\mathcal{E}}^\a_n(\tilde{\bm{u}}^{\infty}_{\a,n}) \tilde{\bm{z}}^\a_n, \tilde{\bm{v}}^\a_n\> = \<\delta^2 \tilde{\mathcal{E}}^\a(\bm{u}^{\infty}_\a) S_{\a,n} R(I_n\tilde{\bm{w}}^\a_n|_{\tilde{\Omega}_\o})|_{\tilde{\Omega}_\a}, \tilde{\bm{v}}^\a_n\> \quad\forall \tilde{\bm{v}}^\a_n \in \tilde{\bm{\mathcal{U}}}^\a_{0,n}.
\]
Thus, taking $\tilde{\bm{v}}^\a_n = \tilde{\bm{z}}^\a_n$, using~\eqref{atLiftingStable}, and the stability of the Scott-Zhang interpolant (see {\bf P.3} or~\cite[Theorem $4.8.16$]{brenner2008}), we see that
$$
\|\nabla I_n\tilde{\bm{z}}^\a_n\|_{L^2(\tilde{\Omega}_\a)} \lesssim \|\nabla S_{\a,n} R(I_n\tilde{\bm{w}}^\a_n|_{\tilde{\Omega}_\o})|_{\tilde{\Omega}_\a} \|_{L^2(\tilde{\Omega}_\a)} \lesssim \|\nabla R(I_n\tilde{\bm{w}}^\a_n|_{\tilde{\Omega}_\o}) \|_{L^2(\tilde{\Omega}_\a)}
\leq C(\tilde{\Omega}_\o) \|\nabla I_n\tilde{\bm{w}}^\a_n\|_{L^2(\tilde{\Omega}_\o)}.
$$
This and the definition of $\bm{z}^\a_n$ imply
\[
\|\nabla S_{\a,n} R(I_n\tilde{\bm{w}}^\a_n|_{\tilde{\Omega}_\o})|_{\tilde{\Omega}_\a} - \nabla I_n\tilde{\bm{w}}^\a_n\|_{L^2(\tilde{\Omega}_\a)} \lesssim~  C(\tilde{\Omega}_\o) \|\nabla I_n\tilde{\bm{w}}^\a_n\|_{L^2(\tilde{\Omega}_\o)},
\]
which further leads to
\[
\|\nabla I_n\tilde{\bm{w}}^\a_n\|_{L^2(\tilde{\Omega}_\a)} \lesssim~ C(\tilde{\Omega}_\o) \|\nabla I_n\tilde{\bm{w}}^\a_n\|_{L^2(\tilde{\Omega}_\o)} + \|\nabla R(I_n\tilde{\bm{w}}^\a_n|_{\tilde{\Omega}_\o})\|_{L^2(\tilde{\Omega}_\a)} \leq~ 2C(\tilde{\Omega}_\o) \|\nabla I_n\tilde{\bm{w}}^\a_n\|_{L^2(\tilde{\Omega}_\o)},
\]
a contradiction to~\eqref{fish3}. This establishes~\eqref{camel1}.

A similar argument utilizing the Scott-Zhang interpolant on $\tilde{\Omega_\c}$ with mesh  $\tilde{\calT}_{h,n}$ yields~\eqref{camel2}.
\end{proof}

%\begin{remark}
%Note that the implied constants in the above proof only depend upon the domain $\Omega_\o$ through the Poincare constant and the constants involved in Stein's Extension Theorem.  We also require a shape regularity assumption on the domains $\Omega_\a$ and $\Omega_\c$ so that the constants $N,M$, and $\epsilon$ are invariant as the diameters of these domains increase.  In the proceeding, we will use a limiting process, but there will be a fixed domain $\tilde{\Omega}_\o$ in which these constants will come into play.  As the domain is fixed, these constants remain fixed.
%\end{remark}

Finally, we show that Theorem~\ref{th:norm_equiv} is a consequence of Theorem~\ref{lem:norm_equiv:desired_result}

\begin{proof}[Proof of Theorem~\ref{th:norm_equiv}]
According to Lemma~\ref{lem:norm_equiv:reduce_to_Omega_o}, if $\bm{w}^\a$ and $\bm{w}^\c$ satisfy equations~\eqref{eq:wan_def} and~\eqref{eq:wcn_def} then,
\[
\|\nabla (I \bm{w}^\a)\|_{L^2(\Omega_\a)}^2+\|\nabla \bm{w}^\c\|_{L^2(\Omega_\c)}^2 \lesssim~ \|\nabla (I \bm{w}^\a)\|_{L^2(\Omega_\o)}^2+\|\nabla \bm{w}^\c\|_{L^2(\Omega_\o)}^2.
\]
Consequently, to prove~\eqref{eq:norm_equiv_alt} in Theorem~\ref{th:norm_equiv} it suffices to show that
%\begin{equation}\label{reducedNorm}
$$
\|\nabla (I \bm{w}^\a)\|_{L^2(\Omega_\o)}^2+\|\nabla \bm{w}^\c\|_{L^2(\Omega_\o)}^2
\lesssim
\|\nabla(I \bm{w}^\a - \bm{w}^\c)\|_{L^2(\Omega_\o)}^2.
$$
%\end{equation}
This result is a direct consequence of Theorem~\ref{lem:norm_equiv:desired_result} since
%Lemma~\ref{lem:norm_equiv:reduce_to_Omega_o} immediately implies that inequality~\eqref{eq:norm_equiv_alt} of Theorem~\ref{th:norm_equiv} will be true provided
%\begin{equation}\label{reducedNorm}
%\|\nabla (I \bm{w}^\a)\|_{L^2(\Omega_\o)}^2+\|\nabla \bm{w}^\c\|_{L^2(\Omega_\o)}^2
%\lesssim
%\|\nabla(I \bm{w}^\a - \bm{w}^\c)\|_{L^2(\Omega_\o)}^2,
%\end{equation}
%for all $\bm{w}^\a$ and $\bm{w}^\c$ such that equations~\eqref{eq:wan_def} and~\eqref{eq:wcn_def} hold since lemma~\ref{lem:norm_equiv:reduce_to_Omega_o} shows
%\[
%\|\nabla (I \bm{w}^\a)\|_{L^2(\Omega_\a)}^2+\|\nabla \bm{w}^\c\|_{L^2(\Omega_\c)}^2 \lesssim~ \|\nabla (I \bm{w}^\a)\|_{L^2(\Omega_\o)}^2+\|\nabla \bm{w}^\c\|_{L^2(\Omega_\o)}^2.
%\]
\begin{align*}
&\|\nabla(I \bm{w}^\a - \bm{w}^\c)\|_{L^2(\Omega_\o)}^2
~=~ \|\nabla I \bm{w}^\a\|_{L^2(\Omega_\o)}^2 + \|\nabla \bm{w}^\c\|_{L^2(\Omega_\o)}^2 - 2 \left(\nabla I \bm{w}^\a,\nabla \bm{w}^\c \right)_{L^2(\Omega_\o)} \\
&~\geq~  \|\nabla I \bm{w}^\a\|_{L^2(\Omega_\o)}^2 + \|\nabla \bm{w}^\c\|_{L^2(\Omega_\o)}^2
- 2 c \|\nabla I \bm{w}^\a\|_{L^2(\Omega_\o)}\|\nabla \bm{w}^\c\|_{L^2(\Omega_\o)} \\
&~\geq~  \|\nabla I \bm{w}^\a\|_{L^2(\Omega_\o)}^2 + \|\nabla \bm{w}^\c\|_{L^2(\Omega_\o)}^2
- c\|\nabla I \bm{w}^\a\|_{L^2(\Omega_\o)}^2 -c\|\nabla \bm{w}^\c\|_{L^2(\Omega_\o)}^2 \\
&~=~ (1-c) \big(\|\nabla I \bm{w}^\a\|_{L^2(\Omega_\o)}^2 + \|\nabla \bm{w}^\c\|_{L^2(\Omega_\o)}^2\big).
\end{align*}
\end{proof}

It remains to prove Theorem~\ref{lem:norm_equiv:desired_result}, and for clarity we break the proof into several intermediate steps.

%%%%%%
%%%%%%
\subsection{Proof of Theorem~\ref{lem:norm_equiv:desired_result}}\label{proofSection}
%%%%%%
%%%%%%

%\begin{comment}
%For convenience, we restate Theorem~\ref{lem:norm_equiv:desired_result}:
%\begin{theorem}\label{prime}
%For $R_\a$ and $R_\c$ large enough with $\mathcal{T}_h$ a mesh satisfying minimum angle requirement,
%\begin{equation}\label{eq:norm_equiv:desired_result0}
%\inf_{\dot{w}^\a, \dot{w}^\c\ne 0} \frac{\left(\nabla I \dot{w}^\a, \nabla \dot{w}^\c\right)}{\|\nabla (I \dot{w}^\a)\|_{L^2(\Omega_\o)} \|\nabla \dot{w}^\c\|_{L^2(\Omega_\o)}} <1,
%\end{equation}
%for all $(\dot{w}^\a, \dot{w}^\c) \in \dot{\mathcal{U}}^\a \times \dot{\mathcal{U}}^\c_h$ such that
%\begin{align*}
%& \<\delta^2 \tilde{\mathcal{E}}^\a(\dot{u}^{\infty}_\a) \dot{w}^\a, \dot{v}^\a\> = 0 \quad\forall \dot{v}^\a \in \dot{\mathcal{U}}^\a_0,
%\\
%& \<\delta^2 \tilde{\mathcal{E}}^\c(\Pi_h \dot{u}^{\infty}) \dot{w}^\c, \dot{v}^\c\> = 0 \quad \forall \dot{v}^\c \in \dot{\mathcal{U}}^\c_{h,0}.
%\end{align*}
%\end{theorem}
%\end{comment}

%
The proof is by contradiction so we start with the following from which we aim to derive a contradiction.

\begin{statement}\label{false}
There exist sequences $R_{\rm core}^{*,n} \to \infty$, $R_{\core,n}\to\infty$, $R_{\c,n} \to \infty$, $R_{\c,n}/R_{\rm core,n} \to \infty$; a corresponding sequence of grids $\mathcal{T}_{h,n}$ with a minimum angle at least $\beta$; and corresponding sequences $\bm{w}^\c_n$, $\bm{w}^\a_n$ satisfying
\begin{align*}
& \<\delta^2 \tilde{\mathcal{E}}^\a(\bm{u}^{\infty}_\a) \bm{w}^\a, \bm{v}^\a\> = 0 \quad\forall \bm{v}^\a \in \bm{\mathcal{U}}^\a_0,
\\
& \<\delta^2 \tilde{\mathcal{E}}^\c( \bm{u}^{\rm con}) \bm{w}^\c, \bm{v}^\c\> = 0 \quad \forall \bm{v}^\c \in \bm{\mathcal{U}}^\c_{h,0},
\end{align*}
such that
\begin{equation}\label{contradiction}
\frac{\left(\nabla I \bm{w}^\a_n, \nabla \bm{w}^\c_n\right)}{\|\nabla (I \bm{w}^\a_n)\|_{L^2(\Omega_\o)} \|\nabla \bm{w}^\c_n\|_{L^2(\Omega_\o)}} \to 1.
\end{equation}
\end{statement}

We will show~\eqref{contradiction} yields a contradiction in four steps.  In the first step, we will again scale the lattice by $\eps_n=1/R_{\core,n}$ to define sequences of functions $\tilde{\bm{w}}^\a_n$ having a common domain of definition and $\tilde{\bm{w}}^\c_n$ having a common domain of definition.  This will allow us to extract weak limits of these sequences.  The second step will show these limits satisfy the homogeneous Cauchy-Born equation. In the third step, we show weak convergence, combined with satisfying atomistic and finite element equations, implies the limit and inner product commute. This will yield a contradiction in the final, fourth step of the proof.

\subsubsection*{Step 1:}
Recall that we use the tilde accent for objects on the scaled domains. Let $I_n$ be the piecewise interpolant onto the lattice $\tilde{\calL}_n$, and normalize $\tilde{\bm{w}}^\a_n$ and $\tilde{\bm{w}}^\c_n$ to functions $\bar{\bm{w}}^\a_n$ and $\bar{\bm{w}}^\c_n$ such that
\begin{equation*}\label{gradProperty}
\|\nabla (I_n \bar{\bm{w}}^\a_n)\|_{L^2(\tilde{\Omega}_\o)} = 1, \quad \mbox{and} \quad \|\nabla \bar{\bm{w}}^\c_n\|_{L^2(\tilde{\Omega}_\o)} = 1.
\end{equation*}
Due to this property and our hypothesis~\eqref{contradiction}, we have that
\begin{equation}\label{vergeResult}
\left(\nabla I_n\bar{\bm{w}}^\a_n,\nabla \bar{\bm{w}}^\c_n\right)_{L^2\left(\tilde{\Omega}_{\o}\right)} \to 1.
\end{equation}

Moreover, $\nabla I_n\bar{\bm{w}}^\a_n$ is a bounded sequence in $L^2(\tilde{\Omega}_\a)$ since
\[
\|\nabla I_n\bar{\bm{w}}^\a_n\|_{L^2(\tilde{\Omega}_\a)} = \|\nabla I_n\tilde{\bm{w}}^\a_n\|_{L^2(\tilde{\Omega}_\a)}/\|\nabla I_n\tilde{\bm{w}}^\a_n\|_{L^2(\tilde{\Omega}_\o)} \lesssim \|\nabla I_n\tilde{\bm{w}}^\a_n\|_{L^2(\tilde{\Omega}_\o)}/\|\nabla I_n\tilde{\bm{w}}^\a_n\|_{L^2(\tilde{\Omega}_\o)} = 1,
\]
after using a scaled version of Lemma~\ref{lem:norm_equiv:reduce_to_Omega_o}.  Similarly, $\nabla \bar{\bm{w}}^\c_n$ is bounded in $L^2(\tilde{\Omega}_\c)$.  Meanwhile, $\bar{\bm{w}}^\a_n$ and $\bar{\bm{w}}^\c_n$ will still satisfy the variational equalities~\eqref{eq:wan_def_scale} and~\eqref{eq:wcn_def_scale} by linearity.

For each $n$, we let $I_n\bar{w}^\a_n$ (without boldface) be the element in the equivalence class of $\bar{\bm{w}}^\a_n$ with mean value $0$ over $\tilde{\Omega}_\a$.
The resulting sequence is bounded in $H^1(\tilde{\Omega}_\a)$ and so it has a weakly convergent subsequence, which we denote again by $I_n\bar{w}_n^\a$. Let $\bar{w}^\a_0 \in H^1(\tilde{\Omega}_\a)$ be the weak limit. By the compactness of the embedding $H^1(\tilde{\Omega}_\a)\subset L^2(\tilde{\Omega}_\a)$ it follows that $I_n\bar{w}^\a_n\rightarrow \bar{w}^\a_0$ in $L^2(\tilde{\Omega}_\a)$.
Similarly, the functions $\bar{\bm{w}}_n^\c$ form a bounded sequence on the Hilbert space (cf.~\cite{suli2012}),
\begin{equation*}\label{ortSulii}
\bm{H}^1(\tilde{\Omega}_{\c}) := \left\{u^\c \in H^1_{\rm loc}(\tilde{\Omega}_\c) : \nabla u^\c \in L^2(\tilde{\Omega}_\c)\right\} / \mathbb{R}^d.
\end{equation*}
Thus, we can extract a weakly convergent subsequence, still denoted by $\bar{\bm{w}}_n^\c$, with limit $\bar{\bm{w}}^\c_0 \in \bm{H}^1(\tilde{\Omega}_\c)$, i.e, $\bar{\bm{w}}_n^\c \weakto \bar{\bm{w}}^\c_0$ in $\bm{H}^1(\tilde{\Omega}_\c)$.

Let $\bar{w}^\c_n$ and $\bar{w}^\c_0$ (without boldface) be equivalence class elements having zero mean over $\tilde{\Omega}_{\o, \ex}$.  Then $\bar{w}^\c_n$ is bounded in $H^1(\tilde{\Omega}_{\o, \ex})$ and converges weakly to some $\bar{w}^\c\in H^1(\tilde{\Omega}_{\o, \ex})$. But since $\bar{\bm{w}}_n^\c \weakto \bar{\bm{w}}^\c_0$ in $\bm{H}^1(\tilde{\Omega}_\c)$ we must have $\nabla \bar{w}^\c = \nabla \bar{w}^\c_0$ on $\tilde{\Omega}_{\o, \ex}$ so the two functions differ almost everywhere by a constant on $\tilde{\Omega}_{\o, \ex}$.  Since both $\bar{w}^\c_0$ and $\bar{w}^\c$ have mean value $0$ over $\tilde{\Omega}_{\o, \ex}$, the two functions are in fact equal on $\tilde{\Omega}_{\o, \ex}$.  Thus $\bar{w}^\c_n$ converges weakly to $\bar{w}^\c_0$ in $H^1(\tilde{\Omega}_{\o, \ex})$. The strong convergence $\bar{w}^\c_n\rightarrow\bar{w}^\c_0$ in $L^2(\tilde{\Omega}_{\o, \ex})$ then follows from the compactness of the embedding $H^1(\tilde{\Omega}_{\o, \ex})\hookrightarrow L^2(\tilde{\Omega}_{\o, \ex})$.

In summary, we have established the following result.
\begin{lemma}\label{convergeLemma}
There exist sequences $\bar{w}^\a_n \in H^1(\tilde{\Omega}_\a)$ and $\bar{w}^\c_n \in L^2_{\rm loc}(\tilde{\Omega}_\c)$ and with $\nabla \bar{w}^\c_n \in L^2(\tilde{\Omega}_\c)$ which satisfy the variational equalities~\eqref{eq:wan_def_scale} and~\eqref{eq:wcn_def_scale} and functions $\bar{w}^\a_0 \in H^1(\tilde{\Omega}_\a)$ and $\bar{w}^\c_0 \in \bm{H}^1(\tilde{\Omega}_\c)$ such that
\begin{align}
&I_n\bar{w}^\a_n \weakto~ \bar{w}^\a_0 \quad \mbox{in} \quad H^1(\tilde{\Omega}_\a), \qquad I_n\bar{w}^\a_n \to~ \bar{w}^\a_0 \quad \mbox{in} \quad L^2(\tilde{\Omega}_\a), \label{varyAt}  \\
&\bar{w}^\c_n \weakto~ \bar{w}^\c_0 \quad \mbox{in} \quad H^1(\tilde{\Omega}_{\o, \ex}), \qquad \bar{w}^\c_n \to~ \bar{w}^\c_0 \quad \mbox{in} \quad L^2(\tilde{\Omega}_{\o, \ex}) \label{varyCont}.
\end{align}
\end{lemma}

\subsubsection*{Step 2:}
\begin{theorem}\label{thm:norm_equiv_limiting_equation1}
The functions $\bar{w}^\a_0$ and $\bar{\bm{w}}^\c_0$ satisfy the weak, linear homogeneous Cauchy-Born elasticity equations
\begin{align}
\int_{\tilde{\Omega}_\a} (\bbC : \nabla \bar{w}^\a_0) : \nabla v  =~& 0 \quad \forall v \in H^{1}_0(\tilde{\Omega}_\a) \label{eq:lim_equation_both1}, \\
\int_{\tilde{\Omega}_\c} (\bbC : \nabla \bar{\bm{w}}^\c_0) : \nabla v =~& 0 \quad \forall v \in H^{1}_0(\tilde{\Omega}_\c) \label{eq:lim_equation_both2}.
\end{align}
\end{theorem}

We break the proof into several lemmas.  We start with the atomistic case~\eqref{eq:lim_equation_both1} where special care must be exercised near the defect at the origin.
\begin{lemma}\label{thm:norm_equiv_limiting_equation}
Let $\tilde{N}$ be any neighborhood of the origin with $\tilde{N} \subset \tilde{\Omega}_\a$ and set $\tilde{\Omega}^{\prime} := \tilde{\Omega}_\a\backslash\tilde{N}$.  Then
$\bar{w}^\a_0$ satisfies
\begin{equation}\label{eq:lim_equation}
\int_{\tilde{\Omega}^{\prime}} (\bbC : \nabla \bar{w}^\a_0) : \nabla v = 0
\quad \forall v \in H^1_0\big(\tilde{\Omega}^{\prime}\big).
\end{equation}
\end{lemma}

The key result in proving Lemma~\ref{thm:norm_equiv_limiting_equation} is the auxiliary Lemma~\ref{lem:norm_equiv:main_aux}.  In the proof, we use the standard notation $\subsubset$ to denote compact subsets.

\begin{lemma}\label{lem:norm_equiv:main_aux}
Let $U$ be a bounded domain in $\mathbb{R}^d$ whose boundary is Lipschitz and a union of edges of $\calT_\a$.  Take a domain $U_1 \subsubset U$, and suppose $v_n$ is piecewise linear with respect to $\tilde{\calL}_n = \epsilon_n\calL$ and $v_n \weakto  v_0$ in $H^1(U)$ for some $v_0 \in H^1(U)$.
Then for $r \in \calR$, $\bar{I}_{n} D_{\eps_{n} r} v_{n} \weakto \nabla_r v_0$ in $L^2(U_1)$.
\end{lemma}

\begin{proof}[Proof of Lemma~\ref{lem:norm_equiv:main_aux}]
We prove the lemma for $v_0=0$ and then reduce the case $v_0 \neq 0$ to this setting.

{\it Case 1 ($v_0=0$).}
Take $\varphi \in {\rm C}_0^{\infty}(U_1)$, and note since $v_n \weakto 0$ in $H^1(U)$, $ v_n \to 0$ strongly in $L^2(U)$.  For $n$ large enough, we may choose $\tilde{\calL}_{n,1} \subset \tilde{\calL}_n$ such that  $U_1 \subset \bigcup_{\xi \in \tilde{\calL}_{n,1}}\varsigma_\xi \subset U$.  Applying Taylor's Theorem with the notation ${\rm conv}(\xi,x)$ representing the convex hull of $\xi$ and $x$ produces
\begin{equation}
\begin{array}{l} \displaystyle
\limsup_{n\to\infty} \big|(\bar{I}_n D_{\eps_n r} v_n, \varphi)_{L^2(U_1)}\big|   \\ \displaystyle
\quad=~ \limsup_{n\to\infty}\bigg|\int_{U_1}\bar{I}_n D_{\eps_n r} v_n(x) \varphi(x)\, dx\bigg|
=~ \limsup_{n\to\infty}\bigg|\sum_{\xi \in \tilde{\calL}_{n,1}}\int_{\varsigma_\xi \cap U_1}\bar{I}_n D_{\eps_n r} v_n(x) \varphi(x)\, dx\bigg|   \\[1ex] \displaystyle
\quad=~ \limsup_{n\to\infty}\bigg|\sum_{\xi \in \tilde{\calL}_{n,1}}\int_{\varsigma_\xi \cap U_1} D_{\eps_n r} v_n(\xi) (\varphi(\xi) + \nabla \varphi(\tau_{\xi,x})(x-\xi))\, dx\bigg| \quad \mbox{for some $\tau_{\xi,x} \in {\rm conv}(\xi,x)$}
\\[1ex] \displaystyle \quad
	\leq~ \limsup_{n\to\infty}\bigg|\underbrace{\sum_{\xi \in \tilde{\calL}_{n,1}}\int_{\varsigma_\xi \cap U_1} D_{\eps_n r} v_n(\xi) \varphi(\xi)\, dx}_{T_1}\bigg|
	+
	\limsup_{n\to\infty} \bigg| \underbrace{\sum_{\xi \in \tilde{\calL}_{n,1}}\int_{\varsigma_\xi \cap U_1} D_{\eps_n r} v_n(\xi) \nabla \varphi(\tau_{\xi,x})(x-\xi)\, dx}_{T_2}\bigg| .
\end{array}\label{duck1}
\end{equation}

Since we are taking limits, we assume throughout that $\epsilon_n < {\rm dist}(U_1, \partial U)$ so that the expressions above are well defined. We first estimate $T_2$ by bounding $|x-\xi| \leq \eps_n$ and $|\varphi(\tau_{\xi,x})| \leq \|\nabla \varphi\|_{L^\infty(U_1)} \lesssim 1$:
\[
%\begin{array}{l}
%\displaystyle
T_2
	\lesssim \eps_n \sum_{\xi \in \tilde{\calL}_{n,1}}\int_{\varsigma_\xi \cap U_1} |D_{\eps_n r} v_n(\xi)| dx
%\\[1ex]\displaystyle \quad
	= \eps_n \sum_{\xi \in \tilde{\calL}_{n,1}} |D_{\eps_n r} v_n(\xi)| \vol(\varsigma_\xi\cap U_1)
%\\[1ex]\displaystyle \quad
	\leq \eps_n |r| \, \|\nabla v_n\|_{L^1(U)}
%\\[1ex]\displaystyle \quad
	\lesssim \eps_n \|\nabla v_n\|_{L^2(U)}
.
%\end{array}
\]
Note that here the bound $\sum_{\xi \in \tilde{\calL}_{n,1}} |D_{\eps_n r} v_n(\xi)| \vol(\varsigma_\xi\cap U_1) \leq |r| \, \|\nabla v_n\|_{L^1(U)}$ easily follows from a local bound $|D_{\eps_n r} v_n(\xi)| \leq \int_0^1 |\nabla_r v_n(\xi+\eps_n r t)| \, dt$ for sufficiently small $\eps_n$.
Since $\|\nabla v_n\|_{L^2(U)}$ are bounded (as a consequence of $v_n \weakto  v_0$ in $H^1$), we have that $T_2 \lesssim \eps_n \to 0$.

To estimate $T_1$, we shift the finite difference operator onto $\varphi(\xi)\vol\left(\varsigma_\xi \cap U_1\right)$, use the product rule for difference quotients (see~\eqref{shift}), and recall that $\varphi \in {\rm C}^\infty_0(U_1)$:
\begin{align}
T_1 &~ = \sum_{\xi \in \tilde{\calL}_{n,1}} D_{\eps_n r} v_n(\xi) \varphi(\xi)\vol\left(\varsigma_\xi \cap U_1\right)
= -\sum_{\xi \in \tilde{\calL}_{n,1}}  v_n(\xi) D_{-\eps_n r}(\varphi(\xi)\vol\left(\varsigma_\xi \cap U_1\right)) \nonumber \\
&~= -\sum_{\xi \in \tilde{\calL}_{n,1}}  v_n(\xi) (D_{-\eps_n r}(\varphi(\xi)){\vol}\left(\varsigma_\xi \cap U_1\right) + \varphi(\xi - \epsilon_nr) D_{-\eps_n r}\vol\left(\varsigma_\xi \cap U_1\right)) \nonumber  \\
&~= -\sum_{\xi \in \tilde{\calL}_{n,1}}  v_n(\xi) D_{-\eps_n r}(\varphi(\xi)){\vol}\left(\varsigma_\xi \cap U_1\right) \nonumber \\
&~\leq \Big(\sum_{\xi \in \tilde{\calL}_{n,1}} |v_n(\xi)|^2{\vol}\left(\varsigma_\xi \cap U_1\right)\Big)^{1/2}\Big( \sum_{\xi \in \tilde{\calL}_{n,1}}|D_{-\eps_n r}\varphi(\xi)|^2{\vol}\left(\varsigma_\xi \cap U_1\right)\Big)^{1/2} \nonumber \\[2ex]
&~\lesssim \|\bar{I}_nv_n\|_{L^2(U_1)} \|\nabla I_n \varphi\|_{L^2(U)}
\lesssim \|\bar{I}_nv_n\|_{L^2(U_1)}
,
\label{halfDuck}
\end{align}
where in the last step we used that the smoothness of $\varphi$ implies that $\|\nabla I_n \varphi\|_{L^2(U)}$ converges to $\|\nabla \varphi\|_{L^2(U)} \lesssim 1$.

We now wish to bound $\|\bar{I}_nv_n\|_{L^2(U_1)}$ by $\|v_n\|_{L^2(U)}$. Consider the cell $\varsigma_\xi$ and take $T$ to be a micro-simplex of $\tilde{\calT}_{\a,n} = \epsilon_n\calT_\a$ such that $\xi$ is a vertex of $T$ and $T \subset\varsigma_\xi$.  Further let $\mathcal{N}(T)$ be the nodes of $T$ and let $\hat{T}$ be a reference simplex with nodes $\mathcal{N}(\hat{T})$. If
$\hat{f}$ is the pullback of a function $f$ on $T$, then
\begin{align*}
\|\bar{I}_nv_n\|_{L^2(\varsigma_\xi)}= \epsilon_n^{d/2}\cdot|v_n(\xi)| \lesssim~ |T|^{1/2} \sup_{\zeta \in \mathcal{N}(T)} |v_n(\zeta)|
=|T|^{1/2} \sup_{\hat{\zeta}\in\mathcal{N}(\hat{T})} |\hat{v}_n(\hat{\zeta})| \lesssim~  |T|^{1/2}\|\hat{v}_n\|_{L^2(\hat{T})}
\lesssim~ \|v_n\|_{L^2(T)}.
\end{align*}
Summing over all $\xi \in \tilde{\calL}_{n,1}$ gives
\[
\|\bar{I}_nv_n\|_{L^2(U_1)} \leq \|v_n\|_{L^2(U)},
\]
Because $v_n$ converges weakly to $0$ in $H^1(U)$, $v_n$ converges strongly to $0$ in $L^2(U)$.
This shows that $T_1\to 0$ which, together with $T_2\to0$, yields
\[
\limsup_{n\to\infty} \big|(\bar{I}_n D_{\eps_n r} v_n, \varphi)_{L^2(U_1)}\big| = 0.
\]
We can use similar computations to those in our estimate of $T_2$, in particular, the local bound $|D_{\eps_n r} v_n(\xi)|^2 \leq \int_0^1 |\nabla_r v_n(\xi+\eps_n r t)|^2 dt$, to conclude that $\|\bar{I}_n D_{\eps_n r} v_n\|_{L^2(U_1)} \lesssim \|v_n\|_{L^2(U)}$ so that boundedness of $\bar{I}_n D_{\eps_n r} v_n$ and density of smooth functions in $L^2(U)$ imply $\bar{I}_n D_{\eps_n r} v_n$ converges weakly to $0$.

{\it Case 2 ($v_0 \ne 0$).}
We reduce this case to the previous one by using a diagonalizing argument to find a sequence of piecewise linear comparison functions which converge weakly to $v_0$ and then applying the previous case to the difference of the comparison sequence and original sequence.

The hypotheses on $U$ imply ${\rm C}^\infty(\bar{U})$ is dense in $H^1(U)$ so we may take $v_{0,j} \in {\rm C}^\infty(\bar{U})$ such that
%Let $\eta_R$ be a standard mollifier on a ball of radius $R$, and define
%\begin{align*}
%v_{0,R}(x) :=~& (\eta_R\ast v_0)(x) = \int_U \eta_R(x-y) v_0(y)\, dy,
%\end{align*}
%for $x$ in $U^R := \left\{x \in U : {\rm{dist}}(x, \partial U) > R\right\}$.  From standard properties of mollifiers, it follows that
\begin{equation}\label{al3}
\| v_{0,j} - v_0\|_{H^1(U)} \leq 1/j.
\end{equation}
Since $v_{0,j}$ is smooth, for any fixed $j$, $I_n v_{0,j} \to v_{0,j}$ in $H^1(U)$. Similarly, $D_{\epsilon_n r}v_{0,j} \to \nabla_r v_{0,j}$ uniformly in $x \in U_1$ as $\epsilon_n \to 0$, and hence $D_{\epsilon_n r}v_{0,j} \to \nabla_r v_{0,j}$ in $L^2(U_1)$.  Furthermore,
\begin{align*}
&\|\bar{I}_n D_{\epsilon_n r}v_{0,j} - D_{\epsilon_n r}v_{0,j}\|_{L^2(U_1)}^2
= \int_{U_1} |\bar{I}_n D_{\epsilon_n r}v_{0,j} - D_{\epsilon_n r}v_{0,j}|^2\, dx \\
&~= \sum_{\xi \in \tilde{\calL}_{n,1}} \int_{\varsigma_\xi \cap U_1} | D_{\epsilon_n r}v_{0,j}(\xi) - D_{\epsilon_n r}v_{0,j}(x)|^2\, dx \\
&~= \sum_{\xi \in \tilde{\calL}_{n,1}} \int_{\varsigma_\xi \cap U_1} | D_{\epsilon_n r}\nabla v_{0,j}(\tau_{\xi,x})(\xi-x) |^2\, dx \quad \mbox{for some $\tau_{\xi,x} \in {\rm conv}(\xi,x)$}  \\
&~\lesssim~    \epsilon_n^2\sum_{\xi \in \tilde{\calL}_{n,1}} \int_{\varsigma_\xi \cap U_1} |D_{\epsilon_n r}\nabla v_{0,j}(\tau_{\xi,x})|^2\, dx
\lesssim~ \epsilon_n^2\|\nabla^2 v_{0,j}\|_{L^2(U)}^2 \to 0 \quad \mbox{as $n \to \infty$.}
\end{align*}
Thus, as $n \to \infty$, we have that
\begin{equation}\label{al1}
\begin{split}
\|\bar{I}_n D_{\eps_n r} v_{0,j} - \nabla_r v_{0,j}\|_{L^2(U_1)} \leq \| \bar{I}_n D_{\eps_n r} v_{0,j} - D_{\eps_n r} v_{0,j}\|_{L^2(U_1)}
+ \|D_{\eps_n r} v_{0,j} - \nabla_r v_{0,j}\|_{L^2(U_1)} \to 0.
\end{split}
\end{equation}

This and $I_n v_{0,j} \to v_{0,j}$ as $n \to \infty$ in $H^1(U)$ imply that for any $j$ there exists $N_j$ (which can be chosen such that $N_j$ strictly increases to infinity as $j$ goes to $\infty$) such that
\begin{align}
\| I_n v_{0,j} -  v_{0,j}\|_{H^1(U)} \leq~& 1/j \quad \forall n \geq~ N_{j},  \label{bumble1} \\
\|\bar{I}_n D_{\eps_n r} v_{0,j} - \nabla_r v_{0,j}\|_{L^2(U_1)} \leq~& 1/j \quad \forall n \geq~ N_{j} \label{bumble2}.
\end{align}
Hence we choose a sequence $J_n$ by letting $J_n:=j$ whenever $N_j \leq n < N_{j+1}$ (and $J_n = 1$ for $n<N_1$).
It is easy to see that $J_n\to\infty$ as $n\to\infty$, hence
equations~\eqref{al3},~\eqref{bumble1}, and~\eqref{bumble2} give
\begin{align} \label{saw1}
\| I_n v_{0,J_n} -  v_{0}\|_{H^1(U)} \leq~& \|I_n v_{0,J_n} - v_{0,J_n}  \|_{H^1(U)} + \|v_{0,J_n} - v_0 \|_{H^1(U)} \leq 2/J_n \to 0, \\ \label{saw2}
\|\bar{I}_{n} D_{\eps_{n} r} v_{0,J_n} - \nabla_r v_{0}\|_{L^2(U_1)} \leq~&
\|\bar{I}_{n} D_{\eps_{n} r} v_{0,J_n} -  \nabla_r v_{0,J_n}\|_{L^2(U_1)}
 + \| \nabla_r v_{0,J_n}- \nabla_r v_{0}\|_{L^2(U_1)}
\lesssim~ 2/J_n \to 0.
\end{align}
The functions $\hat{v}_{n} := I_{n}v_{0,J_n}$ will serve as our comparison functions.  Observe $v_{n} - \hat{v}_{n}$ converges weakly to zero in $H^1(U)$ by~\eqref{saw1} and our hypothesis that $v_n$ converges weakly to $v_0$.  Case 1 then implies
\begin{equation}\label{saw3}
\bar{I}_{n} D_{\eps_{n} r}v_{n} - \bar{I}_{n} D_{\eps_{n} r} \hat{v}_{n} \weakto 0 \quad \mbox{in} \quad L^2(U_1).
\end{equation}
But a straightforward calculation shows
\[
\bar{I}_{n} D_{\eps_{n} r} \hat{v}_{n} = \bar{I}_{n} D_{\eps_{n} r} I_{n}v_{0,J_n} = \bar{I}_{n} D_{\eps_{n} r}v_{0,J_n},
\]
and~\eqref{saw2} states that $ \bar{I}_{n} D_{\eps_{n} r} v_{0,J_n}$ converges strongly, whence weakly, to $\nabla_r v_{0}$ in $L^2(U_1)$. This, along with~\eqref{saw3}, means
\[
\bar{I}_{n} D_{\eps_{n} r}v_{n}  \weakto \nabla_r v_{0} \quad \mbox{in} \quad L^2(U_1).
\]
\end{proof}

\begin{remark}\label{afterLem37}
With only minor modifications to the proof, the statement of the theorem remains true if weak convergence is replaced with strong convergence.  For the $v_0 = 0$ case, one only needs to replace $\varphi$ with $\bar{I}_{n} D_{\eps_{n} r}v_{n}$ and carry out simplified computations while the $v_0 \neq 0$ case can then be proven almost verbatim by replacing weak convergence with strong convergence.
\end{remark}

\begin{proof}[Proof of Lemma~\ref{thm:norm_equiv_limiting_equation}]
First, notice that it is enough to test \eqref{eq:lim_equation} with $v\in \C^{\infty}_0(\tilde{\Omega}_\a\setminus\tilde{N})$, i.e., for $\supp(v) \subsubset \tilde{\Omega}_\a$, $0\notin \supp(v)$. Take a domain $\Omega_1$ such that ${\rm supp}(v) \subset \Omega_1 \subsubset \tilde{\Omega}_\a$. Because $I_n\bar{w}^\a_n \weakto \bar{w}^\a_0$ on $H^1(\tilde{\Omega}_\a)$ by~\eqref{varyAt}, Lemma~\ref{lem:norm_equiv:main_aux} implies
\begin{equation}\label{tail1}
\bar{I}_n D_{\epsilon_n r} \bar{w}^\a_n \weakto \nabla_r \bar{w}^\a_0 \quad  \mbox{in $L^2(\Omega_1)$} \quad \mbox{for all} \quad r \in \calR.
\end{equation}

Since $v$ has compact support inside $\tilde{\Omega}_\a\setminus\tilde{N}$, $D_{\epsilon_n \rho}v(\xi)$ vanishes on $\tilde{\calL}_{\a,n}\backslash\tilde{\calL}_{\a,n}^{\circ\circ}$ for all $n$ large enough and $\rho \in \calR$.  We may therefore rewrite~\eqref{eq:wan_def_scale} with $\bar{w}^\a_n$ using the integral formulation introduced in~\eqref{intForm}
\begin{equation}\label{eq:wan_def_alt}
\begin{split}
0 = \int_{\tilde{\Omega}_{\a}} \bar{I}_n  V_\xi''( D_{\epsilon_n}\tilde{\bm{u}}^{\infty}_{\a,n}) \!:\! \bar{I}_nD_{\epsilon_n}\bar{w}^\a_n : \bar{I}_n D_{\epsilon_n}{v} \, dx.
\end{split}
\end{equation}

Because $v$ is smooth, a calculation analogous to~\eqref{al1} implies
\begin{equation}\label{tail2}
\bar{I}_n D_{\epsilon_n r}v \to \nabla_r  v \quad \mbox{in $L^2(\Omega_1)$ for all $r \in \calR$.}
\end{equation}

According to estimate~\eqref{decayEquation_no_tilde} of Theorem~\ref{decayThm}, the local minimum, $\bm{u}^{\infty}$, of $\calE^\a$ satisfies
\begin{equation*}
\left|\nabla I\bm{u}^{\infty}(x)\right| \lesssim~ |x|^{-d} \quad \mbox{for} \, x \notin \Omega_{\rm core}.
\end{equation*}
After scaling the lattice by $\epsilon_n$ we get a sequence of global solutions $\tilde{\bm{u}}^{\infty}_n(\xi) = \epsilon_n\bm{u}^{\infty}(\xi/\epsilon_n)$ for $\xi \in \tilde{\mathcal{L}}_n$.
Thus, for $x \neq 0$ and large enough $n$ there holds  $x \notin \epsilon_n\Omega_{\rm core} = \tilde{\Omega}_{{\rm core},n}$. Since $d > 1$ it follows that
\begin{equation*}\label{converge0}
\left|\nabla (I_n\tilde{\bm{u}}^{\infty}_n(x))\right| =~ \left|(\nabla I_n \bm{u}^{\infty}_n)(x/\epsilon_n)\right|
\lesssim \left|x/\epsilon_n\right|^{-d}
= \epsilon_n^{d} \left|x\right|^{-d} \to 0
\end{equation*}
uniformly in $x$ as $\epsilon_n \to 0$. This also implies
\[
| \bar{I}_n D_{\epsilon_n} \tilde{\bm{u}}^{\infty}_{\a,n}(x)| \to 0 \quad \quad \mbox{uniformly as $\epsilon_n \to 0$ on $\tilde{\Omega}_\a \backslash \tilde{N}$;}
\]
whence
\[
\bar{I}_n V''(D_{\epsilon_n} \tilde{\bm{u}}^{\infty}_{\a,n}(x)) = V''(\bar{I}_n D_{\epsilon_n}\tilde{\bm{u}}^{\infty}_{\a,n}(x)) \to V''(\bm{0})   \quad \mbox{uniformly as $\epsilon_n \to 0$ on $\tilde{\Omega}_\a \backslash \tilde{N}$}.
\]

Hence, taking the limit of \eqref{eq:wan_def_alt}, and using~\eqref{tail1},~\eqref{tail2}, and the fact that the ``dual pairing'' ($:$) of a weakly convergent and a strongly convergent sequence converges to the dual pairing of the limits, we obtain
\begin{align*}
0 =~& \lim_{n \to \infty} \int_{\tilde{\Omega}_\a} \bar{I}_n V''(D_{\epsilon_n} \tilde{\bm{u}}^{\infty}_\a) \!:\!  \bar{I}_nD_{\epsilon_n}\bar{w}^\a_n : \bar{I}_n D_{\epsilon_n} v \, dx \\
=~& \lim_{n \to \infty} \int_{\tilde{\Omega}_\a} \bar{I}_n V''(D_{\epsilon_n} \tilde{\bm{u}}^{\infty}_\a) \!:\!  \bar{I}_n D_{\epsilon_n} v: \bar{I}_nD_{\epsilon_n}\bar{w}^\a_n   \, dx
\\=~&
\int_{\tilde{\Omega}_\a} V''(\bm{0}) \!:\! \nabla_\calR \bar{w}^\a_0 : \nabla_\calR v \, dx
=
\int_{\tilde{\Omega}_\a} \bbC \!:\! \nabla \bar{w}^\a_0 : \nabla v \, dx,
\end{align*}
where $\nabla_\calR u = (\nabla u \, \rho)_{\rho \in \calR}$. \end{proof}

\begin{proof}[Proof of Theorem~\ref{thm:norm_equiv_limiting_equation1}]
We first prove \eqref{eq:lim_equation_both1}, followed by \eqref{eq:lim_equation_both2}.

\noindent
{\it Proof of \eqref{eq:lim_equation_both1}.}
By density, it suffices to prove the theorem for $v \in {\rm C}^\infty_0(\tilde{\Omega}_\a)$.
Let $\eta$ be a standard mollifier on a unit ball with $\eta_R(x) = \frac{1}{R^d}\eta(x/R)$ its extension to a ball of radius $R$. Let
\[
\chi_R = \begin{cases}
&1 \quad \mbox{if} \quad |x| < 2R, \\
&0 \quad \mbox{if} \quad |x| \geq 2R,
\end{cases}
\]
and define the smooth bump function
\[
\varphi_R(x) := (\eta_R * \chi_R)(x).
\]
Recall that $\varphi_R(x)$ is of class ${\rm C}^\infty$ and satisfies
$$
0 \leq \varphi_R(x) \leq 1,
\quad\mbox{and}\quad
\left\{
\begin{array}{rl}
\varphi_R(x) = 1 &~\mbox{for} \quad |x| < R, \\[1ex]
\varphi_R(x) = 0 &~  \mbox{for} \quad |x| \geq 3R.
\end{array}
\right.
$$
Thus, $v - \varphi_R v$ is smooth and vanishes on $B_R(0)$.  By Theorem~\ref{thm:norm_equiv_limiting_equation},
\begin{align*}
0 =~& \int_{\tilde{\Omega}_\a\backslash B_R(0)} \bbC : \nabla \bar{w}^\a_0 : \nabla (v - \varphi_R v) \, dx
= \int_{\tilde{\Omega}_\a} \bbC : \nabla \bar{w}^\a_0 : \nabla (v - \varphi_R v)\, dx \\
=~& \int_{\tilde{\Omega}_\a} \bbC : \nabla \bar{w}^\a_0 : \nabla v\, dx - \int_{\tilde{\Omega}_\a} \bbC :\nabla \bar{w}^\a_0 : \nabla (\varphi_R v)\, dx
= \int_{\tilde{\Omega}_\a} \bbC : \nabla \bar{w}^\a_0 : \nabla v\, dx - \int_{B_{3R}(0)} \bbC: \nabla \bar{w}^\a_0 : \nabla (\varphi_R v)\, dx.
\end{align*}
This implies
\begin{equation}\label{rockNeg}
\int_{\tilde{\Omega}_\a} \bbC : \nabla \bar{w}^\a_0 : \nabla v\, dx = \int_{B_{3R}(0)} \bbC : \nabla \bar{w}^\a_0 : \nabla (\varphi_R v)\, dx.
\end{equation}
Also note
\begin{equation}\label{rock0}
\Big|\int_{B_{3R}(0)} \bbC : \nabla \bar{w}^\a_0 : \nabla (\varphi_R v)\, dx \Big| \leq~ \|\bbC : \nabla \bar{w}^\a_0\|_{L^2(B_{3R}(0))} \|\nabla (\varphi_R v)\|_{L^2(B_{3R}(0))}.
\end{equation}
Moreover, letting $\mF^\transpose$ be the transpose of the matrix $\mF$,
\begin{equation}\label{rock1}
\begin{split}
\|\nabla (\varphi_R v)\|_{L^2(B_{3R}(0))} \leq~& \|\varphi_R \nabla v\|_{L^2(B_{3R}(0))} + \big\| v \nabla \varphi_R^\transpose  \big\|_{L^2(B_{3R}(0))} \\
\leq~& \|\nabla v\|_{L^2(B_{3R}(0))} +   \| v\|_{L^2(B_{3R}(0))} \|\nabla \varphi_R\|_{L^2(B_{3R}(0))}.
\end{split}
\end{equation}
Furthermore,
\begin{align*}
 &\|\nabla \varphi_R\|_{L^2(B_{3R}(0))}^2 =~ \sum_{i=1}^d \int_{L^2(B_{3R}(0))}{\textstyle{\big|\frac{\partial \varphi_R}{\partial x_i}\big|^2}} \, dx
=~ \sum_{i=1}^d \int_{L^2(B_{3R}(0))}\textstyle{\big|\frac{\partial \eta_R}{\partial x_i} * \chi_R\big|^2} \, dx \\
&\qquad=~ \sum_{i=1}^d {\textstyle{\big\|\frac{\partial \eta_R}{\partial x_i} * \chi_R\big\|^2_{L^2(B_{3R}(0))} }}
\leq~ \sum_{i=1}^d {\textstyle{\big\|\frac{\partial \eta_R}{\partial x_i}\big\|^2_{L^1(B_{3R}(0))} \|\chi_R\|^2_{L^2(B_{3R}(0))}}} \quad \mbox{by Young's Inequality}
\\
&\qquad=~  \sum_{i=1}^d \Big(\int_{B_{3R}(0)}{\textstyle{\big|\frac{\partial \eta_R}{\partial x_i}\, dx\big|}}\Big)^2 \cdot \Big(\int_{B_{3R}(0)}{\textstyle{ |\chi_R|^2 \, dx}}\Big)
\leq~  \sum_{i=1}^d \bigg(\int_{B_{3R}(0)}{\textstyle{\big|\frac{1}{R^{d+1}}\frac{\partial \eta}{\partial x_i}(x/R)\big| \, dx}}\bigg)^2\cdot \bigg( \int_{B_{3R}(0)} {\textstyle{1 \, dx}}\bigg)
\\
&\qquad
=~ \sum_{i=1}^d \bigg(\int_{B_{3}(0)}\textstyle{\big|\frac{1}{R}\frac{\partial \eta}{\partial x_i}(x)\big| \, dx}\bigg)^2 \cdot \bigg(\int_{B_{3R}(0)} 1 \, dx\bigg)
\lesssim R^{d-2}.
\end{align*}
Thus for $d \geq 3, \|\nabla \varphi_R\|_{L^2(B_{3R}(0))} \to 0$ and for $d = 2$, $\|\nabla \varphi_R\|_{L^2(B_{3R}(0))}$ is uniformly bounded in $R$.  Since $v$ is fixed, $\| v\|_{L^2(B_{3R}(0))}  \to 0$ as $R \to 0$ and taking $R \to 0$ in~\eqref{rock0} and using~\eqref{rockNeg} and~\eqref{rock1} shows
\begin{equation*}\label{rock2}
\begin{split}
&\Big|\int_{\tilde{\Omega}_\a} \bbC : \nabla \bar{w}^\a_0 : \nabla  v \Big|
 =~ \lim_{R \to 0} \Big| \int_{B_{3R}(0)} \bbC : \nabla \bar{w}^\a_0 : \nabla (\varphi_R v) \Big| \\
&\quad\leq~ \lim_{R \to 0} \|\bbC : \nabla \bar{w}^\a_0\|_{L^2(B_{3R}(0))} \left(\|\nabla v\|_{L^2(B_{3R}(0))} +   \| v\|_{L^2(B_{3R}(0))} \|\nabla \varphi_R\|_{L^2(B_{3R}(0))}\right) =~ 0
\end{split}
\end{equation*}
so long as $d \geq 2$, which proves~\eqref{eq:lim_equation_both1}. The $d =1$ is special since the atomistic region becomes disconnected when a neighborhood of the origin is deleted.  To remedy this, additional constraints for each connected overlap region are required so the above arguments need to be carried out twice.

\medskip\noindent
{\it Proof of \eqref{eq:lim_equation_both2}.}
We prove (\ref{eq:lim_equation_both2}) for  $v \in \rm C^\infty_0(\tilde{\Omega}_\c)$; the general case follows by density. Interpolation of $v$ on each finite element grid $\tilde{\calT}_{h,n} = \epsilon_n\calT_{h,n}$ yields a sequence, $v^\c_n$, of piecewise linear functions with respect to $\tilde{\calT}_{h,n}$.  Let $V \subsubset \tilde{\Omega}_\c$ be a bounded set such that the support of $v$ and all but finitely many $v^\c_n$ are compactly contained in $V$. Then for all but finitely many $n$,
\begin{equation*}\label{rock3}
\begin{split}
0 =  \int_{\tilde{\Omega}_{\c,n}} W''(\nabla \tilde{\bm{u}}^{\rm con}_n) : \nabla \bar{\bm{w}}^\c_n : \nabla v^\c_n \, dx
= \int_V W''(\nabla \tilde{\bm{u}}^{\rm con}_n) : \nabla \bar{\bm{w}}^\c_n : \nabla v^\c_n \, dx.
\end{split}
\end{equation*}
%Summarizing, $v^\c_{n}$ converges to $v$ strongly in \dao{$H^1(V)$} and $\bar{w}^\c_n \weakto \bar{w}^\c_0$.
%
Taking limits of both sides produces
\begin{equation}\label{rock4}
\begin{split}
0 =~& \lim_{n \to \infty} \int_V W''(\nabla \tilde{\bm{u}}^{\rm con}_n) : \nabla \bar{\bm{w}}_n^\c : \nabla v^\c_{n} \, dx  \\
	=~& \lim_{n \to \infty} \int_V \left(W''(\nabla \tilde{\bm{u}}^{\rm con}_n ) - W''( \nabla I_n\tilde{\bm{u}}^{\infty}_n)\right) : \nabla \bar{\bm{w}}_n^\c : \nabla v^\c_{n} \, dx + \lim_{n \to \infty} \int_V W''(  \nabla I_n\tilde{\bm{u}}^{\infty}_n)  : \nabla \bar{\bm{w}}_n^\c : \nabla v^\c_{n} \, dx.
\end{split}
\end{equation}
Observe
\begin{align*}
&\lim_{n \to \infty} \int_V \left(W''(\nabla \tilde{\bm{u}}^{\rm con}_n ) - W''( \nabla I_n\tilde{\bm{u}}^{\infty}_n)\right) : \nabla \bar{\bm{w}}_n^\c : \nabla v^\c_{n} \\
&\qquad\qquad\lesssim	\lim_{n \to \infty} \epsilon_n\|\nabla \tilde{\bm{u}}^{\rm con}_n - \nabla I_n\tilde{\bm{u}}^{\infty}_n\|_{L^2(V)} \|\nabla \bar{\bm{w}}_n^\c\|_{L^2(V)} \|\nabla v^\c_{n}\|_{L^2(V)} = 0,
\end{align*}
due to Theorem~\ref{contModelError} estimating the continuum error and boundedness of $\nabla \bar{\bm{w}}_n^\c$ and $\nabla v^\c_{n}$.  Hence,~\eqref{rock4} simplifies to
\[
0 = \lim_{n \to \infty} \int_V W''(  \nabla I_n\tilde{\bm{u}}^{\infty}_n)  : \nabla \bar{\bm{w}}_n^\c : \nabla v^\c_{n} \, dx.
\]
Reasoning as in the end of the proof of Lemma~\ref{thm:norm_equiv_limiting_equation}, $W''(  \nabla I_n\tilde{\bm{u}}^{\infty}_n)$ converges uniformly to $W''(\bm{0})$ on $V$ while $\nabla v^\c_{n}$ converges strongly to $\nabla v$ in $H^1(V)$. The functions $\bar{\bm{w}}^\c_n$ converge weakly to $\bar{w}^\c_0$ in $\bm{H}^1(\tilde{\Omega}_\c)$, and since the norms, $\|\nabla \bm{w}\|^2_{L^2(\tilde{\Omega}_\c)},  \int_{\tilde{\Omega}_\c} \bbC :\nabla \bm{w} :\nabla \bm{w}\, dx = \langle \delta^2\calE^\c(\bm{0})\bm{w}, \bm{w}\rangle$ are equivalent on $\bm{H}^1(\tilde{\Omega}_\c)$, the functions $\bar{\bm{w}}^\c_n$ converge weakly to $\bar{w}^\c_0$ with respect to the $\int_{\tilde{\Omega}_\c} \bbC :\nabla \bm{w} :\nabla \bm{u}\, dx$ inner product.  Thus,
\[
0 = \lim_{n \to \infty} \int_V W''(  \nabla I_n\tilde{\bm{u}}^{\infty}_n)  : \nabla \bar{\bm{w}}_n^\c : \nabla v^\c_{n} \, dx = \int_V \bbC  : \nabla \bar{\bm{w}}_0^\c : \nabla v \, dx = \int_{\tilde{\Omega}_{\c}} \bbC  : \nabla \bar{\bm{w}}_0^\c : \nabla v \, dx.
\]
\end{proof}

\subsubsection*{Step 3:}

With the convergence properties of Step 1 and limiting equations of Step 2, we shall prove
\begin{theorem}\label{thm:inner_converge}
Let $\bar{w}^\a_n$ and $\nabla \bar{w}^\c_n$ be as defined in Step 1.  Then
\begin{equation}\label{inner_converge}
\big(\nabla I_{n}\bar{w}^\a_{n},\nabla \bar{w}^\c_{n}\big)_{L^2\left(\tilde{\Omega}_{\o}\right)} \to \left(\nabla \bar{w}^\a_0,\nabla \bar{w}^\c_0\right)_{L^2\left(\tilde{\Omega}_{\o}\right)}.
\end{equation}
\end{theorem}

\begin{proof}[Proof of Theorem~\ref{thm:inner_converge}]
Split $\tilde{\Omega}_\o$ into an inner part, $A_1$, and an outer part, $A_2$ such that $\tilde{\Omega}_\o = A_1 \cup A_2$ and $A_1$ and $A_2$ have disjoint interiors as in Figure~\ref{overlapDecomp}. Specifically, let $\lfloor x\rfloor$ be the greatest integer less than or equal to $x$ and set
\begin{align*}
A_1 :=~&  (\lfloor \psi_\a/2\rfloor \tilde{\Omega}_{\rm core}) \backslash  \tilde{\Omega}_{\rm core}, \\
A_2 :=~&   \tilde{\Omega}_{\o} \backslash A_1.
\end{align*}
\begin{figure}[htp!]
\centering
\includegraphics[width=0.5\textwidth]{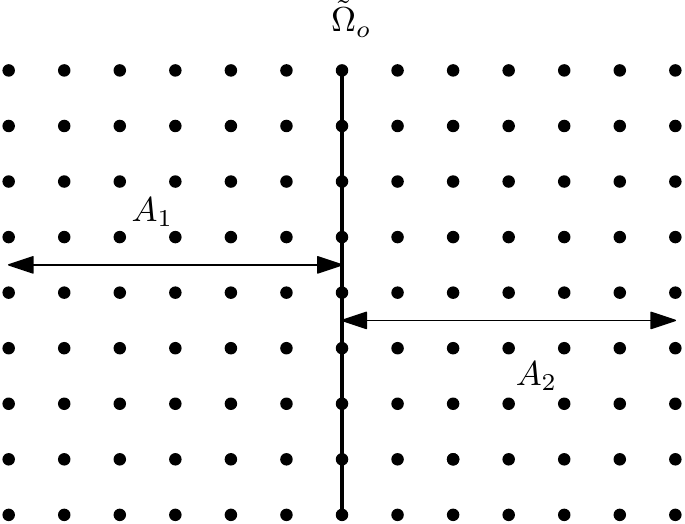}
\caption{An example decomposition of a portion of $\tilde{\Omega}_\o$ into $A_1$ and $A_2$.}
\label{overlapDecomp}
\end{figure}
We prove in Lemma~\ref{lem1:inner_converge} below that
\[
\|\nabla\left(\bar{w}^\c_n - \bar{w}^\c_0\right)\|_{L^2(A_2)} \to 0
\]
and in Lemma~\ref{lem2:inner_converge} that
\[
\big\|\nabla\big(I_{n}\bar{w}^\a_{n} - \bar{w}^\a_0\big)\big\|_{L^2(A_1)} \to 0.
\]
Using these two strong convergence results along with the weak convergence properties of Lemma~\ref{convergeLemma}---namely, $\bar{w}^\c_n \weakto \bar{w}^\c_0$ on $A_1$ and $\bar{w}^\a_n \weakto \bar{w}^\a_0$ on $A_2$---yields
\begin{equation}\label{water4}
\begin{split}
\big(\nabla I_{n}\bar{w}^\a_{n},\nabla \bar{w}^\c_{n}\big)_{L^2(\tilde{\Omega}_{\o})}  =~& \big(\nabla I_{n}\bar{w}^\a_{n},\nabla \bar{w}^\c_{n}\big)_{L^2(A_1)}  + \big(\nabla I_{n}\bar{w}^\a_{n},\nabla \bar{w}^\c_{n}\big)_{L^2(A_2)} \\
\to~&  (\nabla \bar{w}^\a_0,\nabla \bar{w}^\c_0)_{L^2(A_1)}  + \left(\nabla \bar{w}^\a_0,\nabla \bar{w}^\c_0\right)_{L^2(A_2)}
= (\nabla \bar{w}^\a_0,\nabla \bar{w}^\c_0)_{L^2(\tilde{\Omega}_{\o})}.
\end{split}
\end{equation}

\end{proof}

In the preceding theorem, we have made reference to the following lemma, which we now prove.
\begin{lemma}\label{lem1:inner_converge}
Let $\bar{w}^\c_n$ and $\bar{w}^\c_0$ be as defined in Lemma~\ref{convergeLemma}. Then
\begin{equation}\label{finalGoal1}
\|\nabla\left(\bar{w}^\c_n - \bar{w}^\c_0\right)\|_{L^2(A_2)} \to 0.
\end{equation}
\end{lemma}

\begin{proof}
%Recall that each element of the continuum sequence satisfies a variational equality of the form
%\begin{equation}\label{contEquation}
%\int_{\tilde{\Omega}_{\c,n}} W''(\nabla \tilde{\bm{u}}^{\rm con}_n) : \nabla \bar{w}^\c_n : \nabla v^\c_n \, dx=~ 0 \quad \forall v^\c \in \tilde{\bm{\mathcal{U}}}^\c_{h,0,n} .
%\end{equation}
%According to Theorem~\ref{thm:norm_equiv_limiting_equation1}, the function $\bar{w}^\c_0$ satisfies a variational equality of the form
%\begin{equation*}\label{contLimitEquation}
%\int_{\tilde{\Omega}_{\c}} \bbC : \nabla \bar{w}_0^\c : \nabla v^\c_{0} \, dx = 0 \quad \forall v^\c_{0} \in H^1_0(\tilde{\Omega}_{\c}),
%\end{equation*}
%which corresponds to a linear elliptic system.  From elliptic regularity, $\bar{w}^\c_0$ belongs to $H^2_{\rm loc}(\tilde{\Omega}_\c)$~\cite{giaquinta1983}.  Recalling that the mesh is fully resolved on $\tilde{\Omega}_{\o, \ex}$, it follows that
%\begin{equation*}\label{limit1}
%\hat{w}^\c_n := I_n \bar{w}^\c_0 \to \bar{w}^\c_0 \quad \mbox{in \dao{$H^1(\tilde{\Omega}_{\o, \ex})$} and hence in $H^1(A_2).$}
%\end{equation*}
%\commentas{is it possible to completely avoid introducing $\hat{w}^\c_n$?}

%The goal is now to show
%\begin{equation}\label{goal1}
%\|\nabla(\bar{w}^\c_n - \hat{w}^\c_n)\|_{L^2(A_2)} \to 0,
%\end{equation}
%which will further imply~\eqref{finalGoal1}.
We let $\eta$ be a smooth bump function with compact support in $\tilde{\Omega}_{\o,\ex}$ and equal to $1$ on $A_2$. Our starting point in proving~\eqref{finalGoal1} will be to define $z_n := \bar{w}^\c_n - \bar{w}^\c_0$ and bound $\|\nabla z_n\|_{L^2(A_2)} \leq \|\nabla (\eta z_n)\|_{L^2(\tilde{\Omega}_{\o,\ex})}$.  Then we shall prove $\|\nabla (\eta z_n)\|_{L^2(\tilde{\Omega}_{\o,\ex})} \to 0$.

Note that $z_n \weakto 0$ in $H^1(\tilde{\Omega}_{\o, \ex})$ by the definition of $z_n$ and~\eqref{varyCont}.
As a simple corollary, $\eta z_n\weakto 0$ in $H^1(\tilde{\Omega}_{\o, \ex})$, and therefore a short calculation implies $\nabla (\eta z_n) \weakto 0$ in $L^2(\tilde{\Omega}_{\o, \ex})$.  Since $\eta z_n$ can be extended by $0$ to all of $\bbR^d$, coercivity of the continuum Hessian~\eqref{contStability} gives us
\begin{equation}\label{omlette1}
\begin{split}
&\|\nabla z_n\|_{L^2(A_2)}^2 \leq \|\nabla (\eta z_n)\|_{L^2(\tilde{\Omega}_{\o,\ex})}^2
\lesssim~
	\int_{\tilde{\Omega}_{\o,\ex}}W''(\nabla\tilde{\bm{u}}^{\rm con}_n):\nabla(\eta z_n):\nabla(\eta z_n)\, dx
\\&=
	\int_{\tilde{\Omega}_{\o,\ex}}W''(\nabla\tilde{\bm{u}}^{\rm con}_n):\nabla(\eta \bar{w}^\c_n):\nabla(\eta z_n)\, dx
	-
	\int_{\tilde{\Omega}_{\o,\ex}}W''(\nabla\tilde{\bm{u}}^{\rm con}_n):\nabla(\eta \bar{w}^\c_0):\nabla(\eta z_n)\, dx
%\\&=
%	\int_{\tilde{\Omega}_{\o,\ex}}W''(\nabla\tilde{\bm{u}}^{\rm con}_n):\nabla(\eta \bar{w}^\c_n):\nabla(\eta z_n)\, dx
%	-
%	o(1),
\end{split}
\end{equation}
Taking the limit of~\eqref{omlette1} and using that $\nabla(\eta z_n) \weakto 0$ weakly in $L^2(\tilde{\Omega}_{\o,\ex}) $ while $W''(\nabla\tilde{\bm{u}}^{\rm con}_n)\to W''(0)$ strongly in $L^\infty(\tilde{\Omega}_{\o,\ex})$ yields
\begin{equation*}
\lim_{n \to \infty} \|\nabla z_n\|_{L^2(A_2)}^2 \lesssim~ \lim_{n \to \infty} \int_{\tilde{\Omega}_{\o,\ex}}W''(\nabla\tilde{\bm{u}}^{\rm con}_n):\nabla(\eta \bar{w}^\c_n):\nabla(\eta z_n)\, dx.
\end{equation*}
We hence continue to estimate
\begin{equation*}\label{omlette2}
\begin{split}
&\lim_{n\to \infty}\|\nabla z_n\|_{L^2(A_2)}^2
\\\phantom{\mathstrut-\mathstrut}&\lesssim~
	\lim_{n\to \infty}\int_{\tilde{\Omega}_{\o, \ex}}W''(\nabla\tilde{\bm{u}}^{\rm con}_n):\nabla \bar{w}^\c_n:\eta\nabla(\eta z_n)\, dx
	\,\,\,\,+\,\,\,\,
	\lim_{n\to \infty}\int_{\tilde{\Omega}_{\o, \ex}}W''(\nabla\tilde{\bm{u}}^{\rm con}_n): \bar{w}^\c_n (\nabla\eta)^\transpose:\nabla(\eta z_n)\, dx
\\
&= \phantom{\mathstrut-\mathstrut}
	\lim_{n \to \infty} \int_{\tilde{\Omega}_{\o, \ex}}W''(\nabla\tilde{\bm{u}}^{\rm con}_n):\nabla \bar{w}^\c_n:\nabla(\eta^2 z_n)\, dx
\\&\phantom{\mathstrut=\mathstrut}-
	\lim_{n \to \infty} \int_{\tilde{\Omega}_{\o, \ex}}W''(\nabla\tilde{\bm{u}}^{\rm con}_n):\nabla \bar{w}^\c_n:\eta z_n(\nabla\eta)^\transpose\, dx
	+
	\lim_{n \to \infty} \int_{\tilde{\Omega}_{\o, \ex}}W''(\nabla\tilde{\bm{u}}^{\rm con}_n):\bar{w}^\c_n(\nabla\eta)^\transpose :\nabla(\eta z_n)\, dx,
\end{split}
\end{equation*}
where the second limit converges to zero thanks to $z_n \to 0$ in $L^2(\tilde{\Omega}_{\o, \rm ex})$ and $\nabla\bar{w}^\c_n \weakto \nabla\bar{w}^\c_0$ in $L^2(\tilde{\Omega}_{\o, \rm ex})$ and the third term converges to zero because $\bar{w}^\c_n \to \bar{w}^\c_0$ and $\nabla(\eta z_n) \weakto 0$ in $L^2(\tilde{\Omega}_{\o, \rm ex})$ (of course, both together with $W''(\nabla\tilde{\bm{u}}^{\rm con}_n)\to W''(0)$ in $L^\infty(\tilde{\Omega}_{\o, \rm ex})$).  Thus
\begin{equation*}
\lim_{n\to \infty}\|\nabla z_n\|_{L^2(A_2)}^2	\lesssim~ \lim_{n \to \infty} \int_{\tilde{\Omega}_{\o, \ex}}W''(\nabla\tilde{\bm{u}}^{\rm con}_n):\nabla \bar{w}^\c_n:\nabla(\eta^2 z_n)\, dx.
\end{equation*}

To estimate the this term, we recall each $\bar{w}^\c_n$ solves a variational equality of the form
\begin{equation*}
\int_{\tilde{\Omega}_{\c,n}} W''(\nabla \tilde{\bm{u}}^{\rm con}_n) : \nabla \bar{w}^\c_n : \nabla v^\c_n \, dx=~ 0 \quad \forall v^\c \in \tilde{\bm{\mathcal{U}}}^\c_{h,0,n} .
\end{equation*}
We use this equality with $v^\c_n = I_n\big(\eta^2 z_n \big)\in \tilde{\bm{\mathcal{U}}}_{h,0,n}$ to further estimate
\begin{equation}\label{omlette3}
\begin{split}
&\lim_{n \to \infty} \|\nabla z_n\|_{L^2(A_2)}^2 \\
&\lesssim~
\lim_{n \to \infty} \int_{\tilde{\Omega}_{\o, \ex}}W''(\nabla\tilde{\bm{u}}^{\rm con}_n):\nabla \bar{w}^\c_n:\nabla(\eta^2 z_n)\, dx.
	-
	\lim_{n \to \infty} \int_{\tilde{\Omega}_{\o, \ex}}W''(\nabla\tilde{\bm{u}}^{\rm con}_n):\nabla \bar{w}^\c_n:\nabla I_n(\eta^2 z_n)\, dx
\\
&=~
	\lim_{n \to \infty} \int_{\tilde{\Omega}_{\o, \ex}}W''(\nabla\tilde{\bm{u}}^{\rm con}_n):\nabla \bar{w}^\c_n:\nabla(\eta^2 z_n - I_n(\eta^2 z_n) )\, dx\\
&\lesssim~
	\lim_{n \to \infty} \big\|\nabla(\eta^2 z_n - I_n(\eta^2 z_n) )\big\|_{L^2(\tilde{\Omega}_{\o, \ex})}
.
\end{split}
\end{equation}

Next,
\[
\lim_{n \to \infty} \big\|\nabla(\eta^2 z_n - I_n(\eta^2 z_n) )\big\|_{L^2(\tilde{\Omega}_{\o, \ex})} \leq \lim_{n \to \infty}\big\|\nabla(\eta^2 \bar{w}^\c_n - I_n(\eta^2\bar{w}^\c_n))\big\|_{L^2(\tilde{\Omega}_{\o, \ex})} + \lim_{n \to \infty}\big\| \nabla(\eta^2 \bar{w}^\c_0 - I_n(\eta^2\bar{w}^\c_0))\big\|_{L^2(\tilde{\Omega}_{\o, \ex})}.
\]
According to Theorem~\ref{thm:norm_equiv_limiting_equation1}, the function $\bar{w}^\c_0$ satisfies a variational equality of the form
\begin{equation*}\label{contLimitEquation}
\int_{\tilde{\Omega}_{\c}} \bbC : \nabla \bar{w}_0^\c : \nabla v^\c_{0} \, dx = 0 \quad \forall v^\c_{0} \in H^1_0(\tilde{\Omega}_{\c}),
\end{equation*}
which corresponds to a linear elliptic system.  From elliptic regularity, $\bar{w}^\c_0$ belongs to $H^2_{\rm loc}(\tilde{\Omega}_\c)$~\cite{giaquinta1983,suli2012}.  Thus, standard finite element approximation theory implies
\[
\lim_{n \to \infty}\big\| \nabla(\eta^2 \bar{w}^\c_0 - I_n(\eta^2\bar{w}^\c_0))\big\|_{L^2(\tilde{\Omega}_{\o, \ex})} \lesssim \lim_{n \to \infty}\epsilon_n\big\|\nabla^2(\eta^2 \bar{w}^\c_0)\big\|_{L^2(\tilde{\Omega}_{\o, \ex})} = 0.
\]
Finally, to show
\begin{equation}\label{fact2}
\lim_{n \to \infty} \big\| \nabla\big(\eta^2 \bar{w}^\c_n - I_n(\eta^2 \bar{w}^\c_n\big)) \big\|_{L^2(\tilde{\Omega}_{\o, \ex})} = 0,
\end{equation}
observe that $\eta^2 \bar{w}^\c_n - I_n (\eta^2 \bar{w}^\c_n )$ vanishes outside a neighborhood $N_\delta \subsubset \tilde{\Omega}_{\o, \ex}$ of $\supp (\eta)$.  Then
$$
\begin{array}{l}
\displaystyle
\big\| \nabla\big(\eta^2 \bar{w}^\c_n - I_n(\eta^2 \bar{w}^\c_n )\big) \big\|_{L^2(\tilde{\Omega}_{\o, \ex})}^2 = \big\| \nabla\big(\eta^2 \bar{w}^\c_n - I_n(\eta^2 \bar{w}^\c_n)\big) \|_{L^2(N_\delta)}^2
= \int_{N_\delta} \big|\nabla\big(\eta^2 \bar{w}^\c_n - I_n(\eta^2 \bar{w}^\c_n)\big)\big|^2\, dx
\\[2ex]
\displaystyle\quad
\leq~ \sum_{\substack{T \in \tilde{\calT}_{\a,n} \\ T \cap N_\delta \neq \emptyset }}  \int_{T} \big|\nabla\big(\eta^2 \bar{w}^\c_n - I_n(\eta^2 \bar{w}^\c_n)\big)\big|^2\, dx
\lesssim~ \sum_{\substack{T \in \tilde{\calT}_{\a,n} \\ T \cap N_\delta \neq \emptyset }} |T|^2 \big\|\nabla^2(\eta^2 \bar{w}^\c_n)\big\|_{L^2(T)}^2 \lesssim~ \epsilon_n^{2d} \sum_{\substack{T \in \tilde{\calT}_{\a,n} \\ T \cap N_\delta \neq \emptyset }} \big\|\nabla^2(\eta^2 \bar{w}^\c_n)\big\|_{L^2(T)}^2
\end{array}
$$
where the last line follows from the Bramble-Hilbert lemma and scaling.
Because $\bar{w}^\c_n$ is piecewise linear its second derivatives vanish on all $T$. Using the uniform boundedness of
$\eta$ and its derivatives then yields
\begin{align*}
\big\|\nabla^2(\eta^2 \bar{w}^\c_n)\big\|_{L^2(T)}^2
=~&\int_{T} \big|\nabla^2(\eta^2 \bar{w}^\c_n)\big|^2 \, dx
 %\lesssim~& \int_{T} {\textstyle{\sup_{i,k}\big|\frac{\partial \eta}{\partial x_i \partial x_k}\big|^2 |\bar{w}^\c_n |^2\, dx}} + \int_{T} {\textstyle{\sup_{i}\big|\frac{\partial \eta}{x_i}\big|^2 | \nabla( \bar{w}^\c_n)|^2\, dx}} \\
\lesssim~  \int_{T} | \bar{w}^\c_n |^2\, dx + \int_{T} | \nabla \bar{w}^\c_n  |^2\, dx.
\end{align*}
Choose $N_\delta'$ such that $\bigcup_{\substack{T \in \tilde{\calT}_{\a,n} \\ T \cap N_\delta \neq \emptyset }} \subset N_\delta' \subsubset \tilde{\Omega}_{\o, \ex}$ for all but finitely many $n$.  Then for all such $n$,
\begin{align*}
\big\| \nabla\big(\eta^2 \bar{w}^\c_n - I_n(\eta^2 \bar{w}^\c_n)\big) \big\|_{L^2(N_\delta)}^2
\lesssim~&
\epsilon_n^{2d} \sum_{\substack{T \in \tilde{\calT}_{\a,n} \\ T \cap N_\delta \neq \emptyset }}\int_T| \bar{w}^\c_n |^2 + | \nabla \bar{w}^\c_n |^2\, dx \\
\lesssim~&    \epsilon_n^{2d} \big(\| \bar{w}^\c_n \|_{L^2(N_\delta')}^{2} + \|\nabla \bar{w}^\c_n\|_{L^2(N_\delta')}^{2}\big).
\end{align*}

Now note that $\| \bar{w}^\c_n \|_{L^2(N_\delta')} \to 0$ while $\|\nabla \bar{w}^\c_n\|_{L^2(N_\delta')}$ is bounded since $\bar{w}^\c_n$ is weakly convergent in $H^1(N_\delta')$.  As $\epsilon_n$ goes to $0$, we obtain~\eqref{fact2}.  Inserting~\eqref{fact2} into~\eqref{omlette3} proves the theorem.

\end{proof}

Our second task is to prove the atomistic version of Lemma~\ref{lem1:inner_converge} over $A_1$.
\begin{lemma}\label{lem2:inner_converge}
Let $\bar{w}^\a_n$ and $\bar{w}^\a_0$ be as defined in Lemma~\ref{convergeLemma}.  Then
\begin{equation}\label{goal2}
\big\|\nabla\big(I_{n}\bar{w}^\a_{n} - \bar{w}^\a_0\big)\big\|_{L^2(A_1)} \to 0.
\end{equation}
\end{lemma}

\begin{proof}
As in previous case, $\bar{w}^\a_0 \in H^2_{\rm loc}(\tilde{\Omega}_\a)$ so we again consider again a sequence $\hat{w}^\a_n := I_n\bar{w}^\a_0$, which converges in $H^1(A_1)$ to $\bar{w}^\a_0$. Set $X := (\lfloor \psi_\a/2\rfloor + 1) \tilde{\Omega}_{\rm core}$, and take $\eta$ to be a bump function equal to one on $A_1$, zero on a neighborhood of the origin, and $\rm supp(\eta) \subsubset X$, i.e. $\eta$ rapidly vanishes off $A_1$. Note that we still possess convergence of $\hat{w}^\a_n$ to $\bar{w}^\a_0$ in $H^1(X)$.  We also know $\bar{w}^\a_n \weakto \bar{w}^\a_0$ in $H^1(\tilde{\Omega}_\a)$ by Lemma~\ref{convergeLemma} so $y_n := \bar{w}^\a_n - \hat{w}^\a_n$ converges weakly to zero in $H^1(X)$.

We recall that the product rule for difference quotients involves a shift operator which we denote by $T_{r}$:
\begin{equation}\label{shift}
\begin{array}{r@{}l@{\hspace{3em}}l@{\hspace{3em}}r@{}l}
D_{\epsilon_n\rho}(uv)(\xi) =~& (D_{\epsilon_n\rho}u)v + (T_{\epsilon_n\rho}u)D_{\epsilon_n\rho}v,
&\text{where}&
T_{\epsilon_n\rho}v(\xi) :=~& v(\xi + \epsilon_n\rho),
\\
T_{\epsilon_n}v(\xi) :=~& (T_{\epsilon\rho}v(\xi))_{\rho \in \calR},
&\text{and}&
T_{\epsilon_n}u \, D_{\epsilon_n}v =~& \left(T_{\epsilon_n\rho}u \, D_{\epsilon_n\rho}v\right)_{\rho \in \calR},
\end{array}
\end{equation}
and choose a domain $\Omega_1 \subsubset X$ such that ${\rm supp}(T_{\epsilon_nr}\eta) \subsubset \Omega_1$ for all but finitely many $n$.  Because $y_n$ converges weakly to zero in $H^1(X)$, the conclusion of Lemma~\ref{lem:norm_equiv:main_aux} asserts that
\begin{equation}\label{talon}
\bar{I}_n D_{\epsilon_n} y_n \weakto 0 \quad \mbox{in} \quad L^2(\Omega_1).
\end{equation}
Since $\hat{w}^\a_{n}$ converges strongly to $\bar{w}^\a_0$ in $H^1(X)$, Remark~\ref{afterLem37} further implies
\begin{equation}\label{talon2}
\bar{I}_{n}D_{\epsilon_{n}r}\hat{w}^\a_{n} \to \nabla_r\bar{w}^\a_0 \quad \mbox{in} \quad L^2(\Omega_1).
\end{equation}
Furthermore, $I_n(\eta) \to \eta$ in $L^\infty(X)$ and thus $I_n(\eta\hat{w}^\a_n) = (I_n\eta)\hat{w}^\a_n$ converges strongly to $\eta \bar{w}^\a_0$ in $H^1(X)$ so that
\begin{equation}\label{talon3}
\bar{I}_{n}D_{\epsilon_{n}r}I_n(\eta\hat{w}^\a_n) \to \nabla_r (\eta \bar{w}^\a_0) \quad \mbox{in} \quad L^2(\Omega_1).
\end{equation}
Reasoning similarly, we have that $I_n(\eta y_n)$ converges weakly to $0$ on $H^1(X)$ which implies
\begin{equation}\label{talon4}
\bar{I}_{n}D_{\epsilon_{n}r}(I_n(\eta y_n)) \weakto 0 \quad \mbox{in} \quad L^2(\Omega_1).
\end{equation}
%This and the facts that $\bar{I}_n(T_{\epsilon_n}\eta) \to \eta$ in $L^\infty(\Omega_1)$, $\bar{I}_n(D_{\epsilon_n}\eta)
%\begin{equation}\label{talon3}
%\begin{split}
%\bar{I}_{n}D_{\epsilon_{n}r}(\eta\hat{w}^\a_{n}) =~& \bar{I}_n(T_{\epsilon_n}\eta D_{\epsilon_n}\hat{w}^\a_n) + \bar{I}_n(D_{\epsilon_n}\eta \hat{w}^\a_n) = \bar{I}_n(T_{\epsilon_n}\eta)\bar{I}_n(D_{\epsilon_n}\hat{w}^\a_n) + \bar{I}_n(D_{\epsilon_n}\eta)\bar{I}_n(\hat{w}^\a_n) \\
%=~& 1.
%\end{split}
%\end{equation}

These convergence properties and the fact that each $\bar{w}^\a_n$ solves
\begin{equation}\label{pirate0}
0 =~ \sum_{\xi \in \tilde{\calL}_{\a,n}^{\circ\circ}} V_\xi''(D_{\epsilon_n} \tilde{\bm{u}}^{\infty}_{\a,n}) \!:\!  D_{\epsilon_n}\bar{w}^\a_n : D_{\epsilon_n}v^\a \quad \forall v^\a \in \tilde{\bm{\calU}}^\a_{0,n}
\end{equation}
will be used later in the proof.

From coercivity of the atomistic Hessian in~\eqref{condition} and the product rule for difference quotients,
\begin{equation*}\label{pirate}
\begin{split}
	\|\nabla I_n y_n \|_{L^2(A_1)}^2
\lesssim~
	\|\nabla I_n (\eta y_n) \|_{L^2(A_1)}^2
\lesssim~& \langle \delta^2\tilde{\mathcal{E}}^\a(\tilde{\bm{u}}^{\infty}_{\a,n})(\eta y_n),(\eta y_n)\rangle \\
=~& \sum_{\xi \in \tilde{\calL}_{\a,n}^{\circ\circ}} V_\xi''(D_{\epsilon_n} \tilde{\bm{u}}^{\infty}_{\a,n})\!:\! D_{\epsilon_n}(\eta y_n):D_{\epsilon_n}(\eta y_n).
\end{split}
\end{equation*}
We now employ the integral formulation~\eqref{intForm}, which is valid since $\eta$ rapidly vanishes off $A_1$ and due to the choice of $A_1$, and take limits:
\begin{equation}\label{pirate1}
\begin{split}
&\lim_{n \to \infty} \|\nabla I_n y_n \|_{L^2(A_1)}^2 \\
&\lesssim~ \lim_{n \to \infty}\int_{\tilde{\Omega}_\a} \bar{I}_n V''(D_{\epsilon_n} \tilde{\bm{u}}^{\infty}_{\a,n}) \!:\! \bar{I}_nD_{\epsilon_n}(\eta y_n) : \bar{I}_n D_{\epsilon_n}(\eta y_n) \, dx \\
&=~  \lim_{n \to \infty}\int_{\tilde{\Omega}_\a} \bar{I}_n V''(D_{\epsilon_n} \tilde{\bm{u}}^{\infty}_{\a,n}) \!:\! \bar{I}_nD_{\epsilon_n}(\eta \bar{w}^\a_n) : \bar{I}_n D_{\epsilon_n}(\eta y_n) - \lim_{n \to \infty}\int_{\tilde{\Omega}_\a} \bar{I}_n V''(D_{\epsilon_n} \tilde{\bm{u}}^{\infty}_{\a,n}) \!:\! \bar{I}_nD_{\epsilon_n}(\eta \hat{w}^\a_n) : \bar{I}_n D_{\epsilon_n}(\eta y_n).
\end{split}
\end{equation}
%Using the product rule of~\eqref{shift} for difference quotients and taking the limit of the second integral yields
%\begin{equation}\label{pirate2}
%\begin{split}
%&\lim_{n \to \infty} \int_{\tilde{\Omega}_\a} \bar{I}_n V''(D_{\epsilon_n} \tilde{\bm{u}}^{\infty}_{\a,n}) \!:\! \bar{I}_nD_{\epsilon_n}(\eta \hat{w}^\a_n) : \bar{I}_n D_{\epsilon_n}(\eta y_n) \\
%&=~ \lim_{n \to \infty}\int_{\tilde{\Omega}_\a} \bar{I}_n V''(D_{\epsilon_n} \tilde{\bm{u}}^{\infty}_{\a,n}) \!:\! \bar{I}_nD_{\epsilon_n}(\eta \hat{w}^\a_n) : \bar{I}_n (T_{\epsilon_n} \eta D_{\epsilon_n} y_n) +  \lim_{n \to \infty}\int_{\tilde{\Omega}_\a} \bar{I}_n V''(D_{\epsilon_n} \tilde{\bm{u}}^{\infty}_{\a,n}) \!:\! \bar{I}_nD_{\epsilon_n}(\eta \hat{w}^\a_n) : \bar{I}_n (y_n D_{\epsilon_n} \eta) \\
%&=~ \lim_{n \to \infty}\int_{\Omega_1} \bar{I}_n\big(T_{\epsilon_n} \eta V''(D_{\epsilon_n} \tilde{\bm{u}}^{\infty}_{\a,n})\big) \!:\! \bar{I}_nD_{\epsilon_n}(I_n(\eta \hat{w}^\a_n)) : \bar{I}_n D_{\epsilon_n} y_n +  \lim_{n \to \infty}\int_{\Omega_1} \bar{I}_n \big(D_{\epsilon_n} \eta V''(D_{\epsilon_n} \tilde{\bm{u}}^{\infty}_{\a,n}) \big)\!:\! \bar{I}_nD_{\epsilon_n}(I_n(\eta \hat{w}^\a_n)) : \bar{I}_n y_n.
%\end{split}
%\end{equation}
The second limit is zero after noting we may write the integral over $\Omega_1$ (relying on how $\Omega_1$ was chosen) and then using~\eqref{talon3},~\eqref{talon4}, and that $\bar{I}_n \eta V''(D_{\epsilon_n} \tilde{\bm{u}}^{\infty}_{\a,n})$ converges to $V''(0)$ in $L^\infty(\Omega_1)$.

%Since $T_{\epsilon_nr}\eta$ converges to the uniformly continuous function $\eta$ in $L^\infty(\Omega_1)$, and likewise $D_{\epsilon_nr} \eta$ converges to the uniformly continuous $\nabla_r \eta$ in $L^\infty(\Omega_1)$, it follows that $\bar{I}_n\big(T_{\epsilon_n} \eta V''(D_{\epsilon_n} \tilde{\bm{u}}^{\infty}_{\a,n})\big)$ and $\bar{I}_n \big(D_{\epsilon_n} \eta V''(D_{\epsilon_n} \tilde{\bm{u}}^{\infty}_{\a,n}) \big)$ converges to $V''(0)$ in $L^\infty(\Omega_1)$.  Equation~\eqref{talon3} states
%$\bar{I}_{n}D_{\epsilon_{n}r}I_n(\eta\hat{w}^\a_n)$ converges strongly to $\nabla_r (\eta \bar{w}^\a_0)$ in $L^2(\Omega_1)$, while~\eqref{talon} shows $\bar{I}_n D_{\epsilon_n} y_n$ converges weakly to $0$ in $L^2(\Omega_1)$.  Thus, the first limit is zero.  The bound, $\|\bar{I}_n y_n\|_{L^2(\Omega_1)} \lesssim \|y_n\|_{L^2(X)} \to 0$, then implies the second limit is zero.

Returning to~\eqref{pirate1}
\begin{equation*}
\begin{split}
&\lim_{n \to \infty} \|\nabla I_n y_n \|_{L^2(A_1)}^2 \\
&\lesssim~ \lim_{n \to \infty} \int_{\tilde{\Omega}_\a} \bar{I}_n V''(D_{\epsilon_n} \tilde{\bm{u}}^{\infty}_{\a,n}) \!:\! \bar{I}_nD_{\epsilon_n}(\eta \bar{w}^\a_n) : \bar{I}_n D_{\epsilon_n}(\eta y_n)
\\ &=~
	\lim_{n \to \infty} \int_{\tilde{\Omega}_\a}
		\bar{I}_n V''(D_{\epsilon_n} \tilde{\bm{u}}^{\infty}_{\a,n})
		\!:\! (\bar{I}_n D_{\epsilon_n}\bar{w}^\a_n)
		(\bar{I}_nT_{\epsilon_n}\eta)
		:
		\bar{I}_n D_{\epsilon_n}(\eta y_n)
\\&\quad+
	\lim_{n \to \infty} \int_{\tilde{\Omega}_\a} \bar{I}_n V''(D_{\epsilon_n} \tilde{\bm{u}}^{\infty}_{\a,n}) \!:\! (\bar{I}_n\bar{w}^\a_n) (\bar{I}_nD_{\epsilon_n}\eta) : \bar{I}_n D_{\epsilon_n}(\eta y_n)
\\ &=~
	\lim_{n \to \infty} \int_{\tilde{\Omega}_\a} \bar{I}_n V''(D_{\epsilon_n} \tilde{\bm{u}}^{\infty}_{\a,n}) \!:\! \bar{I}_nD_{\epsilon_n}\bar{w}^\a_n : \bar{I}_n D_{\epsilon_n}(\eta^2 y_n) - \lim_{n \to \infty} \int_{\tilde{\Omega}_\a} \bar{I}_n V''(D_{\epsilon_n} \tilde{\bm{u}}^{\infty}_{\a,n}) \!:\! \bar{I}_nD_{\epsilon_n}\bar{w}^\a_n : \bar{I}_nD_{\epsilon_n}(\eta) \bar{I}_n (\eta y_n)
\\&\quad+
	\lim_{n \to \infty} \int_{\tilde{\Omega}_\a} \bar{I}_n V''(D_{\epsilon_n} \tilde{\bm{u}}^{\infty}_{\a,n}) \!:\! (\bar{I}_n\bar{w}^\a_n) (\bar{I}_nD_{\epsilon_n}\eta) : \bar{I}_n D_{\epsilon_n}(\eta y_n)
\end{split}
\end{equation*}
The first of these limits is zero due to~\eqref{pirate0}.  The second is also zero since Lemma~\ref{lem:norm_equiv:main_aux} implies $\bar{I}_nD_{\epsilon_n}\bar{w}^\a_n$ converges weakly to $\nabla \bar{w}^\a_0$, $\|\bar{I}_n(\eta y_n)\|_{L^2(\Omega_1)} \lesssim \|\bar{I}_n(y_n)\|_{L^2(\Omega_1)} \lesssim \|y_n\|_{L^2(X)} \to 0$, and since $D_{\epsilon_nr} \eta$ converges to the uniformly continuous $\nabla_r \eta$ in $L^\infty(\Omega_1)$, it follows that $\bar{I}_n\big(T_{\epsilon_n} \eta V''(D_{\epsilon_n} \tilde{\bm{u}}^{\infty}_{\a,n})\big)$ and $\bar{I}_n \big(D_{\epsilon_n} \eta V''(D_{\epsilon_n} \tilde{\bm{u}}^{\infty}_{\a,n}) \big)$ converges to $V''(0)$ in $L^\infty(\Omega_1)$.  Using this latter fact, the third limit is then zero due to~\eqref{talon4} and $\|\bar{I}_n\bar{w}^\a_n\|_{L^2(\Omega_1)} \lesssim \|I_n\bar{w}^\a_n\|_{L^2(X)} = \|\bar{w}^\a_n\|_{L^2(X)} \to 0$.

\end{proof}

\subsubsection*{Step 4:}

\begin{proof}[Conclusion of Proof of Theorem~\ref{lem:norm_equiv:desired_result}]
We assume the existence of a sequence satisfying~\eqref{contradiction}, which yields sequences of normalized functions $\bar{w}^\a_n$ and $\bar{w}^\c_n$ possessing properties~\eqref{varyAt}--\eqref{varyCont} of Lemma~\ref{convergeLemma}.  Combining~\eqref{inner_converge} of Theorem~\ref{thm:inner_converge} with~\eqref{vergeResult} resulting from Statement~\ref{false} shows
\begin{equation}\label{vergeResult2}
\left(\nabla \bar{w}^\a_0,\nabla \bar{w}^\c_0\right)_{L^2\left(\tilde{\Omega}_{\o}\right)} = 1.
\end{equation}
The weak convergence of $I_n \bar{w}^\a_n$ to $\bar{w}^\a_0$ implies that $\|\nabla \bar{w}^\a_0\|_{L^2(\tilde{\Omega}_\o)} \leq \limsup_{n\to\infty} \|\nabla \bar{w}^\a_n\|_{L^2(\tilde{\Omega}_\o)} = 1$, and likewise we have that $\|\nabla \bar{w}^\c_0\|_{L^2(\tilde{\Omega}_\o)} \leq 1$.
In view of \eqref{vergeResult2}, it is only possible if $\|\nabla \bar{w}^\a_0\|_{L^2(\tilde{\Omega}_\o)} = \|\nabla \bar{w}^\c_0\|_{L^2(\tilde{\Omega}_\o)} = 1$ and
\[
\left(\nabla \bar{w}^\a_0,\nabla \bar{w}^\c_0\right)_{L^2\left(\tilde{\Omega}_{\o}\right)} = \|\nabla \bar{w}^\a_0\|_{L^2\left(\tilde{\Omega}_{\o}\right)} \|\nabla \bar{w}^\c_0\|_{L^2\left(\tilde{\Omega}_{\o}\right)}.
\]
Hence $\nabla \bar{w}^\a_0 = \alpha \nabla \bar{w}^\c_0$ on $\tilde{\Omega}_\o$ for some real number $\alpha$ implying
\[
1 = \left(\alpha \nabla \bar{w}^\c_0, \nabla \bar{w}^\c_0\right)_{L^2\left(\tilde{\Omega}_{\o}\right)} = \alpha \|\nabla \bar{w}^\c_0\|_{L^2\left(\tilde{\Omega}_{\o}\right)}^2 = \alpha.
\]
Thus $\nabla \bar{w}^\a_0$ and $\nabla \bar{w}^\c_0$ are equal on $\tilde{\Omega}_\o$ so $\bar{w}^\a_0$ and $\bar{w}^\c_0$ differ by a constant on $\tilde{\Omega}_\o$.  Let $\hat{w}^\c_0$ be the element of the equivalence class $\bar{\bm{w}}^\c_0$ which is equal to $\bar{w}^\a_0$ on $\tilde{\Omega}_\o$.  We can then define a function
\[
\bar{w}_0 = \begin{cases} &\bar{w}^\a_0 \quad \mbox{on} \quad \tilde{\Omega}_\a \\
&\hat{w}^\c_0 \quad \mbox{on} \quad \tilde{\Omega}_\c
\end{cases},
\]
for which $\bar{w}_0 \in L^2_{\rm loc}(\mathbb{R}^d)$ and $\nabla \bar{w}_0 \in L^2(\mathbb{R}^d)$.  Consequently, $\bar{w}_0$ is a global solution to the linear homogeneous Cauchy-Born equation,
\[
\int_{\mathbb{R}^d} \bbC:\nabla \bar{w}_0: \nabla v = 0, \quad \forall v \in H^1_0(\mathbb{R}^d),
\]
so that $\nabla \bar{w}_0 = 0$. We conclude that $\left(\nabla \bar{w}^\a_0,\nabla \bar{w}^\c_0\right)_{L^2\left(\tilde{\Omega}_{\o}\right)} = 0$, which contradicts~\eqref{vergeResult2}.
\end{proof}

\section{Conclusion}

We have presented an \textit{a priori} error analysis of the optimization-based AtC method proposed in~\cite{olsonPro2013} for the case of a point defect in an infinite lattice in two and three dimensions.  This method is an extension of the virtual control technique for coupling PDEs~\cite{gervasio_2001,Lions_00_JAM,lions1998} and couples a nonlocal, potential-based atomistic model with a continuum finite element model by minimizing the $H^1$ (semi-)norm of solutions to restricted atomistic and continuum subproblems.  Our analysis shows a solution to the AtC method exists provided the atomistic solution is strongly stable and estimates an error between the true solution and AtC solution.  The key result in this analysis was a norm equivalence theorem proven in Section~\ref{sec:normEquiv}.

\appendix

\section{Extension Theorems}\label{extensionAppendix}
In this appendix, we recall Stein's extension theorem~\cite{stein1970} for domains with minimally smooth boundary and a modified extension operator that preserves the $H^1$ seminorm due to Burenkov~\cite{burenkov1985}.

\begin{theorem}[Stein's Extension Theorem]\label{stein}
Let $U$ be a connected, open set for which there exists $\epsilon > 0$, integers $N,M > 0$, and a sequence of open sets $U_1, U_2, \ldots$ satisfying
\begin{enumerate}
\item For each $x \in \partial U$, $B_\epsilon(x) \subset U_i$ for some $i$,
\item The intersection of more than $N$ of the sets $U_i$ is empty,
\item For each $U_i$, there exists a Lipschitz continuous function $\varphi_i$ and domains
\[
D_i = \left\{(x',y) \in \mathbb{R}^{n+1} : y > \varphi_i(x'), \left|\varphi_i(x'_1) - \varphi_i(x'_2)\right| \leq M\left|x'_1 - x'_2\right|\right\}
\]
such that
\[
U_i \cap U = U_i \cap D_i.
\]
\end{enumerate}
Then there exists a bounded linear extension operator $E:H^1(U) \to H^1(\mathbb{R}^d)$.  The bound of the extension depends upon the domain $U$ through $N,M$, and $\epsilon$.
\end{theorem}

Theorem~\ref{stein} can be used to prove an extension theorem with preservation of seminorm due to Burenkov~\cite{burenkov1985}:

\begin{theorem}[Extension with preservation of seminorm]\label{buren}
Let $U$ be a connected, bounded open set for which there exists a bounded linear extension operator $E:H^1(U) \to H^1\left(\mathbb{R}^n\right)$ and a bounded projection operator $P$ from $H^1(U)$ onto the constants with the property that for all $f \in H^1(U)$,
\[
\|f - Pf\|_{L^2(U)} \lesssim c(U) \|f\|_{H^1(U)}.
\]
Then the operator defined by
\[
R = P + E({\rm id} - P)
\]
is a linear extension operator with the property that
\[
\|\nabla R f\|_{L^2(U)} \leq \|E\|\left(c(U) + 1\right)\|\nabla f\|_{L^2(U)}.
\]
\end{theorem}

\begin{remark}
We can set $E$ to be Stein's extension operator and choose
\[
Pu = \frac{1}{m(U)}\int_{U} u(x)\, dx.
\]
In this case, $c(U)$ is the Poincare constant for the domain $U$.
\end{remark}

\section{Notation}\label{app:notation}
For the convenience of the readers, we summarize the key notation used throughout the paper.

\begin{itemize}
\item $\xi$ --- an element of $\mathbb{Z}^d$ or $\epsilon\mathbb{Z}^d$ for $\epsilon > 0$.

\item $|\cdot|$ --- meaning depends on context:  $|\cdot|$ is $\ell^2$ norm of a vector, matrix, or higher order tensor, $|T|$ is area or volume of element $T$ in a finite element partition, $|\alpha|$ is the order of a multiindex.

\item $\|\cdot\|_{\ell^2(A)}$ --- $\ell^2$ norm over a set $A$.  If $f:A \to \mathbb{R}^d$ is a vector-valued function, $\|f\|_{\ell^2(A)} = (\sum_{\alpha \in A}|f(\alpha)|^2)^{1/2}$.

\item $B_{r}(\bm{y}) = \{ \bm{x}\in \mathbb{R}^d : |\bm{y} - \bm{x}| \le r\}$ -
 Ball of radius $r$ in $\mathbb{R}^d$

\item $\bar{U}$ --- closure of a domain $U$.

\item $U^\circ$ --- interior of a domain $U$.

\item ${\rm supp}(f)$ --- support of a function $f$.

\item ${\rm Diam}(U)$ --- diameter of the set $U$ measured with the Euclidean norm.

\item ${\rm dist}(U,V)$ --- distance between the sets $U$ and $V$ measured with the Euclidean norm.

\item ${\rm conv}(x,y)$ --- convex hull of $x$ and $y$.

\item $(\mathbb{R}^d)^{\calR}$ --- direct product of vectors with $|\calR|$ terms.

\item $\mG$ --- a $d \times d$ matrix.

\item $e_i$ --- $i$th standard basis vector in $\mathbb{R}^d$.

\item $\mathstrut^\transpose$ --- transpose of a matrix.

\item $\otimes$ --- tensor product.

\item $\nabla^j$ --- $j$th Fr\'echet derivative of a function defined on $\mathbb{R}^d$.

\item $\partial^\alpha$ --- multiindex notation for derivatives.

\item $L^p(U)$ --- Standard Lebesgue spaces.

\item $(\cdot, \cdot)_{L^2(U)}$ --- $L^2$ inner product over $U$.

\item $W^{k,p}(U)$ --- Standard Sobolev spaces.

\item $W^{k,p}_{\rm loc}(U) = \left\{f:U \to \mathbb{R}^d | f \in W^{k,p}(V) \, \forall V \subsubset U  \right\}$.

\item $H^k(U) = W^{k,2}(U)$, $H^1_0(U) = \left\{f \in H^k(U) : \mbox{Trace}(f) = 0~\, \mbox{on} \,~ \partial U \right\}$.

\item ${\rm C}^{k,1}(\bar{U}) = \big\{f:U \to \mathbb{R}^d : \sum_{|\alpha| \leq k} \sup_{x \in \bar{U}}|\partial^\alpha f(x)| + \sum_{|\alpha| = k} \sup_{\substack{x,y \in \bar{U} \\ x \neq y}} \frac{|\partial^\alpha f(x) - \partial^\alpha f(y)|}{|x-y|}\big\}$. (Standard Lipschitz spaces).

\item $*$ --- used to denote convolution of functions

\item $\dashint_{U} f\, dx$ --- average value of $f$ over $U$.

\item $\calT$ --- a finite element discretization of triangles in $2D$ or tetrahedra in $3D$.

\item $\mathcal{P}^1(T)$ --- set of affine functions over a triangle or tetrahedron, $T$.
	
\item $\mathcal{P}^1(\calT)$ --- set of piecewise affine functions with respect to the discretization $\calT$.

\end{itemize}
%
%%-----------------------------
%%      your bibliography
%%-----------------------------
\bibliographystyle{plain}	% (uses file "plain.bst")
\bibliography{myrefs}		% expects file "myrefs.bib"

\end{document}